\documentclass[11pt]{article}

\let\al=\alpha
\let\b=\beta

\let\la=\lambda

\let\om=\omega
\let\G= \Gamma
\let\pa=\partial
\let\Om=\Omega

\def\vec g{\mathbf g}

\def\rh{\rho h}

\def\Xi{\rh W^2}

\renewcommand{\vec}[1]{\mbox{\boldmath \small $#1$}}

\usepackage{graphicx}
\usepackage{subfigure}
\usepackage{amsmath}
\usepackage{amsfonts}
\usepackage{mathrsfs}
\usepackage{amsbsy}
\usepackage{mathrsfs}
\usepackage{latexsym}
\usepackage{multirow,fancybox}
\usepackage{color}
\usepackage{epstopdf} 
\newtheorem{Theorem}{Theorem}[section]

\newtheorem{lemma}[Theorem]{Lemma}

\newtheorem{example}{Example}[section]
\newenvironment{proof}[1][Proof]{\begin{trivlist}
\item[\hskip \labelsep {\bfseries #1}]}{\end{trivlist}}

\newcommand{\qed}{\hfill \nobreak \ifvmode \relax \else
	\ifdim\lastskip<1.5em \hskip-\lastskip
	\hskip1.5em plus0em minus0.5em \fi \nobreak
	\vrule height0.75em width0.5em depth0.25em\fi}

\usepackage{geometry}
\begin{document}

\markboth{Y.H.~Yuan \& H.Z.~Tang}{Two-stage 4th-order time discretizations}

\title{Two-stage fourth-order  accurate time discretizations for 1D and 2D special relativistic hydrodynamics}

\author{Yuhuan Yuan and 
Huazhong Tang   
\\  HEDPS, CAPT \& LMAM, School of Mathematical Sciences, \\
 Peking University,
Beijing 100871, 
 P.R. China }

\maketitle

\begin{abstract}
	This paper studies the two-stage fourth-order accurate
	time { discretization [J.Q. Li and Z.F. Du, {\em SIAM J. Sci. Comput.}, 38(2016)] and}   its application
	to the special relativistic hydrodynamical equations.
	Our analysis reveals that { the} new two-stage fourth-order accurate time discretizations
can be proposed.
	With the aid of the direct Eulerian GRP (generalized Riemann
	problem) methods
	and  the analytical resolution of  the local ``quasi 1D'' GRP,
	the two-stage fourth-order accurate time discretizations are successfully  implemented
	for the 1D and 2D special relativistic hydrodynamical equations. Several  numerical
	experiments demonstrate the performance and accuracy as well as robustness of our
	schemes.
\end{abstract}


{\bf Keywords}:
	Time discretization, shock-capturing scheme, GRP method, relativistic hydro-dynamics, hyperbolic conservation laws.

\section{Introduction}
\label{sec:into}
The relativistic hydrodynamics (RHD)
plays the leading role  in astrophysics and nuclear physics etc. The RHDs
is necessary in situations where the local velocity of the
flow is close to the light speed in vacuum, or where the local internal energy density is comparable (or larger) than the local rest mass density of  fluid.
The paper is concerned with  developing higher-order accurate numerical schemes
for {the} 1D and 2D special RHDs.
The $d$-dimensional governing equations of the special RHDs is a first-order quasilinear
hyperbolic system.  In the laboratory frame, it can be written
{into} the divergence form
\begin{align}\label{eq:rhd-eq001}
\frac{\partial \vec U}{\partial\vec  t}
+\sum_{i=1}^d \frac{\partial \vec F_i(\vec U)}{\partial  x_i}=0,
\end{align}
where $d=1$ or 2,  and $\vec U$ and $\vec F_i(\vec U)$  denote the conservative vector and the
flux in the $x_i$-direction, respectively, defined by
\begin{align}
\vec U=(D,\vec m, E)^T, \
\vec F_i(\vec U)=(D v_i, v_i \vec m+ p \vec e_i, m_i)^T, i=1,\cdots,d,
\end{align}
with the mass density $D = \rho W$, the momentum density (row) vector $\vec m = D hW\vec v$, the energy density $E=D h W - p$, and
the row vector $\vec e_i$ denoting the $i$-th row of the unit matrix of size 2.
Here $\rho$ is the rest-mass density, { $v_i$} denotes
the fluid velocity in the $x_i$-direction,  $p$ is  the gas pressure,
$W=1/\sqrt{1- v^2}$ is the Lorentz factor with $v:=\left(v_1^2+\cdots+v_d^2\right)^{1/2}$,
$h$ is the {specific enthalpy} defined by
$$h = 1 + e + \frac{p}{\rho},$$
with units in which the speed of light $c$ is equal to one, and $e$ is the specific internal energy.
Throughout the paper, the equation of state (EOS) will be restricted to the $\Gamma$-law
\begin{equation}\label{eq:EOS}
p = (\Gamma-1) \rho e,
\end{equation}
where the adiabatic index $\Gamma \in (1,2]$. The restriction of $\Gamma\le 2$ is required for the compressibility assumptions   and the causality in the theory of relativity (the sound speed does not exceed the speed of light {$c=1$}).

The RHD equations \eqref{eq:rhd-eq001} are highly nonlinear so that their analytical treatment is
extremely difficult. Numerical computation has become a major way in studying RHDs.
The pioneering numerical work {can} date back to
the Lagrangian finite difference code via artificial viscosity for the
spherically symmetric general RHD equations
\cite{May-White1966,May-White1967}.
Multi-dimensional RHD equations  were first solved in \cite{Wilson:1972} by using the
Eulerian   finite difference method
with the artificial viscosity technique.
Later, modern shock-capturing methods were extended to the RHD (including RMHD) equations.
{Some representative  methods} are the HLL (Harten-Lax-van Leer) scheme \cite{Zanna:2003}, HLLC (HLLC contact) schemes \cite{MignoneHLLCRMHD,Honkkila:2007}, Riemann solver \cite{Balsara1994},  approximate Riemann solvers based on the local linearization \cite{GodunovRMHD,Koldoba:2002},
second-order
GRP (generalized Riemann problem) schemes \cite{Yang-He-Tang:2011,Yang-Tang:2012,Wu-Tang:2016}, third-order GRP scheme
\cite{Wu-Yang-Tang:2014b},  locally evolution Galerkin method \cite{WuEGRHD},
discontinuous Galerkin (DG) methods \cite{Zhao-Tang:2017a,Zhao-Tang:2017b},
gas-kinetic schemes {(GKS)} \cite{QamarKinetic2004,Chen-Kuang-Tang:2017},  adaptive mesh refinement methods \cite{Anderson:2006,Host:2008},
and    moving mesh methods
\cite{HeAdaptiveRHD,HeAdaptiveRMHD} etc.
Recently, the higher-order accurate physical-constraints-preserving (PCP) WENO (weighted essentially non-oscillatory) and DG
schemes were developed for the special RHD equations \cite{Wu-Tang-JCP2015,Wu-Tang-RHD2016,QinShuYang-JCP2016}.
They were built on studying the  admissible state set of {the} special RHDs.
The  admissible state set  and   PCP schemes of
{the} special ideal RMHDs were also studied
for the first time in \cite{Wu-Tang-RMHD2016}, where  the importance of divergence-free fields
was revealed in achieving PCP methods.  Those works were  successfully extended to
the special RHDs with a general equation of state \cite{Wu-Tang-RMHD2017b,Wu-Tang-RHD2016}
and the general RHDs \cite{Wu-PRD2017}.

In comparison with second-order shock-capturing schemes,
the higher-order methods can provide more accurate solutions, but they are less robust and more complicated. For most of the existing higher-order methods, the Runge-Kutta time discretization
is usually used to achieve higher order temporal accuracy. For example, a four-stage fourth-order  Runge-Kutta   method (see e.g. \cite{Zhao-Tang:2017b}) is  used to achieve a fourth-order time accuracy.
{If each stage of the
time discretization needs to call the Riemann solver or the resolution of the local GRP, then the shock-capturing scheme  with multi-stage time discretization for \eqref{eq:rhd-eq001} is very time-consuming.}
Recently, based on the time-dependent flux function of the GRP, a two-stage fourth-order  accurate  time discretization was developed for Lax-Wendroff (LW) type  flow solvers, particularly applied for the hyperbolic conservation laws \cite{LI-DU:2016}. Such two-stage LW time stepping method does also provide an alternative framework for the development of a fourth-order GKS with a second-order flux function \cite{PXLL:2016}.
%

The aim of this paper  is to  study the two-stage fourth-order accurate time discretization {\cite{LI-DU:2016}} and
its application to the special {RHD} equations {\eqref{eq:rhd-eq001}}. {Based our analysis, the new}
two-stage fourth-order accurate time discretizations {can}  be proposed. With the aid of
the direct Eulerian 
GRP  methods and  the analytical resolution of  the local ``quasi 1D'' GRP, those two-stage
fourth-order accurate time discretizations can be conveniently implemented for the
special {RHD} equations. {Their performance and accuracy as well as robustness can be demonstrated by numerical experiments}.
The paper
is organized as follows. 
Section \ref{sec:scheme}  studies   the two-stage fourth-order accurate time discretizations
and applies them to  the special RHD equations.
Section \ref{sec:example} conducts several numerical experiments
to demonstrate the accuracy and efficiency of the proposed methods.
Conclusions are given in Section \ref{Section-conclusion}.

%

\section{Numerical methods}
\label{sec:scheme}
{In this section, we study the two-stage fourth-order accurate time discretization   \cite{LI-DU:2016} and then propose the new
	two-stage fourth-order accurate time discretizations.}
With the aid of
the direct Eulerian GRP  methods, those two-stage
time discretizations {can be}  implemented for the
special RHD equations \eqref{eq:rhd-eq001}.

\subsection{Two-stage fourth-order  time discretizations}\label{subsec:timedis}
Consider the time-dependent differential equation
\begin{align}\label{eq:ode01}
u_t=L(u),\ t>0,
\end{align}
which {can} be a semi-discrete scheme for the conservation laws \eqref{eq:rhd-eq001}.

Assume that the solution $u$ of \eqref{eq:ode01} is a fully smooth function of $t$
and $L(u)$ is also fully smooth, and
give a partition of the time interval by $t_n=n\tau$,
where  $\tau$ denotes the time stepsize and $n\in \mathbb Z$.
The Taylor series expansion of $u$ in $t$ reads
\begin{align}\nonumber
u^{n+1}=&\Big(u +\tau u_t +\frac{\tau^2}{2!} u_{tt}
+   \frac{\tau^3}{3!} u_{ttt}+\frac{\tau^4}{4!} u_{tttt}\Big )^n+{\mathcal O}(\tau^5)
\\ \label{eq:add-decomp}
=&\Big(u +\tau u_t +\frac{\tau^2}{6} u_{tt}\Big)^n
+{  2} \frac{\tau^2}{6} \Big( \big(  u
+ \frac{\tau}{2}  u_{t}+\frac{\tau^2}{8} u_{tt}\big)_{tt} \Big )^n+{\mathcal O}(\tau^5).
\end{align}
Substituting \eqref{eq:ode01} {into} \eqref{eq:add-decomp} gives
\begin{align}\nonumber
u^{n+1}=&\Big(u +\tau L(u) +\frac{\tau^2}{6} \partial_t L(u)\Big)^n
+ {  2} \frac{\tau^2}{6} \Big( \big(  u
+ \frac{\tau}{2}  L(u)+\frac{\tau^2}{8} \partial_t L(u)\big)_{tt} \Big )^n+{\mathcal O}(\tau^5)
\\  \label{eq:add-decomp02}
=&\Big(u +\tau L(u) +\frac{\tau^2}{6} \partial_t L(u)\Big)^n
+ {  2} \frac{\tau^2}{6} \left( (u^*)_{tt} \right)^n+{\mathcal O}(\tau^5),
\end{align}
where
\begin{align*}
u^*:= u + \frac{1}{2} \tau L(u) +  \frac{1}{8} \tau^2  L_t(u) =u\left(t_n+\frac{\tau}2\right)+{\mathcal O}(\tau^3).
\end{align*}
Because
\begin{align*}
u^*_t=& u_t + \frac{1}{2} \tau L_t(u) +  \frac{1}{8} \tau^2   L_{tt}(u)
=L(u)+\frac{1}{2} \tau L_u L(u)+\frac{1}{8} \tau^2
\left( (L_u)^2L+ L_{uu} L^2  \right),
\\
L(u^*)
=&L(u)+L_u \left(\frac{1}{2} \tau L(u) +  \frac{1}{8} \tau^2   L_t(u)\right)
+\frac12 L_{uu} \left(\frac{1}{2} \tau L(u) +  \frac{1}{8} \tau^2   L_t(u)\right)^2 {+\cdots},
\end{align*}
one has
\begin{align*}
u^*_t=L(u^*)+{\mathcal O}(\tau^3), \ u^*_{tt}=L_t(u^*)+{\mathcal O}(\tau^3).
\end{align*}
Combining the second equation  with \eqref{eq:add-decomp02} gives
\begin{align*}
u^{n+1}
=&\Big(u +\tau L(u) +\frac{\tau^2}{6} \partial_t L(u)\Big)^n
+ { 2} \frac{\tau^2}{6} \left( L_t(u^*) \right)^n+{\mathcal O}(\tau^5).
\end{align*}

The above discussion gives  the two-stage fourth-order accurate time discretization of \eqref{eq:ode01}  \cite{LI-DU:2016}:
\begin{itemize}
	\item[Step 1.] Compute the intermediate value
	\begin{equation} \label{eq:1stage}
	u^* = u^n + \frac{1}{2} \tau L(u^n) +  \frac{1}{8} \tau^2 \frac{\pa}{\pa t} L(u^n), 
	\end{equation}
	\item[Step 2.] Evolve the solution at $t_{n+1}$ via
	\begin{equation}
	u^{n+1} = u^n + \tau L(u^n)
	+ \frac{1}{6} \tau^2\left( \frac{\pa}{\pa t} L(u^n) + { 2} \frac{\pa}{\pa t} L(u^*)  \right).
	\end{equation}
\end{itemize}

Obviously, the   additive decomposition in \eqref{eq:add-decomp} is not unique. For example, it can be  replaced {with} a more general decomposition
\begin{align}
u^{n+1}
=&\Big(u +\tau L(u) +\frac{{\alpha}\tau^2}{2} \partial_t L(u)\Big)^n
+ { } \frac{(1-\alpha)\tau^2}{2} \Big( \big(  \tilde{u}^*\big)_{tt} \Big )^n+ {  O(\tau^5) },
\label{EQ:decom03}\end{align}
{with} $\alpha\neq 1$
and
\begin{align}\label{EQ:decom03a}
\tilde{u}^*:= u+ \frac{\tau}{3(1-\alpha)}  L(u)+\frac{\tau^2}{12(1-\alpha)} \partial_t L(u).
\end{align}
{If} $\alpha=\frac13$, then {\eqref{EQ:decom03} becomes the  additive  decomposition in \eqref{eq:add-decomp}} for  the two-stage fourth-order time discretization \cite{LI-DU:2016}.

{The identity \eqref{EQ:decom03a}} implies
\begin{align}\label{EQ:decom03b}
\tilde{u}^*_t=& L(u)+ \frac{\tau}{3(1-\alpha)}  L_t(u)+\frac{\tau^2}{12(1-\alpha)} \partial_{tt} L(u).
\end{align}
{Comparing \eqref{EQ:decom03b} to the following Taylor series expansion}
\begin{align*}
L(\tilde{u}^*)=& L(u)+ L_u\left( \frac{\tau}{3(1-\alpha)} L  +\frac{\tau^2}{12(1-\alpha)} L_t(u)\right)+\frac12 L_{uu} \left( \frac{\tau}{3(1-\alpha)} L  +\frac{\tau^2}{12(1-\alpha)} L_t(u)\right)^2 {+\cdots},
\end{align*}
{ gives
	\begin{align*}
	\tilde{u}^*_t=L(\tilde{u}^*)+\frac{\tau^2}{12(1-\alpha)} \left(1-\frac{2}{3(1-\alpha)}\right)
	L_{uu} (u_t)^2+{\mathcal O}(\tau^3).
	\end{align*}
	If
	\begin{align}\label{EQ:decom03z}
	1-\frac{2}{3(1-\alpha)}=C \tau^p, \ p\geq 1,
	\end{align}
	where  $C$ is independent on $\tau$,
	then
	\begin{align*}
	\tilde{u}^*_t=L(\tilde{u}^*)+{\mathcal O}(\tau^3).
	\end{align*}
{Therefore}, if $\alpha=\alpha(\hat{\tau})$ is a differentiable function of $\hat{\tau}$
	and  satisfies
	$\alpha(0)=1/3$, $\alpha\neq 1$, and $\hat{\tau}=C \tau^p$,
	then \eqref{EQ:decom03z} does hold.
	For example, we may choose $\alpha=(1-6\hat{\tau})/(3-6\hat{\tau})$  with $\hat{\tau}\neq 1/2$.
	At this time, one has
	\begin{align*}
	\tilde{u}^*= u\left(t+\frac{\tau}{3(1-\alpha)}\right)+{\mathcal O}(\tau^3),
	\end{align*}
	and similarly, from \eqref{EQ:decom03a} and the Taylor series expansion
	of $L_t\left(\tilde{u}^*\right)$ at $u$, one can get
	\begin{align}\label{EQ:00178}
	\tilde{u}^*_{tt}=
	L_t\left(\tilde{u}^*\right)+{\mathcal O}(\tau^3).
	\end{align}
}
Substituting \eqref{EQ:00178} into \eqref{EQ:decom03} gives
\begin{align}\nonumber
u^{n+1}
=&\Big(u +\tau L(u) +\frac{{\alpha}\tau^2}{2} \partial_t L(u)\Big)^n
+  \frac{{ (1-\alpha)}\tau^2}{2} \Big( \partial_t L(\tilde{u}^*) \Big )^n+{\mathcal O}(\tau^5).
\label{EQ:decom02}\end{align}
{In conclusion,
	when $\alpha=\alpha(\hat{\tau})$ is a differentiable function of $\hat{\tau}=C \tau^p$
	and  satisfies $\alpha(0)=1/3$ and $\alpha\neq 1$,
	where $p\geq 1$ and $C$ is independent on $\tau$,
	the   additive decomposition  \eqref{EQ:decom03} can give new two-stage fourth-order
	accurate time discretizations as follows:
\begin{itemize}
	\item[Step 1.] Compute the intermediate value
	\begin{equation} \label{eq:1stage02}
	u^* = u^n + \frac{1}{3(1-\al)} \tau L(u^n) +  \frac{\tau^2 }{12(1-\al)} \frac{\pa}{\pa t} L(u^n), 
	\end{equation}
	\item[Step 2.] Evolve the solution at $t_{n+1}$ via
	\begin{equation}
	u^{n+1} = u^n + \tau L(u^n)
	+ \frac{\al\tau^2}{2} \frac{\pa}{\pa t} L(u^n) + \frac{(1-\al)\tau^2}{2} \frac{\pa}{\pa t} L(u^*) .
	\end{equation}
\end{itemize}
}

\subsection{Application of two-stage time discretizations to 1D RHD equations}
\label{subsec:method-1d}
{In this section, we apply the above two-stage fourth-order time discretizations to the 1D RHD equations, i.e. \eqref{eq:rhd-eq001} with $d=1$.}
For the sake of convenience, the symbols $x_1$ and $v_1$ {are} replaced with $x$ and $u$, respectively, and  a uniform partition of the spatial domain
is given  by  $I_j = (x_{j-\frac{1}{2}} ,\, x_{j+\frac{1}{2}})$ with $\Delta x=x_{j+\frac{1}{2}}-x_{j-\frac{1}{2}}$.

For the given ``initial'' approximate cell averages $\{\overline{\vec U}^n_j\}$ at $t_n$,
{we want to reconstruct the WENO  values of $\vec U$ and $\partial_x\vec U$ at the cell boundaries,
	denoted by $\vec U^{\pm,n}_{j+\frac12}$
	and $(\partial_x \vec U)^{\pm,n}_{j+\frac12}$.
	Our initial reconstruction procedure is given as follows:
	\begin{description}
		\item[(1)] Use the standard 5th-order WENO reconstruction \cite{Jiang-Shu1996} to get
$\vec U^{\pm,n}_{j+\frac12}$ with the aid of the characteristic decomposition. If $\vec U^{\pm,n}_{j+\frac12}$ does not belong to the admissible state set of 1D RHD equations \cite{Wu-Tang-JCP2015}, then we set $\vec U^{\pm,n}_{j+\frac12} = \overline{\vec U}_j^n$ directly  in order to avoid the nonphysical solution as soon as possible.
		
		\item[(2)] Calculate $(\overline{\partial_x \vec U})_j^n = \frac{1}{\Delta x}( \vec U^{-,n}_{j+\frac12} - \vec U^{+,n}_{j-\frac12} )$, which is the approximate cell average value of $\partial_x\vec U$
over the cell $I_j$, and then use the above WENO again to get $(\partial_x \vec U)^{\pm,n}_{j+\frac12}$.  
		
		
	\end{description}
}
{Such initial} reconstruction is also used {at} $t_*=t_n+\tau/(3-3\alpha)$,
where {$\alpha=\alpha(\hat{\tau})$
is a differentiable function of $\hat{\tau}$
	and  satisfies
	$\alpha(0)=1/3$, $\alpha\neq 1$, and $\hat{\tau}=C \tau^p$ with   $p\geq 1$ and $C$ independent on $\tau$. } 

The two-stage fourth-order time discretizations in Section \ref{subsec:timedis}
can be applied to the 1D RHD equations by the following steps.

\begin{itemize}
	\item[Step 1.]  { For the reconstructed data $\vec U^{\pm,n}_{j+\frac12}$
	and $(\partial_x \vec U)^{\pm,n}_{j+\frac12}$, 
	following the GRP method \cite{Yang-He-Tang:2011},
	solve the Riemann problem of
	\begin{equation} \label{eq:GRP1D}
	\vec U_t + \vec F_1(\vec U)_x = 0,\ t>t_n,\\
	\end{equation}
	to get $\vec U_{j+\frac12}^{RP,n}$  and then
resolve  analytically the  GRP of \eqref{eq:GRP1D}
to obtain the value of $(\partial \vec U/\partial t)_{j+\frac12}^n$.}
	
	\item[Step 2.]  Compute the intermediate values $\{\overline{\vec U}_j^*\}$ {by}
	\begin{equation}
	\overline{\vec U}_j^* = \overline{\vec U}^n_{j} + \frac{\tau}{3(1-\alpha)} L_j (\overline{\vec U}^n) +  \frac{\tau^2}{12(1-\alpha)}  \pa_t L_j(\overline{\vec U}^n),
	\end{equation}
	where the terms
	$L_j (\overline{\vec U}^n)$ and  $\pa_t L_j(\overline{\vec U}^n)$ are given by
	\begin{equation}
	L_j(\overline{\vec U}^n) = -\frac{1}{\Delta x} \left( \vec F_1(\vec U_{j+\frac{1}{2}}^{RP,n}) - \vec F_1(\vec U_{j-\frac{1}{2}}^{RP,n})\right),
	\end{equation}
	and
	\begin{equation}\label{eq:1dLt}
	\pa_t L_j(\overline{\vec U}^n) = -\frac{1}{\Delta x} \left( (\pa_t \vec F_1)_{j+\frac{1}{2}}^n  - (\pa_t \vec F_1)_{j-\frac{1}{2}}^n\right),
	\end{equation}
	with
	$$(\pa_t \vec F_1)_{j\pm \frac{1}{2}}^n=
	\frac{\partial \vec F_1}{\partial \vec U} \left( \vec U_{j\pm \frac{1}{2}}^{RP,n} \right)
	\cdot   \left(\frac{\partial \vec U}{\partial t}\right)_{j\pm \frac12}^n.
	$$
	
	\item[Step 3.]  Reconstruct the values {$\vec U^{\pm,*}_{j+\frac12}$
		and $(\partial_x \vec U)^{\pm,*}_{j+\frac12}$} with $\{\overline{\vec U}_j^*\}$ by the above initial reconstruction procedure, and then
	resolve  analytically the local  GRP {of \eqref{eq:GRP1D}}
	to get $\vec U_{j+\frac12}^{RP,*}$ and $(\partial \vec U/\partial t)_{j+\frac12}^*$.
	
	\item[Step 4.]  Evolve the solution at $t_{n+1} = t_n + \tau$ by
	\begin{equation}
	\overline{\vec U}_j^{n+1} = \overline{\vec U}_j^n + \tau L_j( \overline{\vec U}^n)
	+ \frac{\alpha\tau^2}{2} \pa_t L_j(\overline{\vec U}^n)
	+ \frac{(1-\alpha)\tau^2}{2} \pa_t L_j(\overline{\vec U}^*),
	\end{equation}
	where
	\begin{equation}
	\pa_t L_j(\overline{\vec U}^*) = -\frac{1}{\Delta x} \left( (\pa_t \vec F_1)_{j+\frac{1}{2}}^*  - (\pa_t \vec F_1)_{j-\frac{1}{2}}^*\right),
	\end{equation}
	with
	$$(\pa_t \vec F_1)_{j\pm \frac{1}{2}}^*=
	\frac{\partial \vec F_1}{\partial \vec U} \left( \vec U_{j\pm \frac{1}{2}}^{RP,*}\right )
	\cdot   \left(\frac{\partial \vec U}{\partial t}\right)_{j\pm \frac12}^*.
	$$
\end{itemize}

%

\subsection{Application of two-stage time discretizations to 2D RHD equations}
{In this section, we apply} the  two-stage fourth-order time discretizations to the 2D RHD equations, i.e. \eqref{eq:rhd-eq001} with $d=2$
with the aid of  the analytical resolution
of the local ``quasi 1D'' GRP and  an adaptive primitive-conservative scheme.
{The latter given in \cite{E.F.Toro:2013}} is used to reduce the spurious solution  generated by the conservative scheme across the contact discontinuity, see Example \ref{example2.1}.
Similarly, the symbols $(x_1, x_2)$ and $(v_1,v_2)$ {are} replaced with $(x,y)$ and $(u,v)$, respectively, and  a uniform partition of the spatial domain
is given  by  $I_{jk} {=} (x_{j-\frac{1}{2}} ,\, x_{j+\frac{1}{2}})\times (y_{k-\frac{1}{2}} ,\, y_{k+\frac{1}{2}})$ with $\Delta x=x_{j+\frac{1}{2}}-x_{j-\frac{1}{2}}$ and
$\Delta y=y_{k+\frac{1}{2}}-y_{k-\frac{1}{2}}$.

\begin{example} \label{example2.1}\rm
	Because of the nonlinearity of \eqref{eq:rhd-eq001},
	when a conservative scheme is used, a spurious solution across the contact
	discontinuity, a well-known phenomenon in multi-fluid systems, {can} arise even for {a single
	material}. It is similar to the phenomenon mentioned in \cite{E.F.Toro:2013}.
	To clarify that, let us solve the Riemann problem of \eqref{eq:rhd-eq001}   
	with the initial data
	\begin{equation} \label{eq:example}
	(\rho ,  u , v, p)(x,y,0)
	\; = \; \begin{cases}
	(0.5, -0.5, 0.5, 5 ) , & x>0.5, \\
	(0.5,  -0.5, -0.5, 5) , & x<0.5.
	\end{cases}
	\end{equation}
	The computational domain is taken as $[0,1]\times [0,1]$.
	Fig. \ref{fig:001} gives the solutions  obtained
	by using the 2D (first-order, conservative) Godunov method.
	Fig. \ref{fig:example001b} gives the solutions  obtained
	by using the 2D two-stage fourth-order conservative method.
	The  obvious oscillations near  the contact
	discontinuity are observed, in other words, the spurious solutions
	have been generated near the contact discontinuity.
	It is  easy to verify it theoretically. To overcome such difficulty,
	the generalized Osher-type scheme in an adaptive primitive-conservative framework \cite{E.F.Toro:2013}
	can be employed to avoid {or reduce} the above spurious solutions at the expense of the conservation.
	{Figs. \ref{fig:example001c} and \ref{fig:example001d} do respectively display more better solutions  obtained
		by the adaptive primitive-conservative scheme with  the reconstructions of the characteristic
		and primitive variables than those in Figs. \ref{fig:001} and \ref{fig:example001b},} in which
	the generalized Osher-type scheme
	is adaptively used to solve  the RHD equations \eqref{eq:rhd-eq001}
	in the equivalently primitive form
	\begin{equation}\label{eq:RHD2d-prim}
	\pa_t \vec V + \widetilde{\vec A}(\vec V) \pa_x \vec V +\widetilde{\vec B}(\vec V) \pa_y \vec V = 0,
	\end{equation}
	where $\vec V = (\rho,\, u,\, v,\, p)^T$ and
	\begin{equation} \label{eq:matrixA}
	\widetilde{\vec A}(\vec V) =
	u \cdot \vec I_4 +
	\begin{pmatrix}
	0 & \dfrac{\rho}{1-(u^2+v^2)c_s^2} & 0 & \dfrac{-u}{W^2h[1-(u^2+v^2)c_s^2]} \\
	0 & \dfrac{-uc_s^2}{W^2[1-(u^2+v^2)c_s^2]} & 0 & \dfrac{H}{\rho h W^2[1-(u^2+v^2)c_s^2]} \\
	0 & \dfrac{-vc_s^2}{W^2[1-(u^2+v^2)c_s^2]} & 0 &  \dfrac{-uv(1-c_s^2)}{\rho h W^2[1-(u^2+v^2)c_s^2]} \\
	0 & \dfrac{\rho h c_s^2}{1-(u^2+v^2)c_s^2} & 0 & \dfrac{-uc_s^2}{W^2[1-(u^2+v^2)c_s^2]}
	\end{pmatrix},
	\end{equation}
	with $H = 1-u^2 - v^2c_s^2$ and $c_s^2 = \frac{\G p} { \rho h}$.
	The matrix $\widetilde{\vec B}(\vec V)$ can be gotten by
	exchanging $u$ and $v$,  and then the second and third row,
	and the second and third column of the matrix $\widetilde{\vec A}$.
	%
	%

	\begin{figure}[htbp]
\centering
		\setlength{\abovecaptionskip}{0.cm}
		\setlength{\belowcaptionskip}{-0.cm}
		\subfigure[$\rho$]{
			\begin{minipage}[t]{0.22\textwidth}
				\centering
				\includegraphics[width=\textwidth]{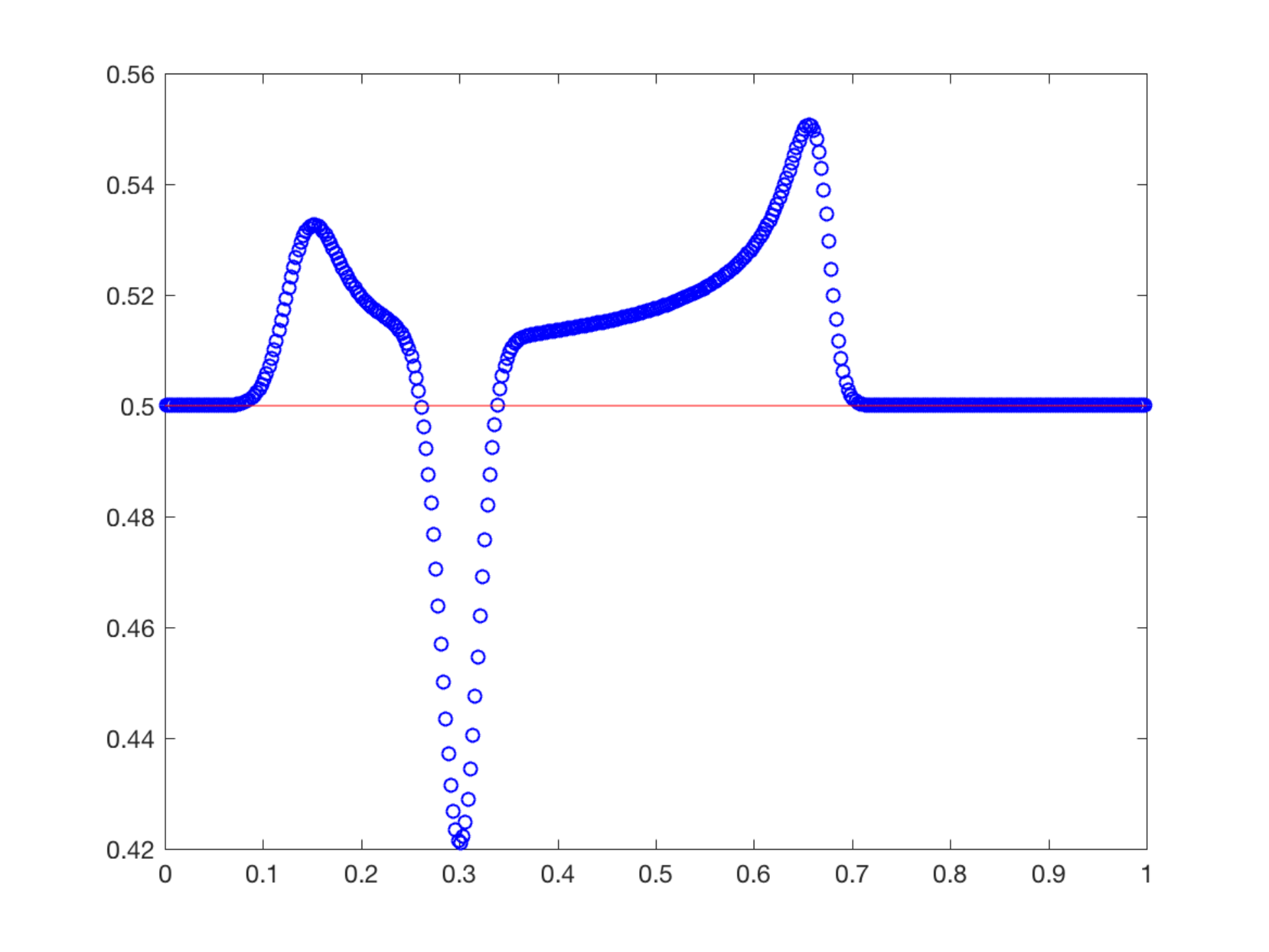}
			\end{minipage}
		}
		\subfigure[$u$]{
			\begin{minipage}[t]{0.22\textwidth}
				\centering
				\includegraphics[width=\textwidth]{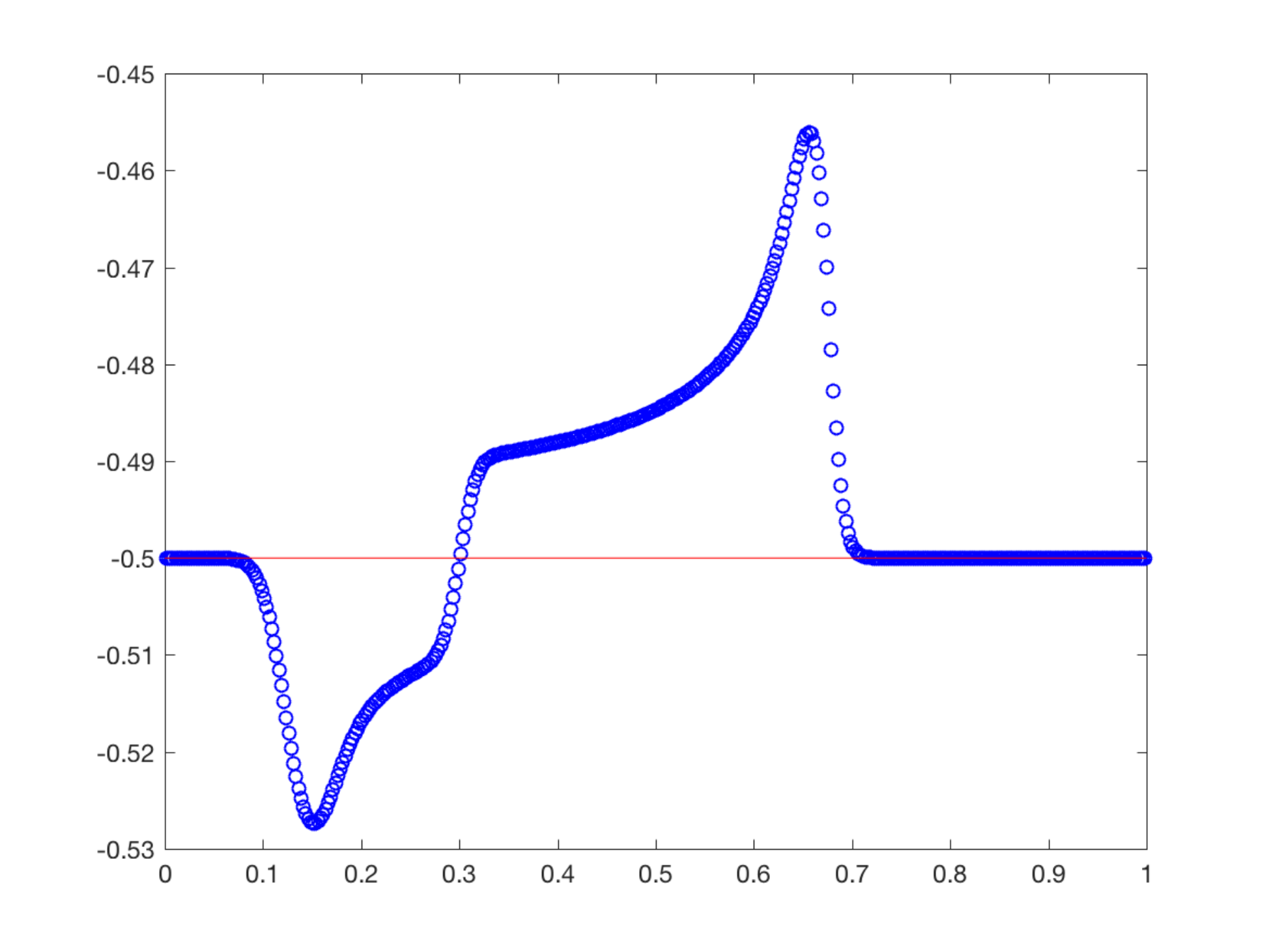}	
			\end{minipage}
		}
		\subfigure[$v$]{
			\begin{minipage}[t]{0.2\textwidth}
				\centering
				\includegraphics[width=\textwidth]{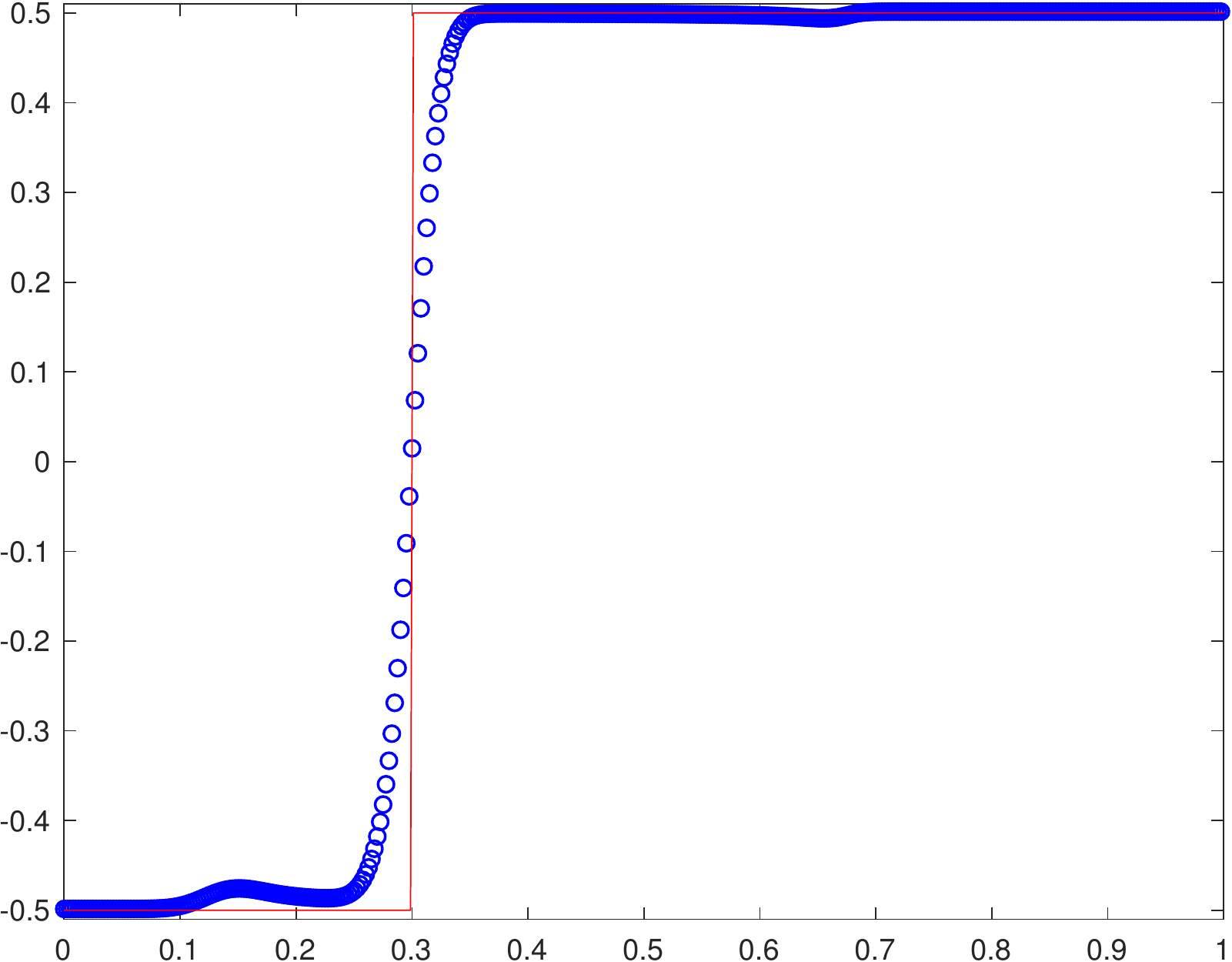}
			\end{minipage}
		}
		\subfigure[$p$]{
			\begin{minipage}[t]{0.22\textwidth}
				\centering
				\includegraphics[width=\textwidth]{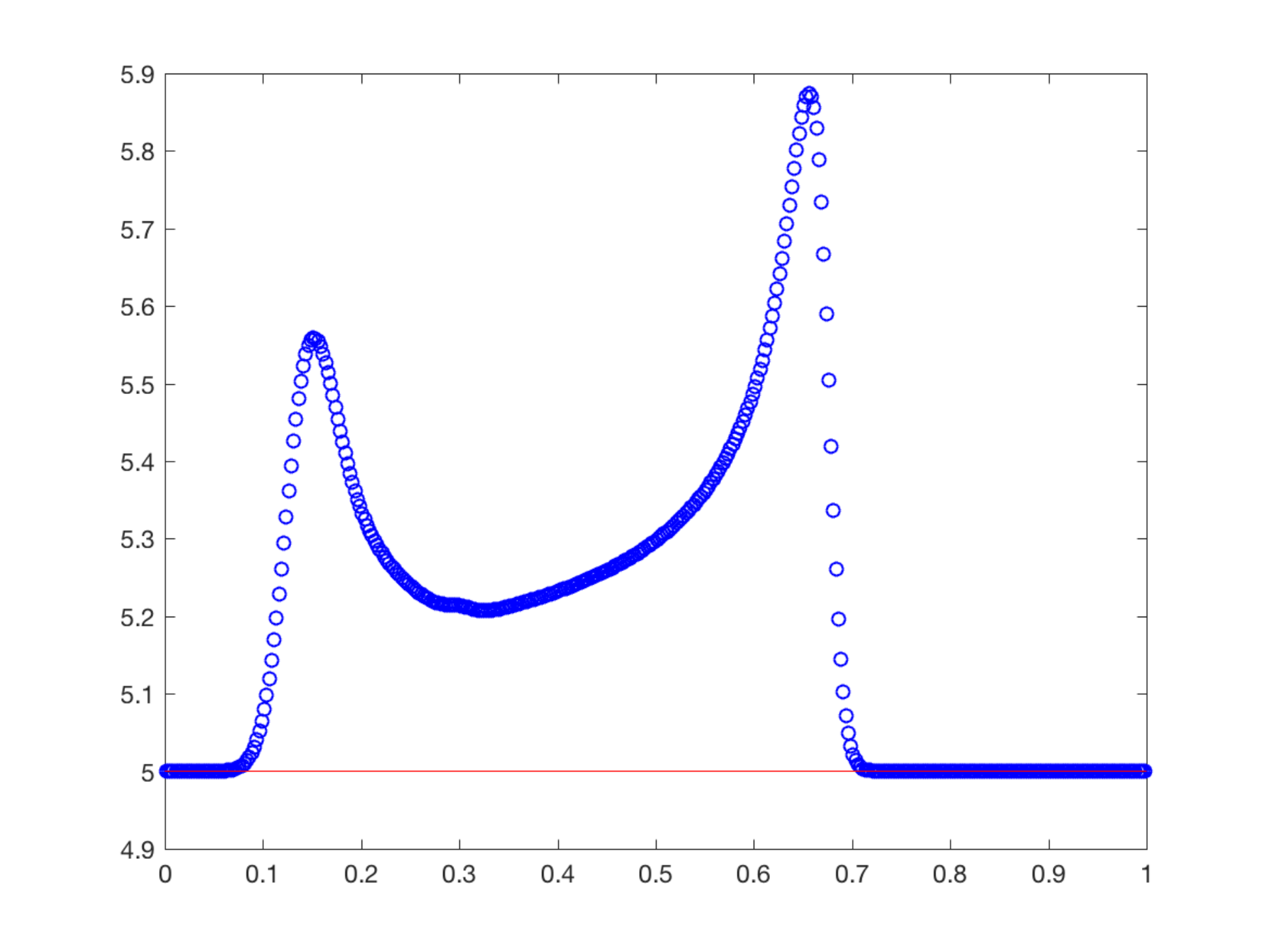}	
			\end{minipage}
		}
		\caption{\small Example \ref{example2.1}: The solutions at $t = 0.4$ obtained by the first-order Godunov method   along the line $y = 0.5$ with $400 \times 40$ uniform cells.}\label{fig:001}
	\end{figure}
	
	\begin{figure}[htbp]
\centering
		\setlength{\abovecaptionskip}{0.cm}
		\setlength{\belowcaptionskip}{-0.cm}
		\subfigure[$\rho$]{
			\begin{minipage}[t]{0.22\textwidth}
				\centering
				\includegraphics[width=\textwidth]{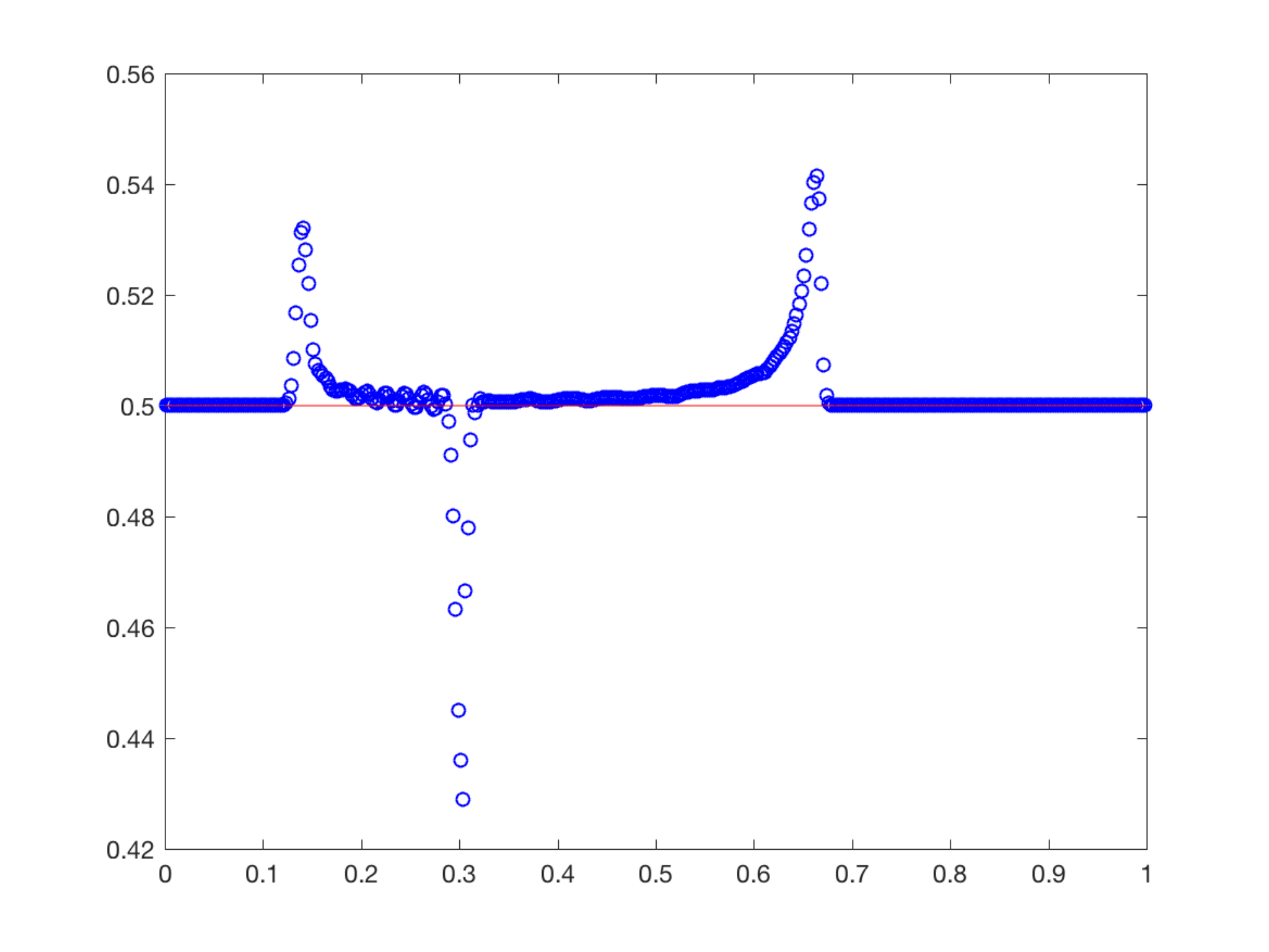}
			\end{minipage}
		}
		\subfigure[$u$]{
			\begin{minipage}[t]{0.22\textwidth}
				\centering
				\includegraphics[width=\textwidth]{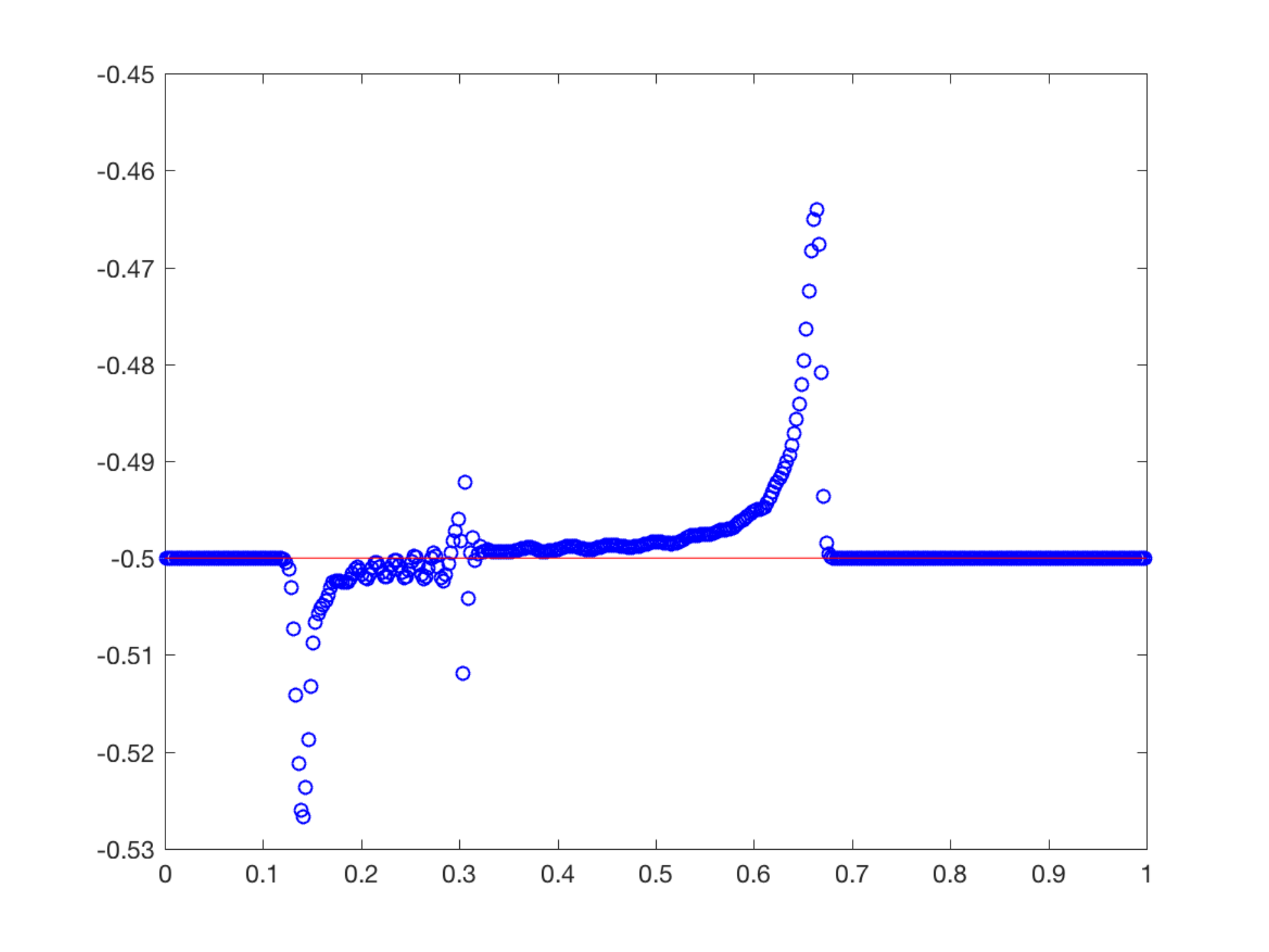}
			\end{minipage}
		}
		\subfigure[$v$]{
			\begin{minipage}[t]{0.2\textwidth}
				\centering
				\includegraphics[width=\textwidth]{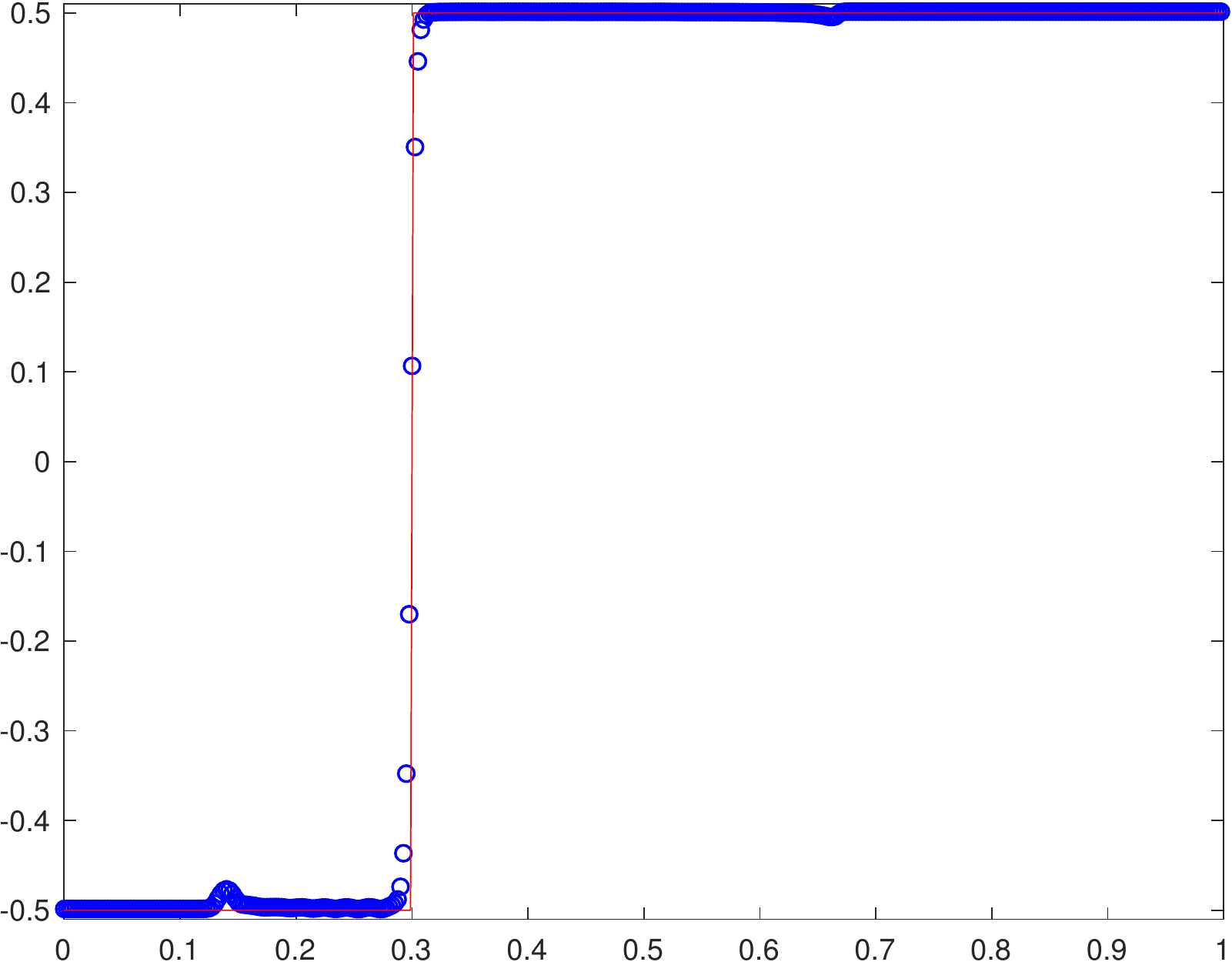}
			\end{minipage}
		}
		\subfigure[$p$]{
			\begin{minipage}[t]{0.22\textwidth}
				\centering
				\includegraphics[width=\textwidth]{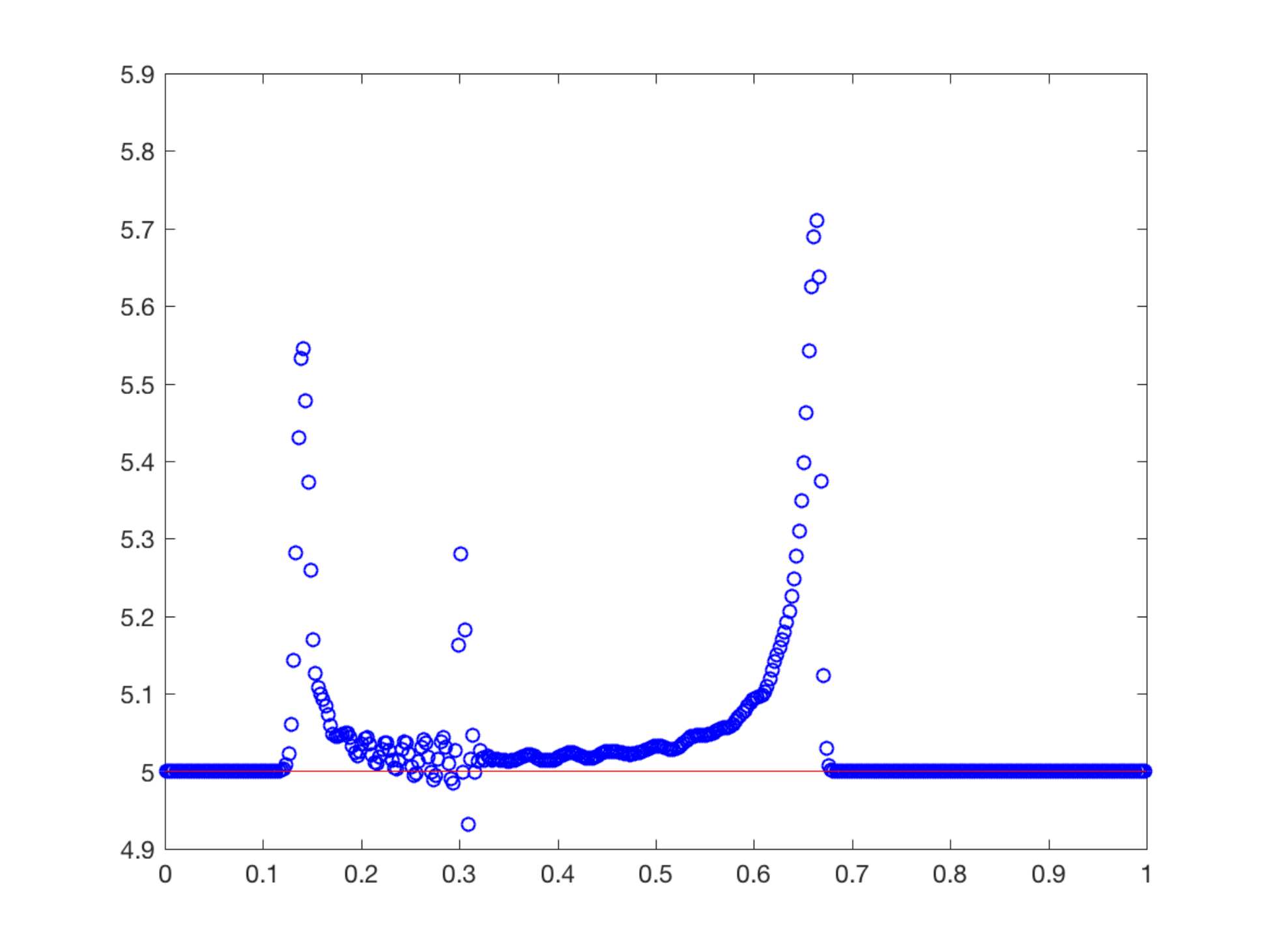}
			\end{minipage}
		}
		\caption{\small Same as Fig. \ref{fig:001} except for two-stage fourth-order conservative scheme {with the   reconstructed characteristic variables}.}\label{fig:example001b}
	\end{figure}
	
	%
	\begin{figure}[htbp]
\centering
		\setlength{\abovecaptionskip}{0.cm}
		\setlength{\belowcaptionskip}{-0.cm}
		\subfigure[$\rho$]{
			\begin{minipage}[t]{0.22\textwidth}
				\centering
				\includegraphics[width=\textwidth]{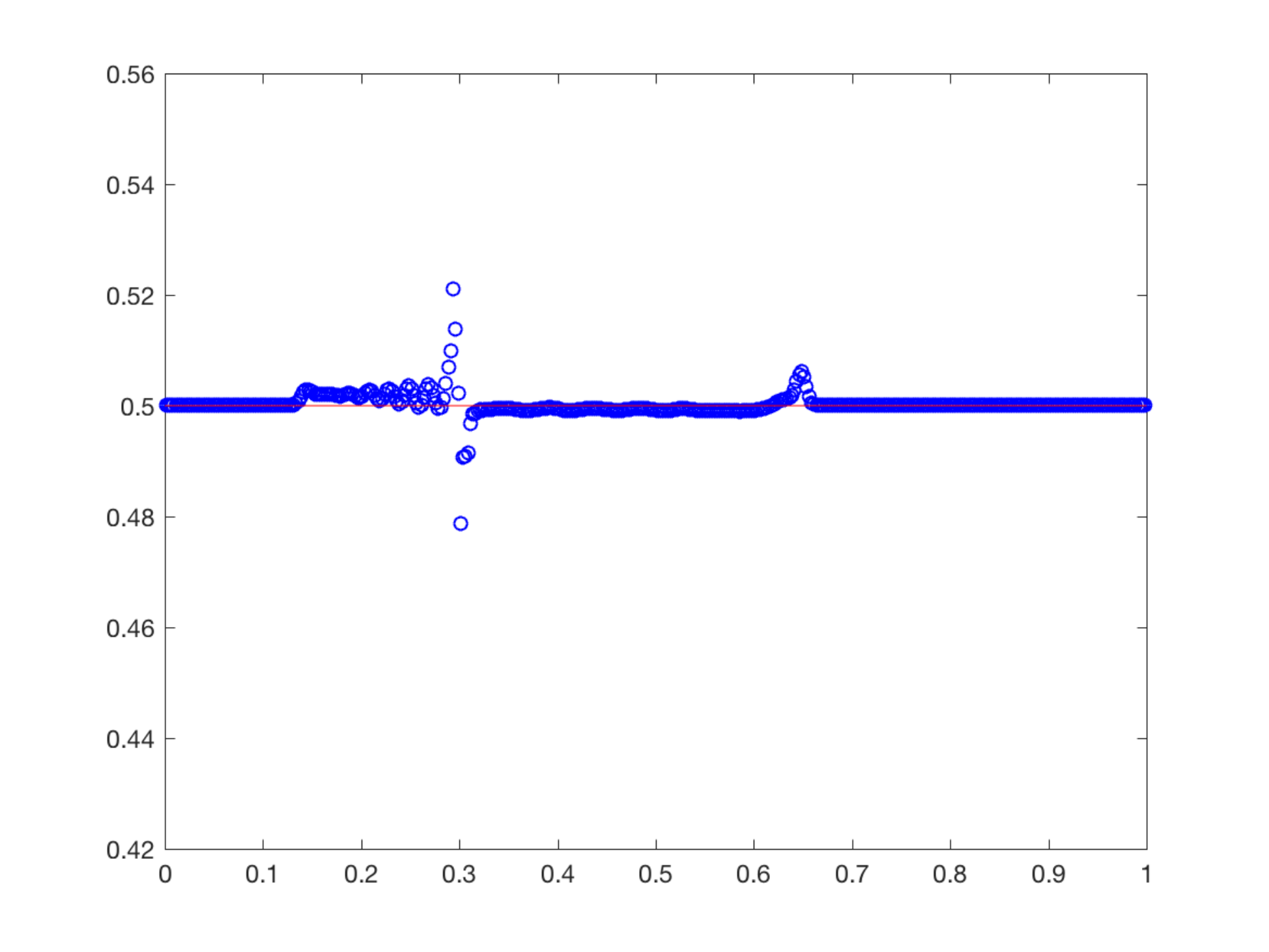}
			\end{minipage}
		}
		\subfigure[$u$]{
			\begin{minipage}[t]{0.22\textwidth}
				\centering
				\includegraphics[width=\textwidth]{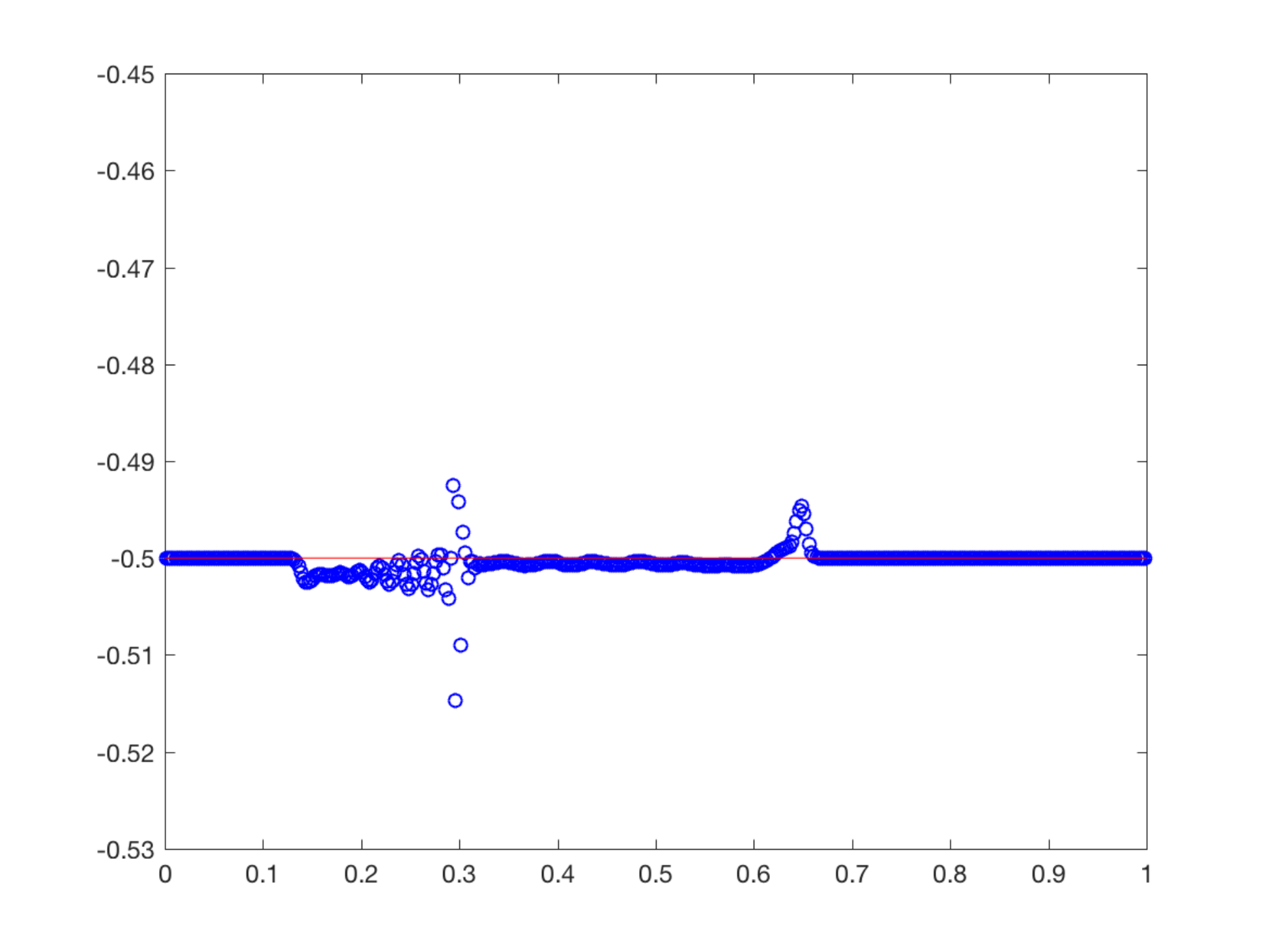}
			\end{minipage}
		}
		\subfigure[$v$]{
			\begin{minipage}[t]{0.2\textwidth}
				\centering
				\includegraphics[width=\textwidth]{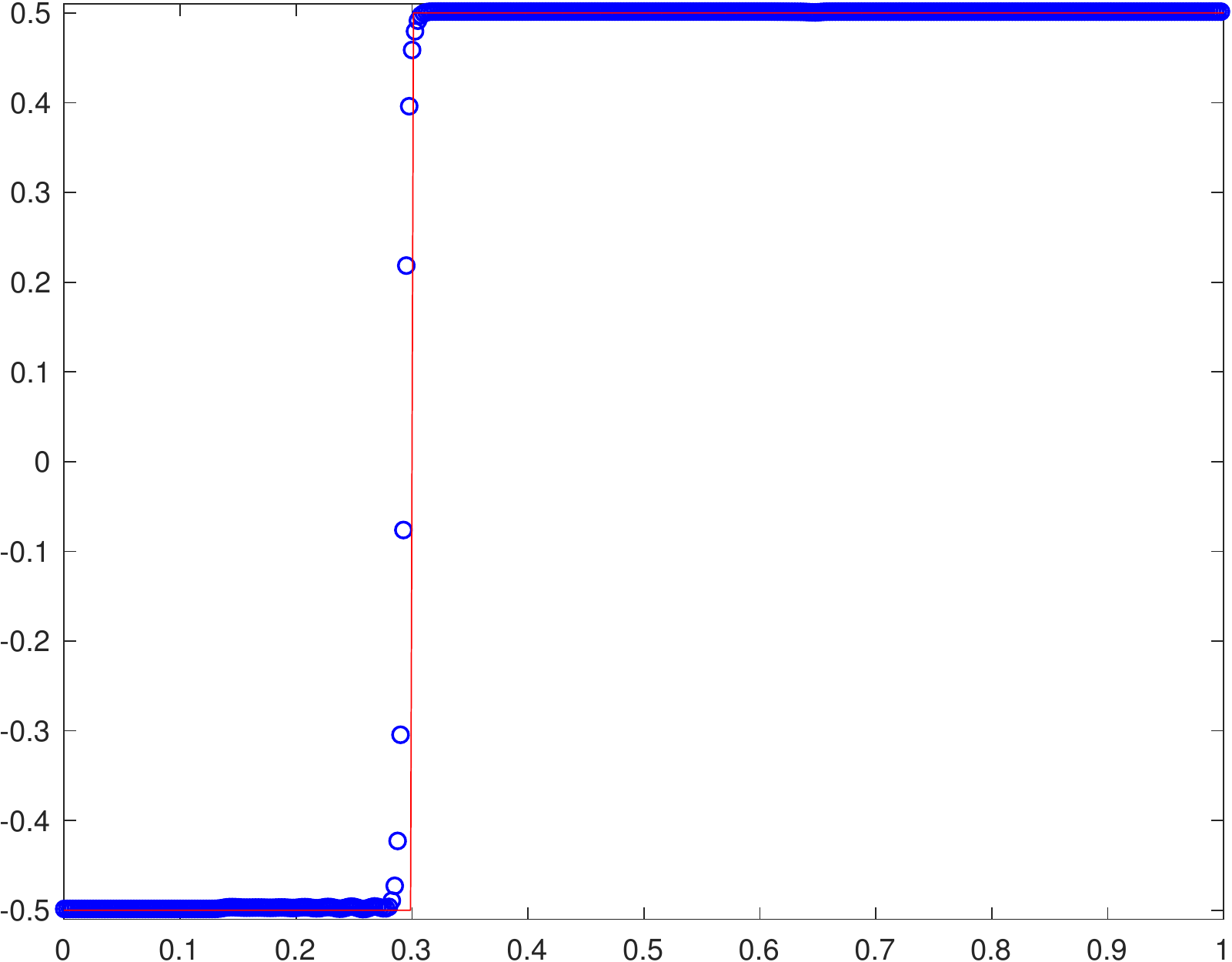}
			\end{minipage}
		}
		\subfigure[$p$]{
			\begin{minipage}[t]{0.22\textwidth}
				\centering
				\includegraphics[width=\textwidth]{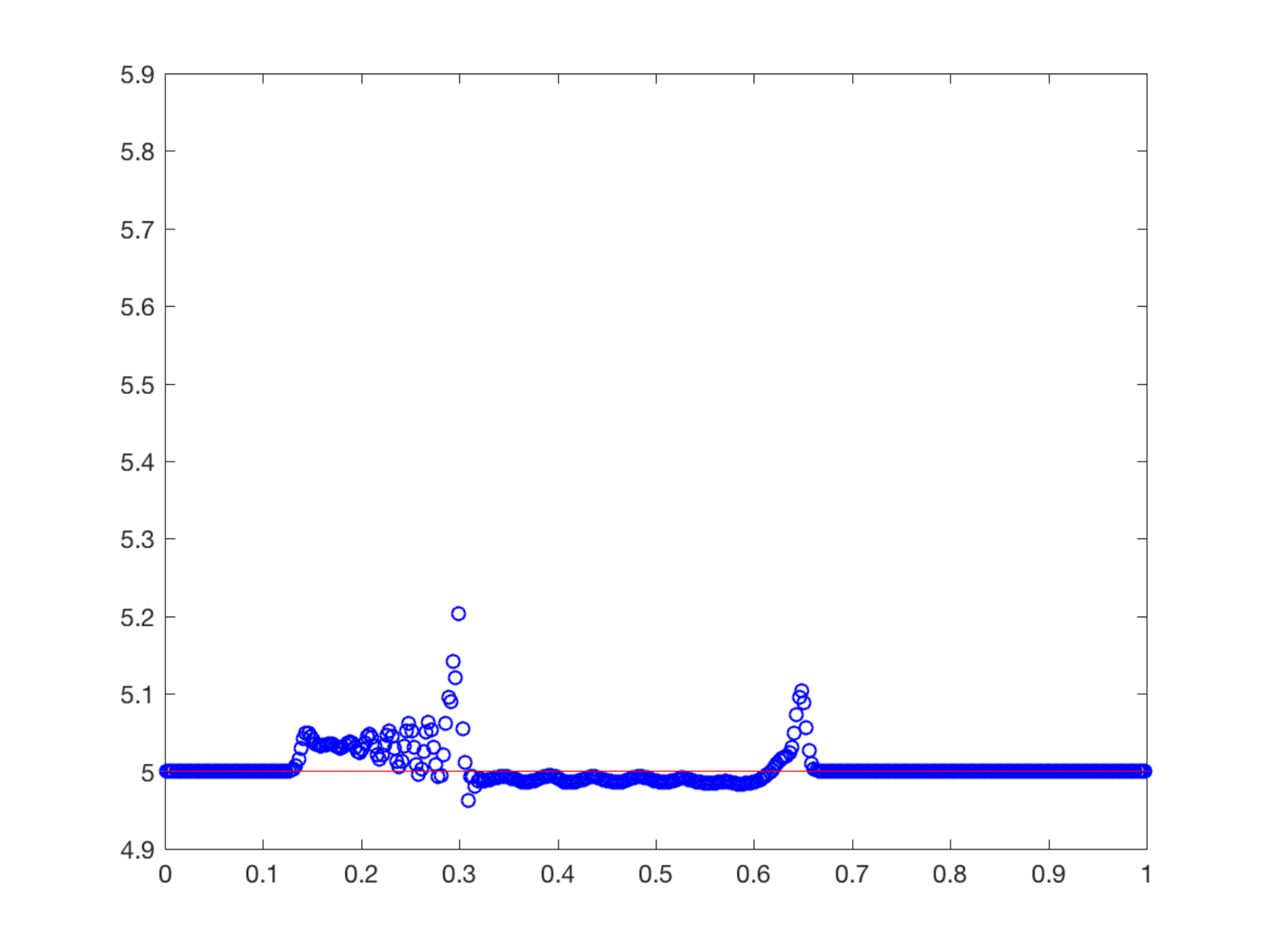}
			\end{minipage}
		}
		\caption{\small Same as Fig. \ref{fig:001} except for {two-stage fourth-order} adaptive primitive-conservative scheme {with the   reconstructed characteristic variables}.}\label{fig:example001c}
	\end{figure}
	
\begin{figure}[htbp]
	\centering
	\setlength{\abovecaptionskip}{0.cm}
	\setlength{\belowcaptionskip}{-0.cm}
	\subfigure[$\rho$]{
		\begin{minipage}[t]{0.22\textwidth}
			\centering
			\includegraphics[width=\textwidth]{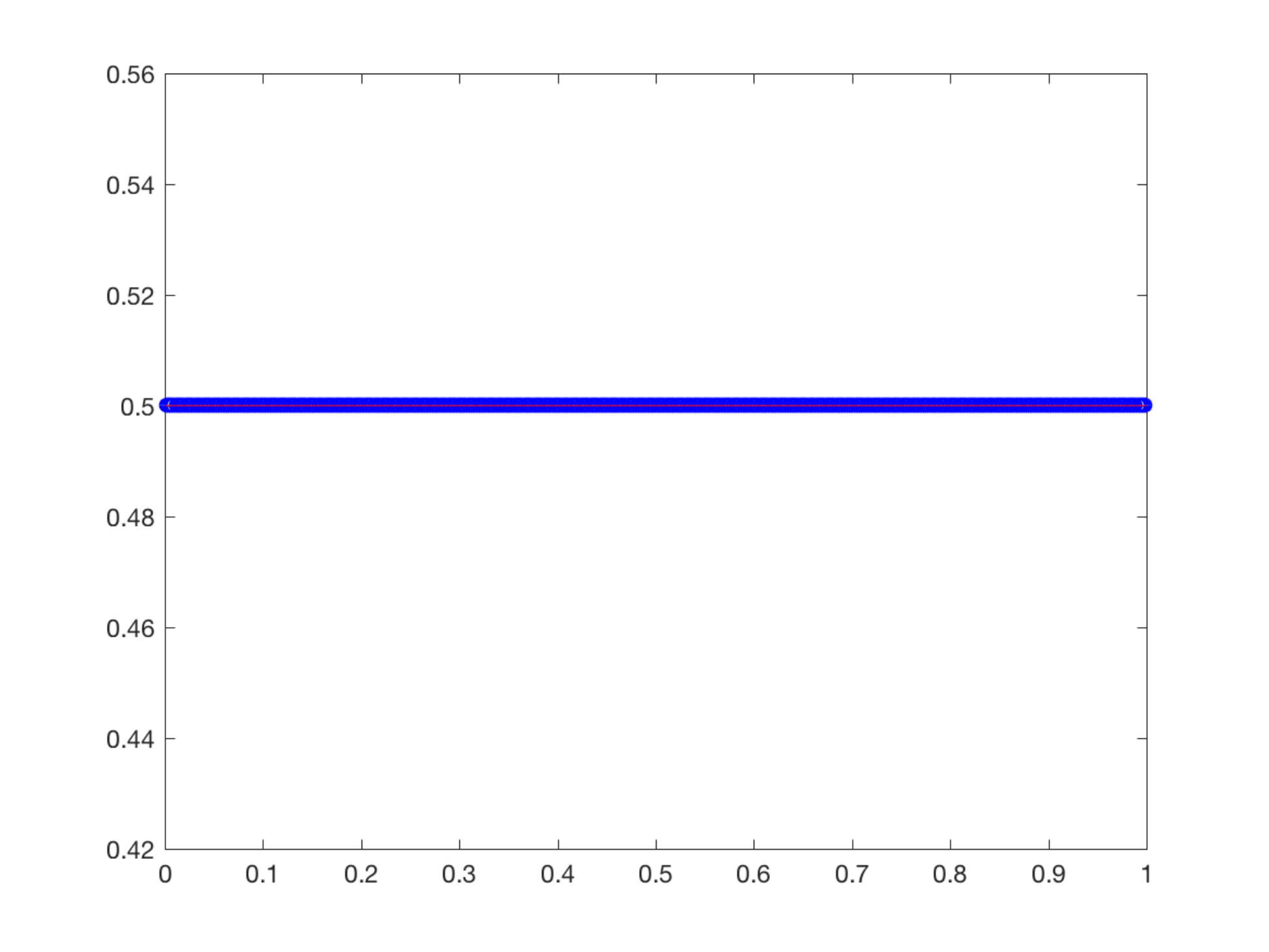}
		\end{minipage}
	}
	\subfigure[$u$]{
		\begin{minipage}[t]{0.22\textwidth}
			\centering
			\includegraphics[width=\textwidth]{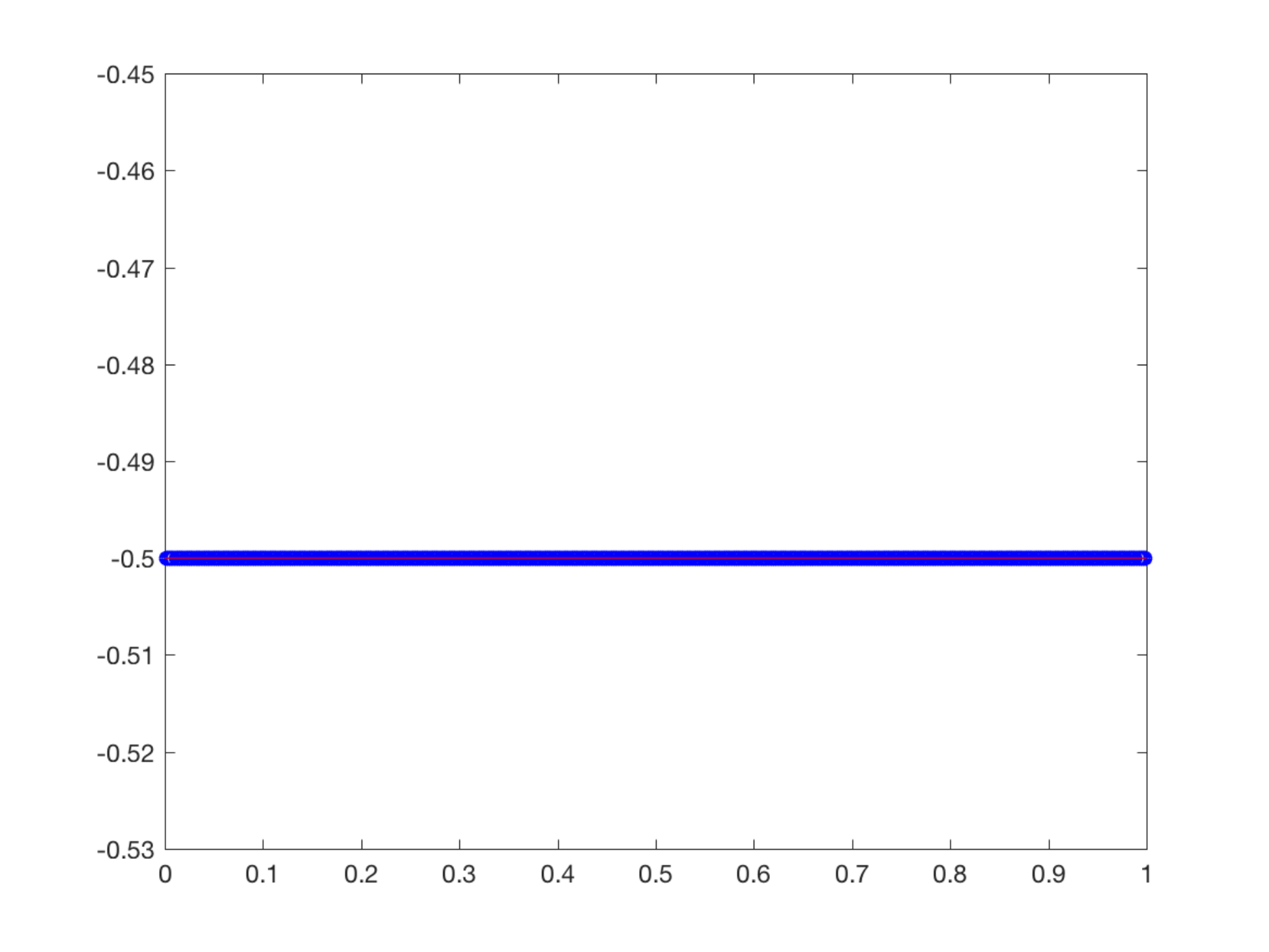}	
		\end{minipage}
	}
	\subfigure[$v$]{
		\begin{minipage}[t]{0.2\textwidth}
			\centering
			\includegraphics[width=\textwidth]{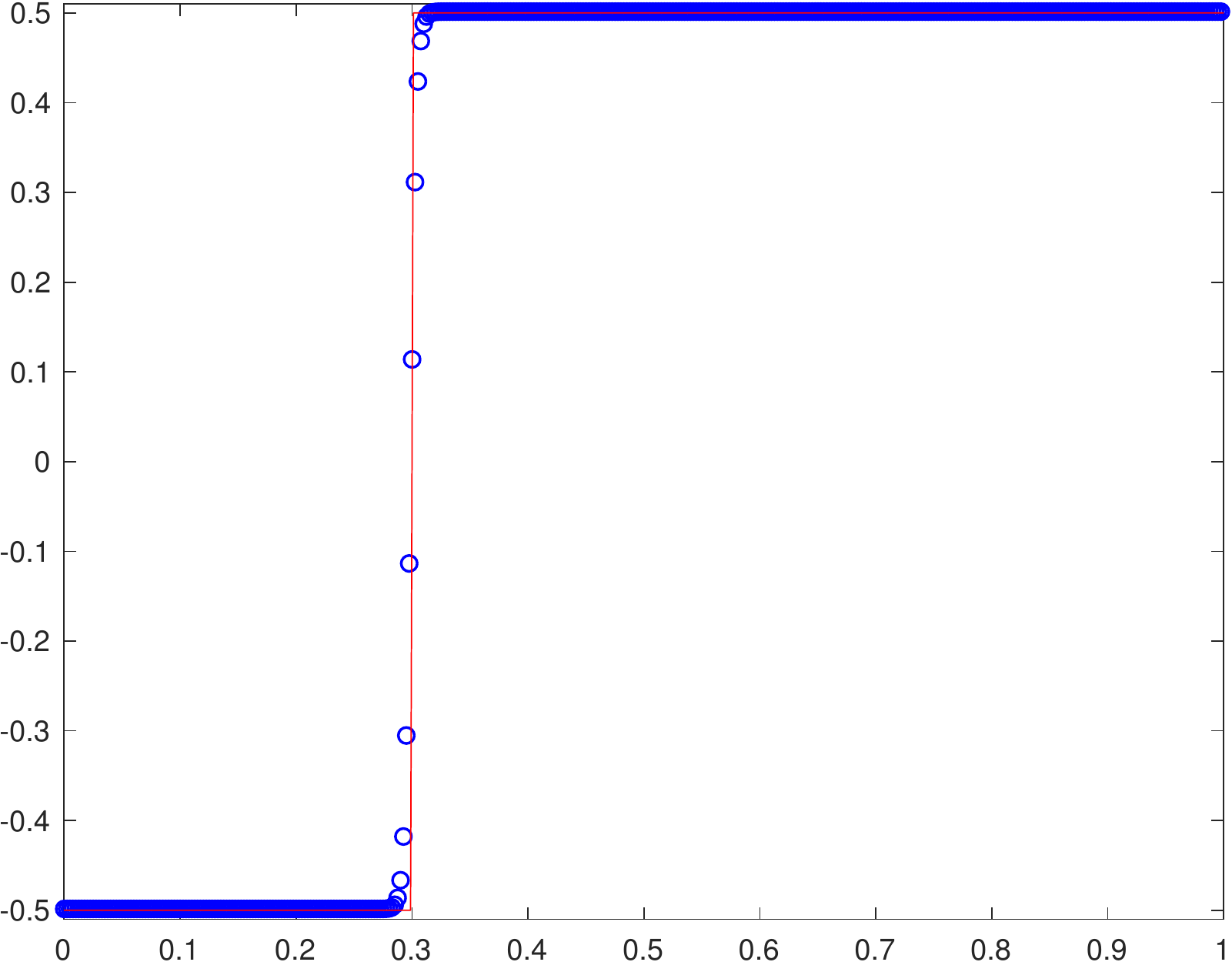}
		\end{minipage}
	}
	\subfigure[$p$]{
		\begin{minipage}[t]{0.22\textwidth}
			\centering
			\includegraphics[width=\textwidth]{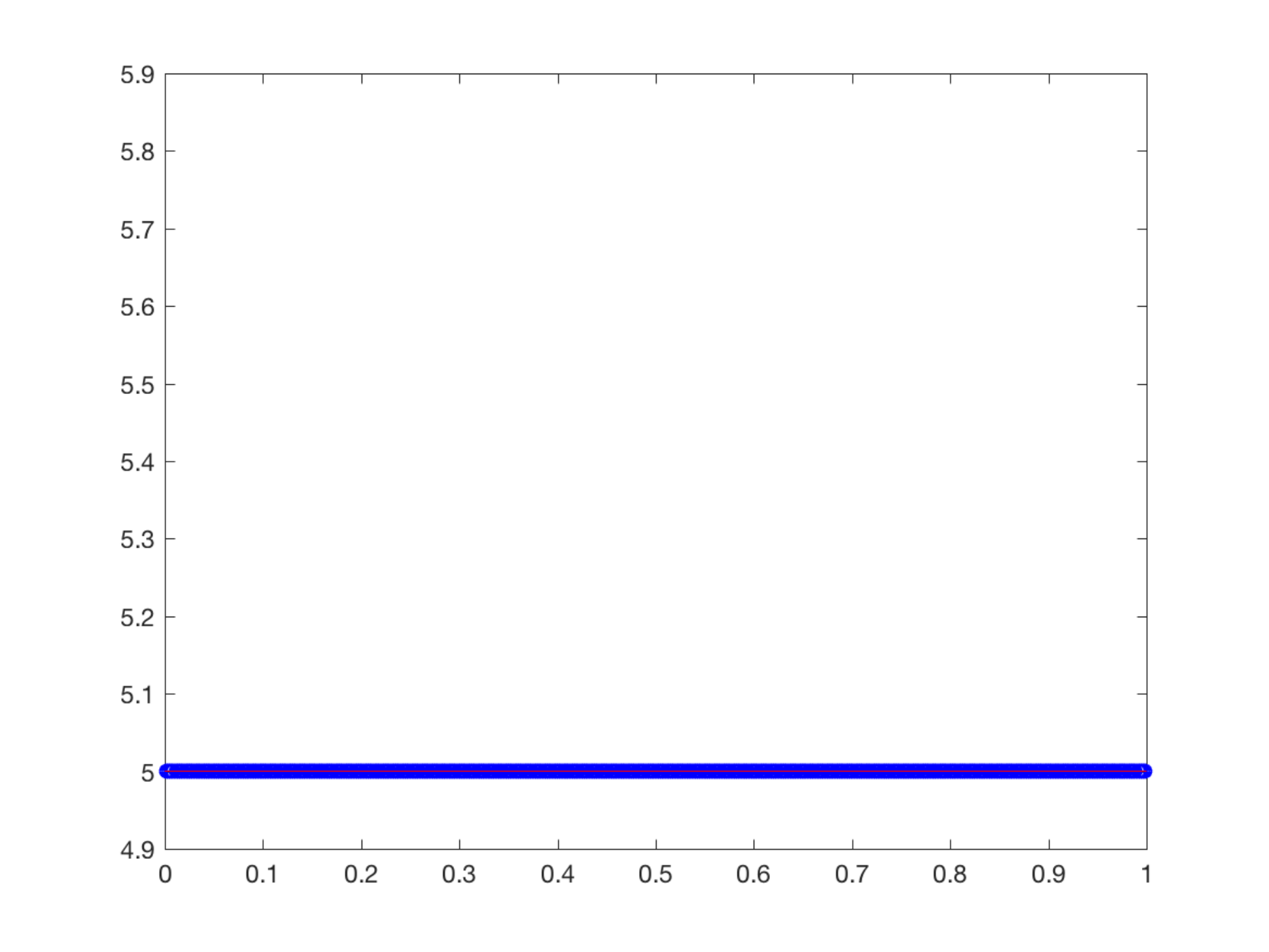}	
		\end{minipage}
	}
	\caption{\small  { Same as Fig. \ref{fig:example001c} except for   the reconstructed primitive variables}.}\label{fig:example001d}
	
\end{figure}

	%
	
\end{example}


With the given ``initial'' cell-average data $\{{\overline{\vec U}}^n_{jk}\}$,
{in the $x-$direction, we want to reconstruct $\vec U^{\pm,n}_{j+\frac12,k_l}$, $(\partial_x \vec U)^{\pm,n}_{j+\frac12,k_l}$ and $(\partial_x \vec F_2)^{\pm,n}_{j+\frac12,k_l}$, where
$	\vec U^{\pm,n}_{j+\frac12,k_l} \approx \vec U(x_{j+\frac12}\pm 0, y_{k_l}^G, t_n)$,
	$(\partial_x \vec U)^{\pm,n}_{j+\frac12,k_l} \approx (\partial_x \vec U)(x_{j+\frac12}\pm 0, y_{k_l}^G, t_n)$,
	$(\partial_y \vec F_2)^{\pm,n}_{j+\frac12,k_l} \approx (\partial_y \vec F_2)(x_{j+\frac12}\pm 0, y_{k_l}^G, t_n)$,
	here $y_{k_l}^G \in (y_{k-\frac{1}{2}}, y_{k+\frac{1}{2}})$ denotes  the associated Gauss-Legendre point,
$l=1,2,\cdots,K$.
The procedure is given as follows:
	\begin{description}
		\item[(1)] Calculate $\vec U^{\pm,n}_{j+\frac12,k_l}$ and $(\partial_x \vec U)^{\pm,n}_{j+\frac12,k_l}$
by the following two steps:
		\begin{description}
			\item[$-$] For each $j$, use the standard 5th-order WENO technique \cite{Jiang-Shu1996} to reconstruct the approximate value of $\vec U$ at the point  $y_{k_l}^G$, denoted by $\overline{\vec U}_{j,k_l}^n$, which is an approximation of $\frac{1}{\Delta x}\int_{x_{j-\frac12}}^{x_{j+\frac12}} \vec U(x,y_{k_l}^G,t_n) \; dx$, $l=1,2,\cdots,K$.
			
			\item[$-$]  Reconstruct $\vec U^{\pm,n}_{j+\frac12,k_l}$ and $(\partial_x \vec U)^{\pm,n}_{j+\frac12,k_l}$ by the initial reconstruction procedure in Section \ref{subsec:method-1d} and  the data $\{\overline{\vec U}_{j,k_l}^n\}$.
%
		\end{description}

		\item[(2)] Calculate $(\partial_y \vec F_2)^{\pm,n}_{j+\frac12,k_l}$ as follows:
		\begin{description}
			\item[$-$] Calculate $\{  {\vec F}_2({\overline{\vec U}}^n_{jk}) \}$ and
then for each $j$, use those data and the  5th-order WENO technique to reconstruct  $(\overline{\vec F}_2)_{j,k+\frac12}^{\pm,n}$, approximating $\frac{1}{\Delta x} \int_{x_{j-\frac12}}^{x_{j+\frac12}} $ $\vec F_2(x,y_{k+\frac12}\pm 0,t_n) \; dx$.
			
			\item[$-$] Calculate $ {(\overline{\pa_y \vec F_2})}^n_{j,k} = \frac{1}{\Delta y}
 \big((\overline{\vec F}_2)_{j,k+\frac12}^{-,n} - (\overline{\vec F}_2)_{j,k-\frac12}^{+,n} \big)$, and then use those data and   the  5th-order WENO technique to reconstruct $(\overline{\pa_y \vec F_2})_{j,k_l}^n$ at the point $y_{k_l}^G$. Here $(\overline{\pa_y \vec F_2})_{j,k_l}^n \approx \frac{1}{\Delta x} \int_{x_{j-\frac12}}^{x_{j+\frac12}} \pa_y \vec F_2(x,y_{k_l}^G,t_n) \; dx$ and $l=1,2,\cdots,K$.
			
			\item[$-$] Use the data $\{(\overline{\pa_y \vec F_2})_{j,k_l}^n\}$ and the 5th-order WENO technique to get $(\partial_y \vec F_2)^{\pm,n}_{j+\frac12,k_l}$.
%
		\end{description}
	\end{description}
	
}
Such reconstruction is also used at $t_*=t_n+\tau/(3-3\alpha)$,
where {$\alpha=\alpha(\hat{\tau})$
is a differentiable function of $\hat{\tau}$
	and  satisfies
	$\alpha(0)=1/3$, $\alpha\neq 1$, and $\hat{\tau}=C \tau^p$ with $p\geq 1$ and $C$ independent on $\tau$. }

The two-stage fourth-order time discretizations in Section \ref{subsec:timedis} can be applied to the 2D RHD equations by the following steps.
\begin{itemize}
	\item[Step 1.]
	In the $x$-direction, solve the local Riemann problem
	\begin{equation}
	\begin{cases}
	\vec U_t + \vec F_1(\vec U)_x = 0,\\
	\vec U(x,y_{k_l}^G,t_n) = \begin{cases}
	{\vec U^{-,n}_{j+\frac12,k_l}}, & x < x_{j+\frac{1}{2}},\\
	{\vec U^{+,n}_{j+\frac12,k_l}}, & x > x_{j+\frac{1}{2}},
	\end{cases}
	\end{cases}
	\end{equation}
	to get $\vec U_{j+\frac{1}{2},k_l}^{RP,n}$  {and $\vec V_{j+\frac{1}{2},k_l}^{RP,n}$}, and resolve the local ``quasi 1D'' GRP of
	\begin{equation} \label{eq:quasi1D}
	\vec U_t + \vec F_1(\vec U)_x = -{\widehat{(\partial_y \vec F_2)}^n_{j+\frac12,k_l}},\ t>t_n,\\
	%
	\end{equation}
	to obtain
	$\left(\frac{\pa}{\pa t}\vec U\right)_{j+\frac{1}{2},k_l}^n$ {and $\left(\frac{\pa}{\pa t}\vec V\right)_{j+\frac{1}{2},k_l}^n$},
	where
	{
		\begin{equation*}
		\widehat{(\partial_y \vec F_2)}^n_{j+\frac12,k_l}
		= \vec R\vec I^+\vec R^-  (\partial_y \vec F_2)^{-,n}_{j+\frac12,k_l}
		+ \vec R\vec I^-\vec R^-  (\partial_y \vec F_2)^{+,n}_{j+\frac12,k_l},
		\end{equation*}
	}
	and
	\begin{equation*}
	\vec A =\frac{\pa \vec F_1}{\pa \vec U}(\vec U_{j+\frac{1}{2},k_l}^{RP,n}) = \vec R \vec\Lambda \vec R^-,  \vec  \Lambda = \mbox{diag}\{\la_i\}, \
	\vec I^\pm  = \frac{1}{2}\mbox{diag}\{1\pm \mbox{sign}(\la_i)\}.
	\end{equation*}
	
	The analytical resolution of {the} ``quasi-1D'' GRP is give in Appendix \ref{Appendix:001}.
	Similarly,  solve the Riemann problem  and
	resolve {the}  ``quasi 1D'' GRP in the $y$-direction to get $\vec U_{j_l,k+\frac{1}{2}}^{RP,n}$, $\left(\frac{\pa}{\pa t}\vec U\right)_{j_l,k+\frac{1}{2}}^n $,
 {	$\vec V_{j_l,k+\frac{1}{2}}^{RP,n}$ and $\left(\frac{\pa}{\pa t}\vec V\right)_{j_l,k+\frac{1}{2}}^n$}.

	\item[Step 2.]   Compute the intermediate solutions $\overline{\vec U}_{jk}^*$ or
	$\overline{\vec V}_{jk}^*$ at $t^*$ with the adaptive procedure \cite[Section 3.3]{E.F.Toro:2013},
	whereby the conservative scheme is only applied to the cells in
	which the {shock waves} are involved and the primitive scheme is used elsewhere to address the issue mentioned in Example \ref{example2.1}.
	With the help of $\vec U_{j\pm \frac{1}{2},k_l}^{RP,n}$ and  $\vec U_{j_l,k\pm\frac{1}{2}}^{RP,n}$, the pressures $p_{j\pm \frac{1}{2},k_l}^n$ and $p_{j_l,k\pm\frac{1}{2}}^n$,
  {the} fastest shock speeds $s_{j+\frac{1}{2},k_l}^{n,L}$, $s_{j-\frac{1}{2},k_l}^{n,R}$, $s_{j_l,k+\frac{1}{2}}^{n,L}$, $s_{j-\frac{1}{2},k_l}^{n,R}$
	are first obtained and then  {we} do the followings.
	\begin{itemize}
		\item
		If
		\begin{equation*}
		\begin{cases}
		\frac{p_{j+\frac{1}{2},k_l}^n}{p_{jk}^n} > P_{\mbox{sw}}, \\
		s_{j+\frac{1}{2},k_l}^{n,L} < 0,
		\end{cases}
		\mbox{or}\;
		\begin{cases}
		\frac{p_{j-\frac{1}{2},k_l}^n}{p_{jk}^n} > P_{\mbox{sw}}, \\
		s_{j-\frac{1}{2},k_l}^{n,R} > 0,
		\end{cases}
		\mbox{or}\;
		\begin{cases}
		\frac{p_{j_l,k+\frac{1}{2}}^n}{p_{jk}^n} > P_{\mbox{sw}}, \\
		s_{j_l,k+\frac{1}{2}}^{n,L} < 0,
		\end{cases}
		\mbox{or}\;
		\begin{cases}
		\frac{p_{j_l,k-\frac{1}{2}}^n}{p_{jk}^n} > P_{\mbox{sw}}, \\
		s_{j_l,k-\frac{1}{2}}^{n,R} > 0,
		\end{cases}
		\end{equation*}
		the cell $I_{jk}$ is marked  {and the solution in $I_{jk}$ is evolved} by the conservative scheme
		\begin{equation*}
		\overline{\vec U}_{jk}^* = \overline{\vec U}^n_{jk}
		+ \frac{\tau}{3(1-\alpha)} L_{jk} (\overline{\vec U}^n)
		+  \frac{\tau^2}{12 (1-\alpha)}  \pa_t L_{jk}(\overline{\vec U}^n),
		\end{equation*}
		where \begin{align*}
		L_{jk}^n(\vec U)
		= & - \frac{1}{\Delta x}  \left( \sum_{l=1}^K \om_l \vec F_1(\vec U_{j+\frac{1}{2},k_l}^{RP,n}) - \sum_{l=1}^K \om_l \vec F_1(\vec U_{j-\frac{1}{2},k_l}^{RP,n}) \right) \;
		\\ & - \;\frac{1}{\Delta y}  \left( \sum_{l=1}^K \om_l \vec F_2(\vec U_{j_l,k+\frac{1}{2}}^{RP,n}) - \sum_{l=1}^K  \om_l\vec F_2(\vec U_{j_l,k-\frac{1}{2}}^{RP,n}) \right),
		\end{align*}
		the term $\pa_t L_{jk}(\overline{\vec U}^n)$ can be similarly given   to \eqref{eq:1dLt},
		and $P_{\mbox{sw}}=1+\epsilon$ denotes the shock sensing parameter.

		\item Otherwise, the cell $I_{jk}$ is marked to be updated by the non-conservative scheme
		\begin{equation*}
		\overline{\vec V}_{jk}^* = \overline{\vec V}^n_{jk}
		+  \frac{\tau}{3(1-\alpha)} \widetilde{L}_{jk} (\overline{\vec V}^n)
		+  \frac{\tau^2}{12 (1-\alpha)} \pa_t \widetilde{L}_{jk}(\overline{\vec V}^n),
		\end{equation*}
		with
		
		\begin{align*}
		-\widetilde{L}_{jk} (\overline{\vec V}^n)
		= &   \frac{1}{\Delta x}  \sum_{l=1}^K \om_l \left(
		\int_{ { \vec V^{-,n}_{j-\frac{1}{2},k_l} } }
		^ {\vec V_{j-\frac{1}{2},k_l}^{RP,n}}      \widetilde{\vec A}^+ \;d\vec V
		+   \int^{ {  \vec V^{+,n}_{j+\frac{1}{2},k_l} } }   _{\vec V_{j+\frac{1}{2},k_l}^{RP,n}}
		\widetilde{\vec A} ^- \;d\vec V
		+ \int_{\vec V_{j-\frac{1}{2},k_l}^{RP,n}} ^{\vec V_{j+\frac{1}{2},k_l}^{RP,n}}
		\widetilde{\vec A}  \;d\vec V   \right)
		\\
		+ &  \frac{1}{\Delta y}  \sum_{l=1}^K \om_l \left(  \int_{  { \vec V^{-,n}_{j_l,k-\frac{1}{2}} }   }^{\vec V_{j_l,k-\frac{1}{2}}^{RP,n}}
		\widetilde{\vec B}^+ \;d\vec V
		+    \int^{ {  \vec V^{+,n}_{j_l,k+\frac{1}{2}} } }_{\vec V_{j_l,k+\frac{1}{2}}^{RP,n}}
		\widetilde{\vec B}^- \;d\vec V
		+ \int_{\vec V_{j_l,k-\frac{1}{2}}^{RP,n}} ^{\vec V_{j_l,k+\frac{1}{2}}^{RP,n}}
		\widetilde{\vec B}  \;d\vec V   \right),
		\end{align*}
		and
		\begin{align*}
		-\pa_t \widetilde{L}_{jk} (\overline{\vec V}^n)
		= &  \frac{1}{\Delta x} \sum_{l=1}^K \om_l \left(    \widetilde{\vec A}(\vec V_{j+\frac{1}{2},k_l}^{RP,n}) \cdot \pa_t \vec V_{j+\frac{1}{2},k_l}^n
		- \widetilde{\vec A} (\vec V_{j-\frac{1}{2},k_l}^{RP,n}) \cdot \pa_t \vec V_{j-\frac{1}{2},k_l}^n   \right) \\
		+ & \frac{1}{\Delta y} \sum_{l=1}^K \om_l \left(    \widetilde{\vec B}(\vec V_{j_l,k+\frac{1}{2}}^{RP,n}) \cdot \pa_t \vec V_{j_l,k+\frac{1}{2}}^n - \widetilde{\vec B} (\vec V_{j_l,k-\frac{1}{2}}^{RP,n})\cdot \pa_t \vec V_{j_l,k-\frac{1}{2}}^n   \right),
		\end{align*}
		where  {
		$ \vec V^{\pm,n}_{j+\frac{1}{2},k_l},  \vec V^{\pm,n}_{j_l,k+\frac{1}{2}} $
		are obtained from $ \vec U^{\pm,n}_{j+\frac{1}{2},k_l},  \vec U^{\pm,n}_{j_l,k+\frac{1}{2}} $.
		}
		The above integrals are evaluated by using a numerical integration such as the Gauss-Legendre quadrature along a simple canonical path defined by
		\begin{equation}
		\vec \Psi(s; \vec V_L, \vec V_R) = \vec V_L + s(\vec V_R - \vec V_L), \quad s \in [0,1].
		\end{equation}
	\end{itemize}
	
	\item[Step 3.]  With the ``initial'' data $\{\overline{\vec U}^*_{{jk}}\}$, {reconstruct values $\vec U^{\pm,*}_{j+\frac12,k_l}, (\partial_x \vec U)^{\pm,*}_{j+\frac12,k_l}, (\partial_y \vec F_2)^{\pm,*}_{j+\frac12,k_l}$ and $\vec U^{\pm,*}_{j_l,k+\frac12}, (\partial_x \vec U)^{\pm,*}_{j_l,k+\frac12}, (\partial_x \vec F_1)^{\pm,*}_{j_l,k+\frac12}$}.  Then, similar to Step 1,
	compute  $\vec U_{{j+\frac{1}{2},k_l}}^{RP,*}$,
	$\left(\frac{\pa}{\pa t}\vec U\right)_{{j+\frac{1}{2},k_l}}^* $,  $\vec U_{j_l,k+\frac{1}{2}}^{RP,*}$ and $\left(\frac{\pa}{\pa t}\vec U\right)_{{j_l,k+\frac{1}{2}}}^* $.
	
	\item[Step 4.]  Evolve the solution $\overline{\vec U}_{{jk}}^{n+1}$ or
	$\overline{\vec V}_{{jk}}^{n+1}$ at $t_{n+1} = t_n + \Delta t$ by
	the adaptive primitive-conservative scheme in Step 2
	with
	\begin{equation}
	\overline{\vec U}_{jk}^{n+1} = \overline{\vec U}_{jk}^n + \tau L_{jk}( \overline{\vec U}^n)
	+ \frac{\alpha\tau^2}{2} \pa_t L_{jk}(\overline{\vec U}^n)
	+ \frac{(1-\alpha)\tau^2}{2} \pa_t L_{jk}(\overline{\vec U}^*),
	\end{equation}
	and
	\begin{equation}
	\overline{\vec V}_{jk}^{n+1} = \overline{\vec V}_{jk}^n + \tau \widetilde{L}_{jk}( \overline{\vec V}^n)
	+ \frac{\alpha\tau^2}{2} \pa_t \widetilde{L}_{jk}(\overline{\vec V}^n)
	+ \frac{(1-\alpha)\tau^2}{2} \pa_t \widetilde{L}_{jk}(\overline{\vec V}^*).
	\end{equation}
	%
\end{itemize}



\section{Numerical results}
\label{sec:example}
{In this section, several one-dimensional and two-dimensional tests are presented to demonstrate the performance of our methods.}
Unless otherwise stated, the time stepsizes for the 1D and 2D schemes are respectively chosen as
\begin{equation*}
\tau = \frac{\mu \Delta x}{\max_{\ell, j}\{ |\la_\ell^1( \overline{\vec U}_j^n )|\}},
\end{equation*}
and
\begin{equation*}
\tau = \frac{\mu}{
	\max_{\ell,j,k}\{   |\la_\ell^1 \left( \overline{\vec U}_{jk}^n \right) | \} /{\Delta x}
	+ \max_{\ell,j,k}\{   | \la_\ell^2\left( \overline{\vec U}_{jk}^n \right) | \} /{\Delta y} },
\end{equation*}
where $\la_\ell^1$ ({resp.} $\la_\ell^2$) is the $\ell$th eigenvalue of   2D  RHD equations in the direction $x$ ({resp.} $y$), $\ell=1,2,3,4$.
{The} parameter $\al$ is taken $\frac13$, the CFL number $\mu$ are  taken as  {$0.7$ and $0.5$}
for the 1D and 2D problems, respectively.
Our numerical experiments  { show} that there is no obvious difference between  $\alpha=\frac13$ {and
$\alpha=(1-6 {\tau})/(3-6 {\tau})$ or $\alpha=\frac13+\tau$}.
Here we take  $K \geq 3$ in order to ensure that {the degree of} the algebraic precision of   corresponding quadrature is at least 4.

\subsection{One-dimensional case}
\begin{example}[Smooth problem] \label{example3.1}\rm
	It is used to verify the numerical accuracy. The initial data are
	taken as
	\begin{equation*}
	(\rho ,  u , p)(x,0) \; = \;( 1+0.2\sin(2x), 0.2, 1), \ x\in [0,\pi],
	\end{equation*}
	and the periodic boundary condition is specified.  The exact solutions can be given by
	\begin{equation*}
	\rho (x,t) \; = \; 1+0.2\sin\big(2(x-u  t)\big),  \  u(x,t) = 0.2,\  p(x,t) = 1.
	\end{equation*}
	In our computations,  the adiabatic index $\Gamma=5/3$ and {$ \tau = \frac{\mu \Delta x^{5/4}}{\max_{\ell, j}\{ |\la_\ell^1( \overline{\vec U}_j^n )|\}}, $}
	the computational domain $[0,\pi]$ is divided into $N$ uniform cells.
	Tables \ref{tabel:001}$-$\ref{table:001b} list the errors and convergence rates in $\rho$ at $t = 2$
	obtained by using our 1D method with different $\alpha$. It is seen that the two-stage schemes  {can} get the theoretical orders.
	
	\begin{table}[htpb]
		\centering
		\caption{The errors and convergence rates for solution at $t = 2$.  $\alpha=\frac13$. } \label{tabel:001}
		\begin{tabular}{ccccccc}
			\hline
			$N$ & $l^1$ error & order & $l^2$ error & order & $l^{\infty}$ error & order \\ \hline
			5    &2.8450e-02   & -    &3.2450e-02   & -    &4.5761e-02   & -\\
			10   &2.5393e-03   & 3.4859    &2.9509e-03   & 3.4590    &3.8805e-03   & 3.5598    \\
			20   &1.1042e-04   & 4.5233    &1.3168e-04   & 4.4861    &1.9960e-04   & 4.2811    \\
			40   &3.3904e-06   & 5.0255    &4.0143e-06   & 5.0357    &7.2420e-06   & 4.7846    \\
			80   &1.0513e-07   & 5.0112    &1.2192e-07   & 5.0412    &2.1990e-07   & 5.0415    \\
			160    &3.3151e-09   & 4.9870    &3.7593e-09   & 5.0193    &6.9220e-09   & 4.9895    \\
			\hline
		\end{tabular}
	\end{table}
	
	\begin{table}[htpb]
		\centering
		\caption{{Same as Table \ref{tabel:001} except for }  $\alpha=\frac{1-6\tau}{3-6\tau}$.} \label{table:001a}	
		\begin{tabular}{ccccccc}
			\hline
			$N$ & $l^1$ error & order & $l^2$ error & order & $l^{\infty}$ error & order \\ \hline
			5    &2.8452e-02   & -    &3.2453e-02   & -    &4.5767e-02   & -\\
			10   &2.5392e-03   & 3.4861    &2.9509e-03   & 3.4591    &3.8805e-03   & 3.5600    \\
			20   &1.1042e-04   & 4.5233    &1.3168e-04   & 4.4861    &1.9960e-04   & 4.2810    \\
			40   &3.3904e-06   & 5.0255    &4.0143e-06   & 5.0357    &7.2420e-06   & 4.7846    \\
			80   &1.0513e-07   & 5.0112    &1.2192e-07   & 5.0412    &2.1990e-07   & 5.0415    \\
			160    &3.3151e-09   & 4.9870    &3.7593e-09   & 5.0193    &6.9221e-09   & 4.9895    \\
			\hline
		\end{tabular}
	\end{table}
	
	\begin{table}[htpb]
		\centering
		\caption{{Same as Table \ref{tabel:001} except for }    $\alpha=\frac13 + \tau$. } \label{table:001b}	
		\begin{tabular}{ccccccc}
			\hline
			$N$ & $l^1$ error & order & $l^2$ error & order & $l^{\infty}$ error & order \\ \hline
			5    &2.8448e-02   & -    &3.2447e-02   & -    &4.5756e-02   & -\\
			10   &2.5393e-03   & 3.4858    &2.9510e-03   & 3.4588    &3.8805e-03   & 3.5596    \\
			20   &1.1042e-04   & 4.5233    &1.3168e-04   & 4.4861    &1.9960e-04   & 4.2811    \\
			40   &3.3904e-06   & 5.0255    &4.0143e-06   & 5.0357    &7.2420e-06   & 4.7846    \\
			80   &1.0513e-07   & 5.0112    &1.2192e-07   & 5.0412    &2.1990e-07   & 5.0415    \\
			160    &3.3151e-09   & 4.9870    &3.7593e-09   & 5.0193    &6.9224e-09   & 4.9894 \\
			\hline
		\end{tabular}
	\end{table}
	
\end{example}

\begin{example}[Riemann problems] \label{example3.2}\rm
	This example  { considers four} Riemann problems, whose
	initial data are given in Table \ref{table:002} with
	the initial discontinuity  located at $x = 0.5$ in the computational domain $[0,\,1]$.
	The adiabatic index $\G$ is taken as  $5/3$, but  $4/3$ for the third problem.
	The numerical solutions (``$\circ$'') at $t = 0.4$ are displayed in Figs. \ref{fig:002}-\ref{fig:005} with  400 uniform cells, respectively.
	The exact solutions ({``solid line"}) with 2000 uniform cells are also provided for comparison.
	It is seen that the numerical solutions are in good agreement with the exact,
	and the shock and rarefaction waves and contact discontinuities are well captured,
	and the positivity of the density and the pressure can be well-preserved.
	However, there exist {slight} oscillations in the density behind the left-moving
	shock wave of {\tt RP3}  and  serious undershoots in the density at $x = 0.5$ {of} {\tt RP4}, similar to those in the literature,
	see e.g. \cite{Wu-Tang-JCP2015,Wu-Yang-Tang:2014,Yang-He-Tang:2011}{.}
	It is worth noting that no obvious oscillation is observed in the densities of {\tt RP3}  obtained by the Runge-Kutta central DG methods \cite{Zhao-Tang:2017a} and {the} adaptive moving mesh method \cite{HeAdaptiveRHD}.
	\begin{table}[htpb]
		\centering
		\caption{ Initial data of four RPs. } \label{table:002}
		\begin{tabular}{c|c|ccc|c|c|ccc}
			\hline
			\multicolumn{2}{c|}{}						& $\rho$ & $u$   & $p$ & \multicolumn{2}{c|}{}						& $\rho$ & $u$   & $p$\\
			\hline
			\hline
			\multirow{2}{*}{\tt RP1} 		& left state	 	& 10 & 0 & 40/3
			& \multirow{2}{*}{\tt RP2} 		& left state	 	& 1 & 0 & $10^{3}$ \\
			\cline{2-5}
			\cline{7-10}
			& right state	& 1 & 0 & $10^{-6}$ &						& right state	& 1 & 0 & $10^{-2}$ \\
			\hline
			\multirow{2}{*}{\tt RP3} 	& left state	 	& 1 & 0.9 & 1
			& \multirow{2}{*}{\tt RP4} 	& left state	 	& 1 & $-0.7$ & 20\\
			\cline{2-5}
			\cline{7-10}
			& right state	& 1 & 0 & 10 & 						& right state	& 1 & 0.7 & 20\\
			\hline
		\end{tabular}
	\end{table}
	
	\begin{figure}[htbp]
		\setlength{\abovecaptionskip}{0.cm}
		\setlength{\belowcaptionskip}{-0.cm}
		\subfigure[$\rho/10$]{
			\begin{minipage}[t]{0.3\textwidth}
				\centering
				\includegraphics[width=\textwidth]{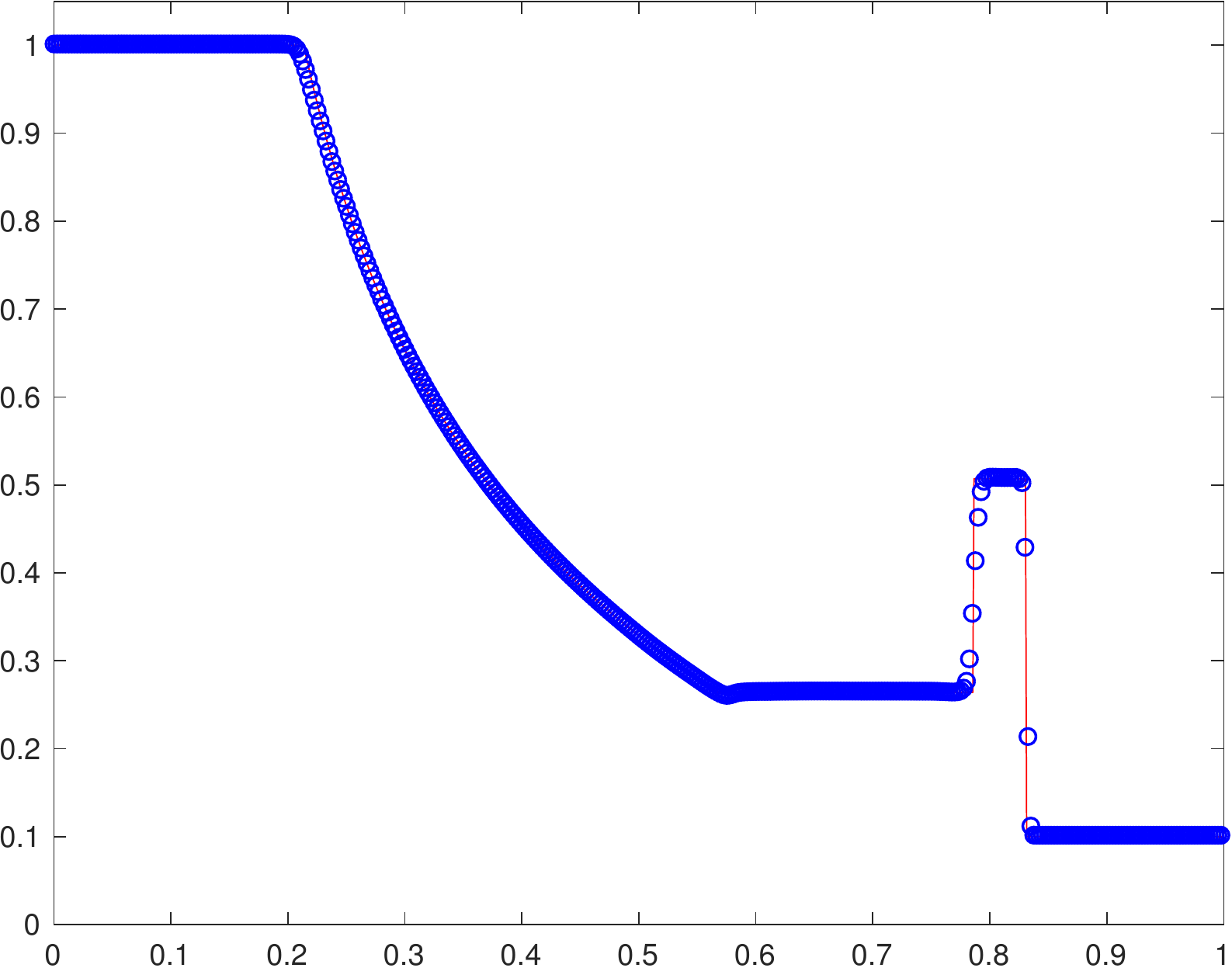}
			\end{minipage}
		}
		\subfigure[$u$]{
			\begin{minipage}[t]{0.3\textwidth}
				\centering
				\includegraphics[width=\textwidth]{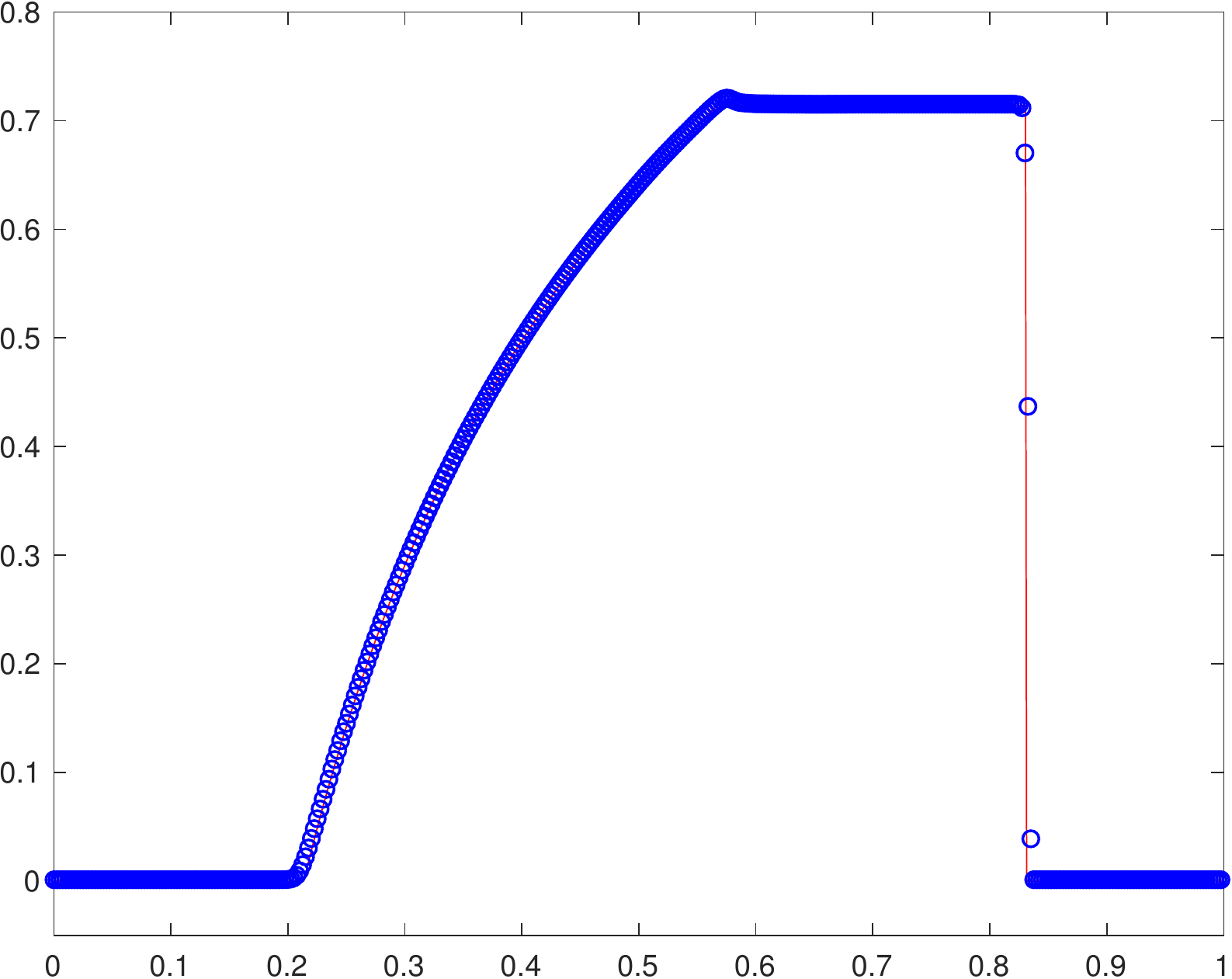}
			\end{minipage}
		}
		\subfigure[$\frac{3}{40}p$]{
			\begin{minipage}[t]{0.3\textwidth}
				\centering
				\includegraphics[width=\textwidth]{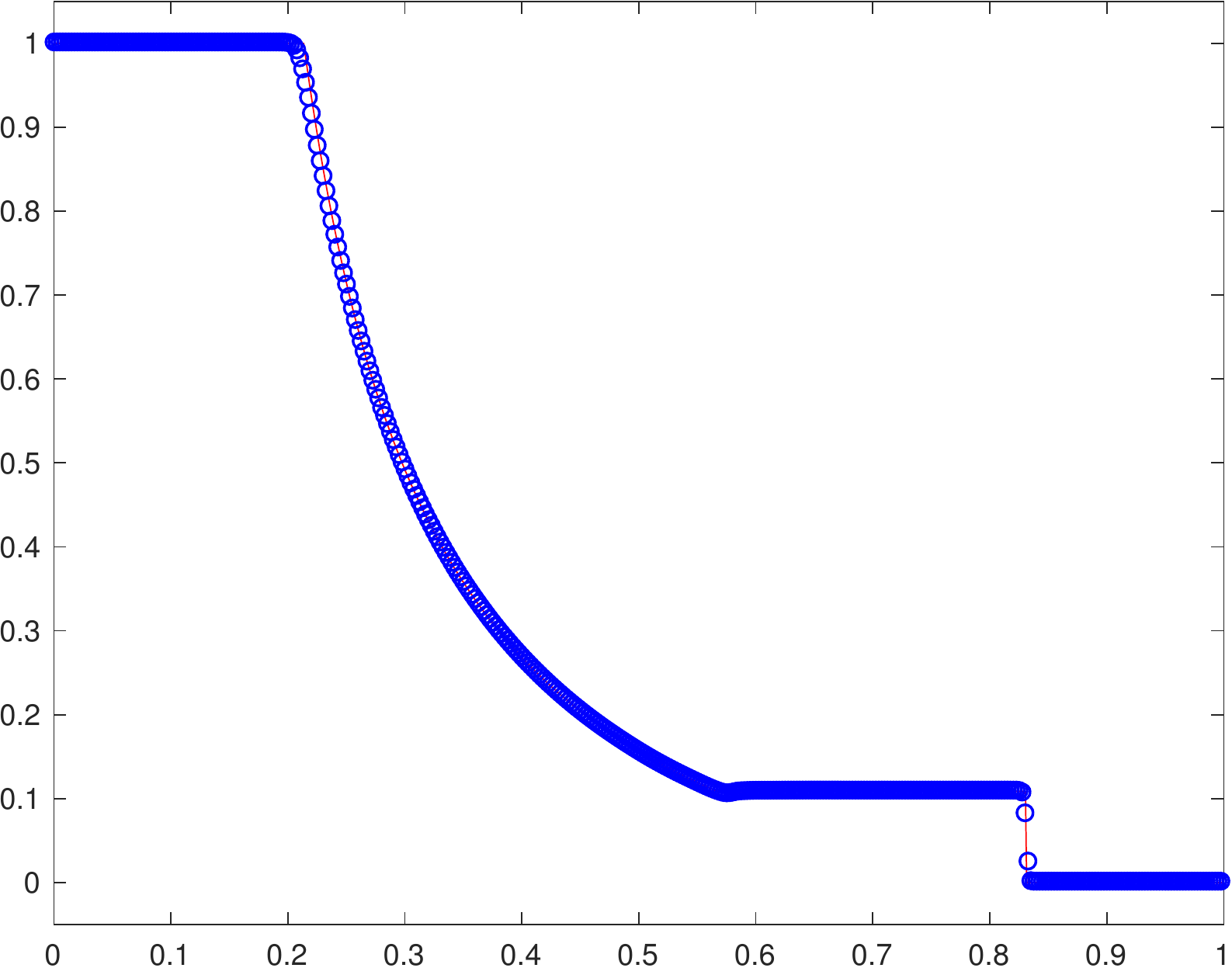}
			\end{minipage}
		}
		\caption{\small{{\tt RP1} in Example \ref{example3.2}:  The solutions at $t = 0.4$. }} \label{fig:002}
	\end{figure}

	\begin{figure}[htbp]
		\setlength{\abovecaptionskip}{0.cm}
		\setlength{\belowcaptionskip}{-0.cm}
		\subfigure[$\rho/7$]{
			\begin{minipage}[t]{0.3\textwidth}
				\centering
				\includegraphics[width=\textwidth]{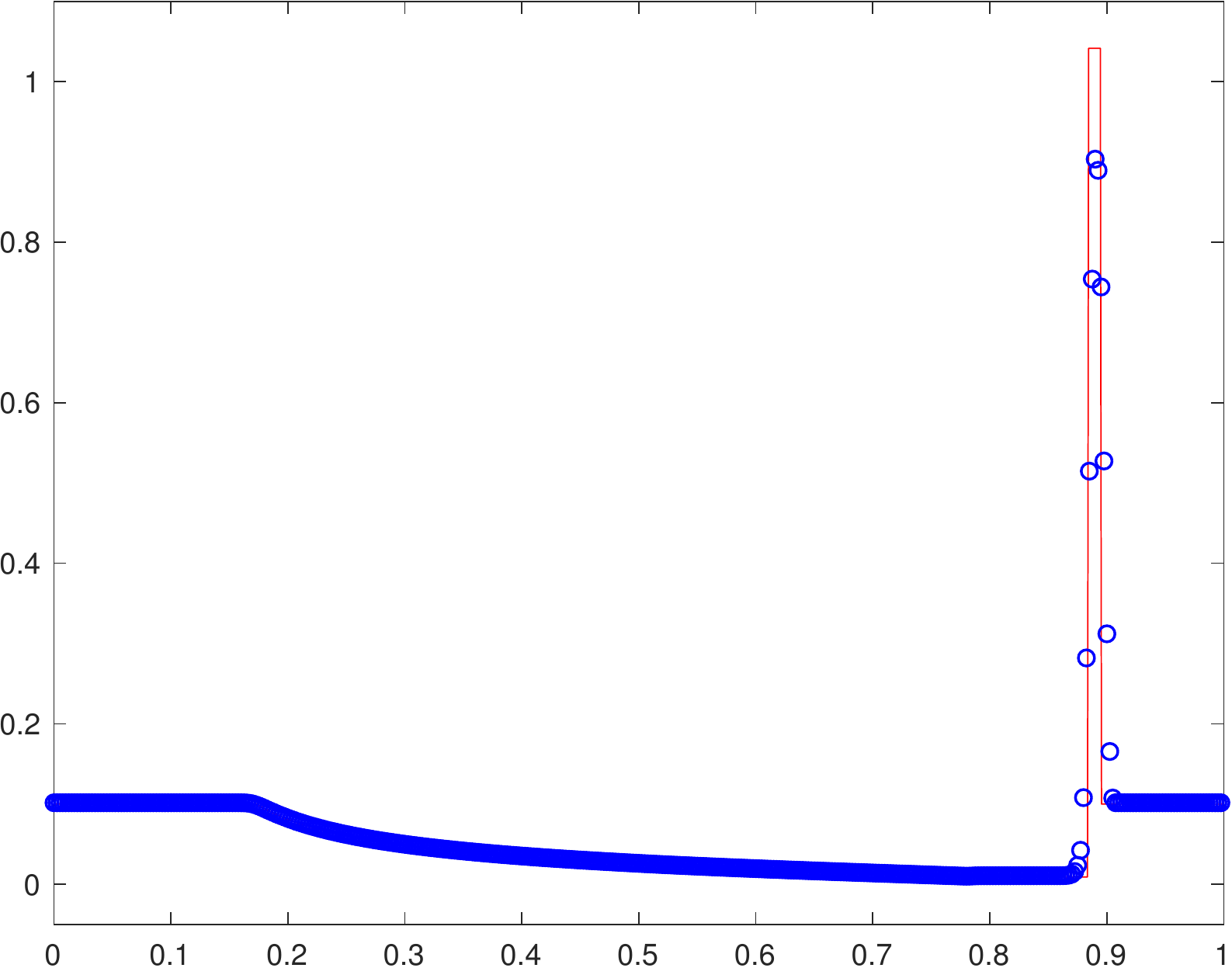}
			\end{minipage}
		}
		\subfigure[$u$]{
			\begin{minipage}[t]{0.3\textwidth}
				\centering
				\includegraphics[width=\textwidth]{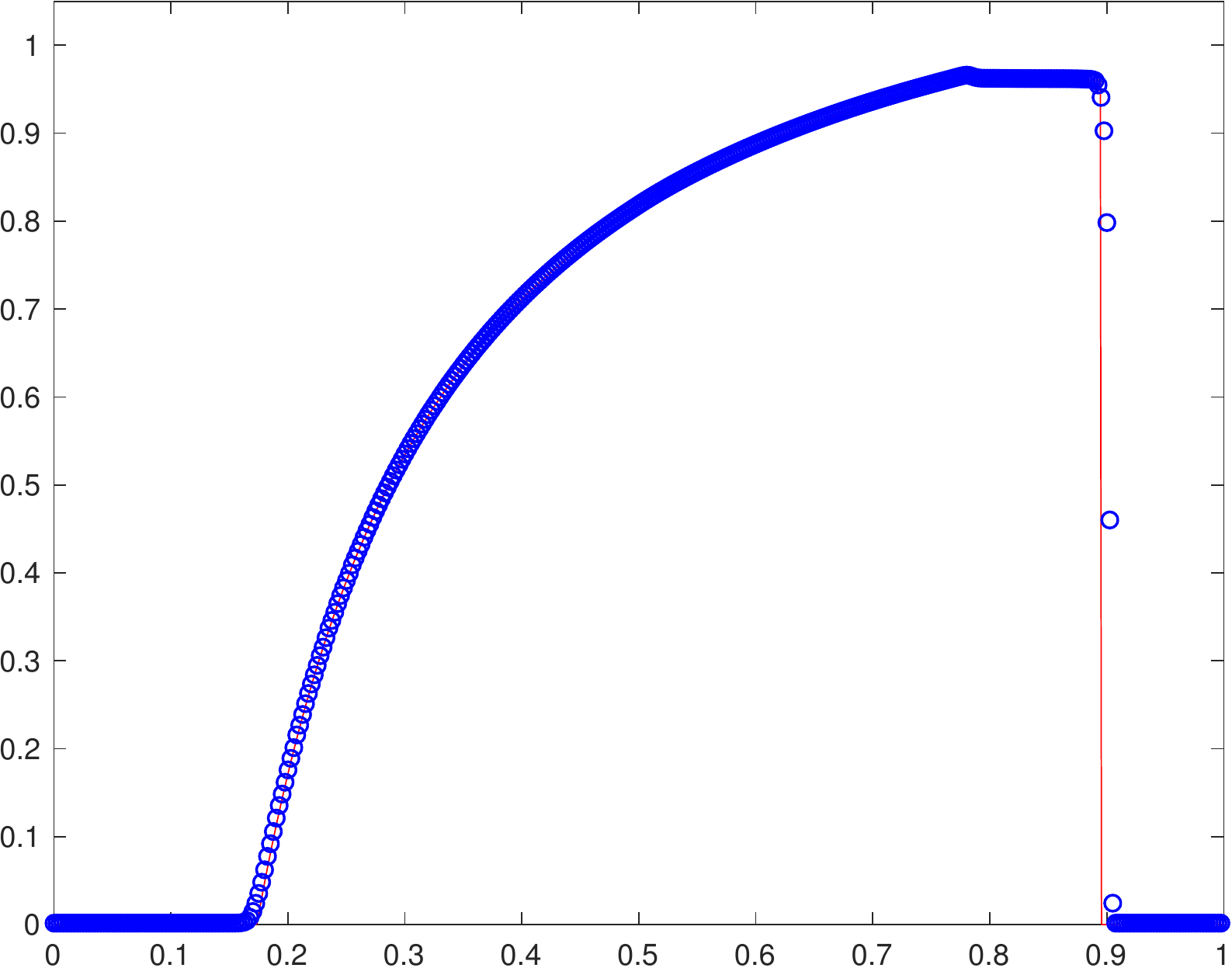}
			\end{minipage}
		}
		\subfigure[$p/1000$]{
			\begin{minipage}[t]{0.3\textwidth}
				\centering
				\includegraphics[width=\textwidth]{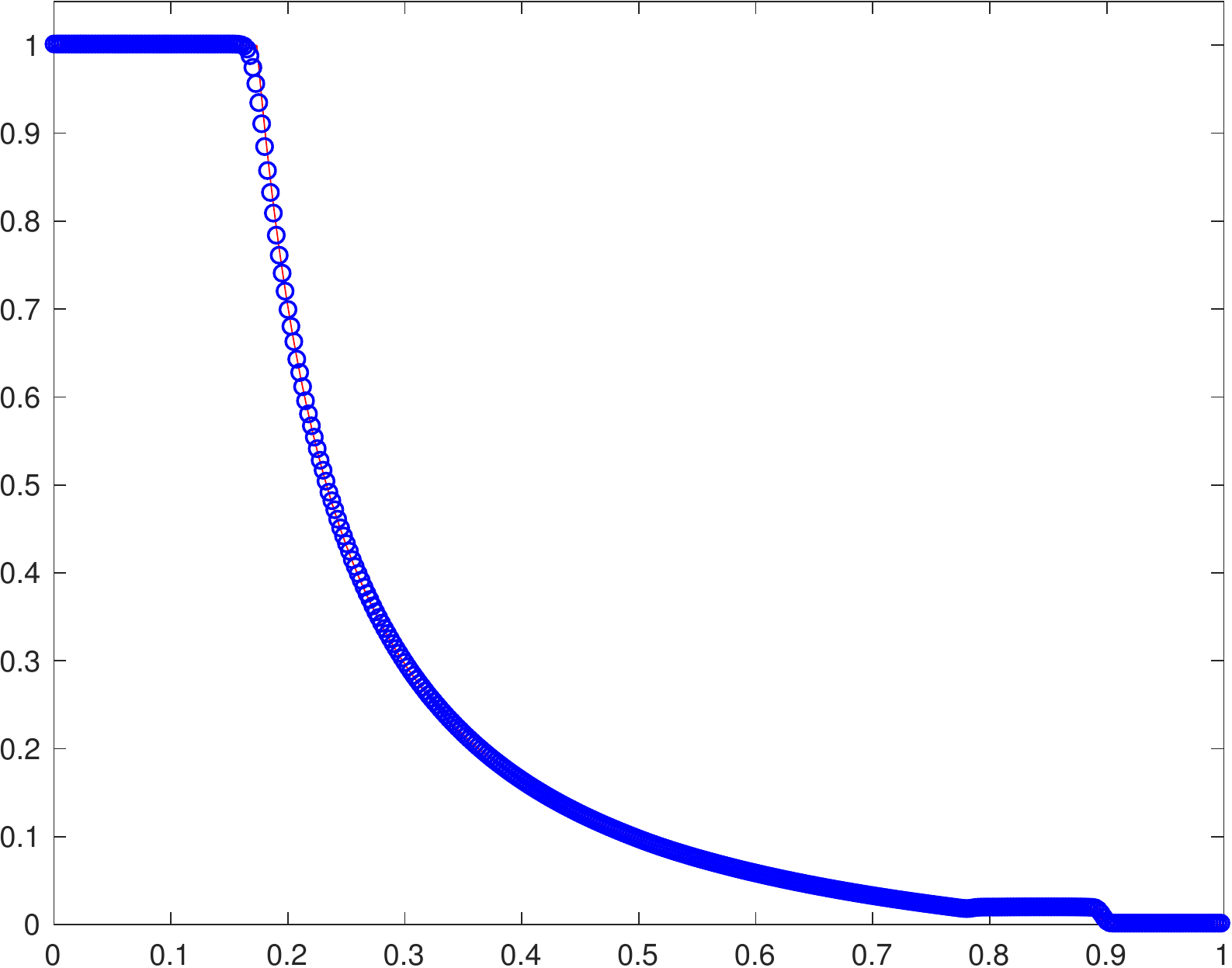}
			\end{minipage}
		}
		\caption{\small{{\tt RP2} in Example \ref{example3.2}: The solutions at $t = 0.4$. }}
		\label{fig:003}
	\end{figure}
	
	\begin{figure}[htbp]
		\setlength{\abovecaptionskip}{0.cm}
		\setlength{\belowcaptionskip}{-0.cm}
		\subfigure[$\rho/7$]{
			\begin{minipage}[t]{0.3\textwidth}
				\centering
				\includegraphics[width=\textwidth]{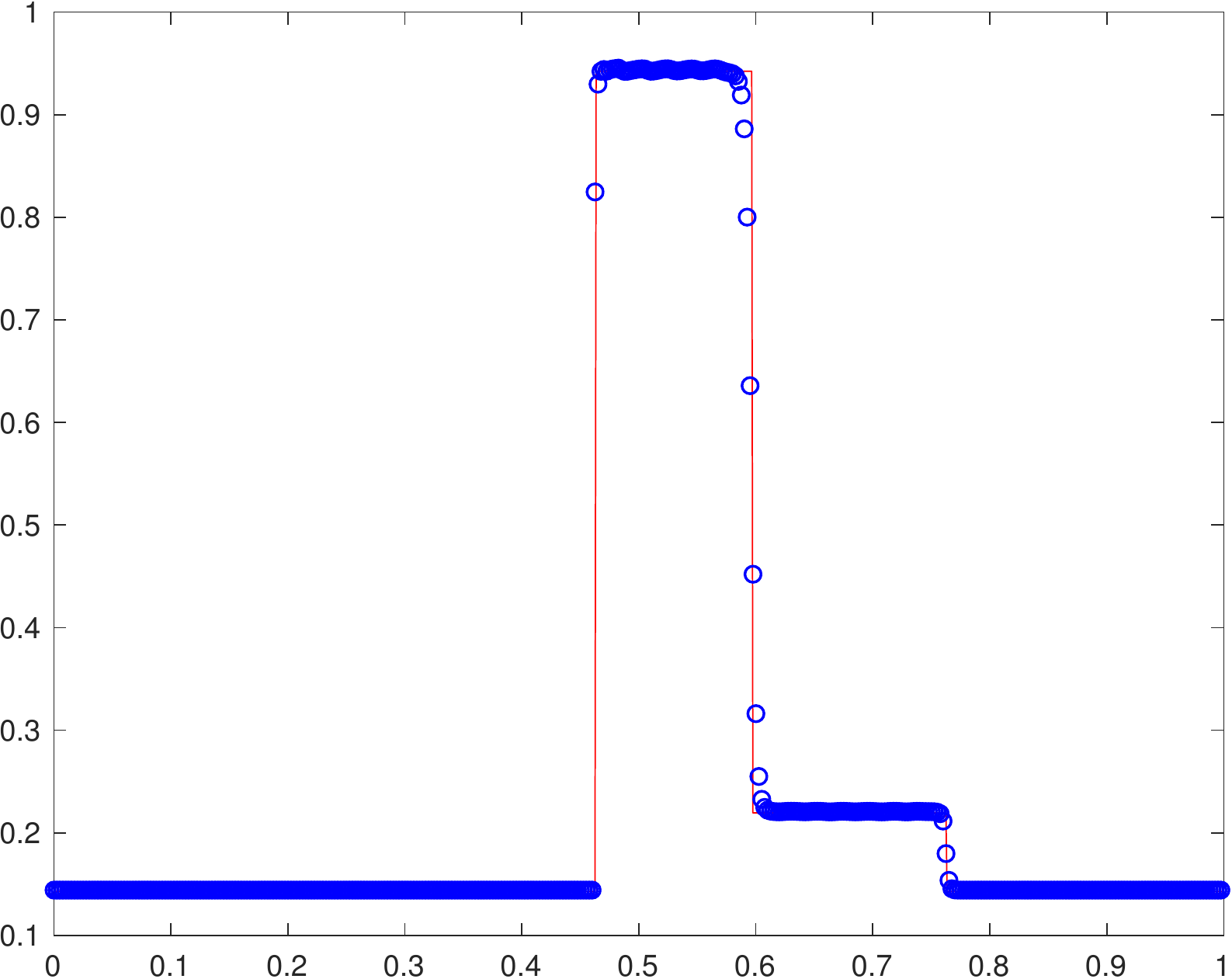}
			\end{minipage}
		}
		\subfigure[$u$]{
			\begin{minipage}[t]{0.3\textwidth}
				\centering
				\includegraphics[width=\textwidth]{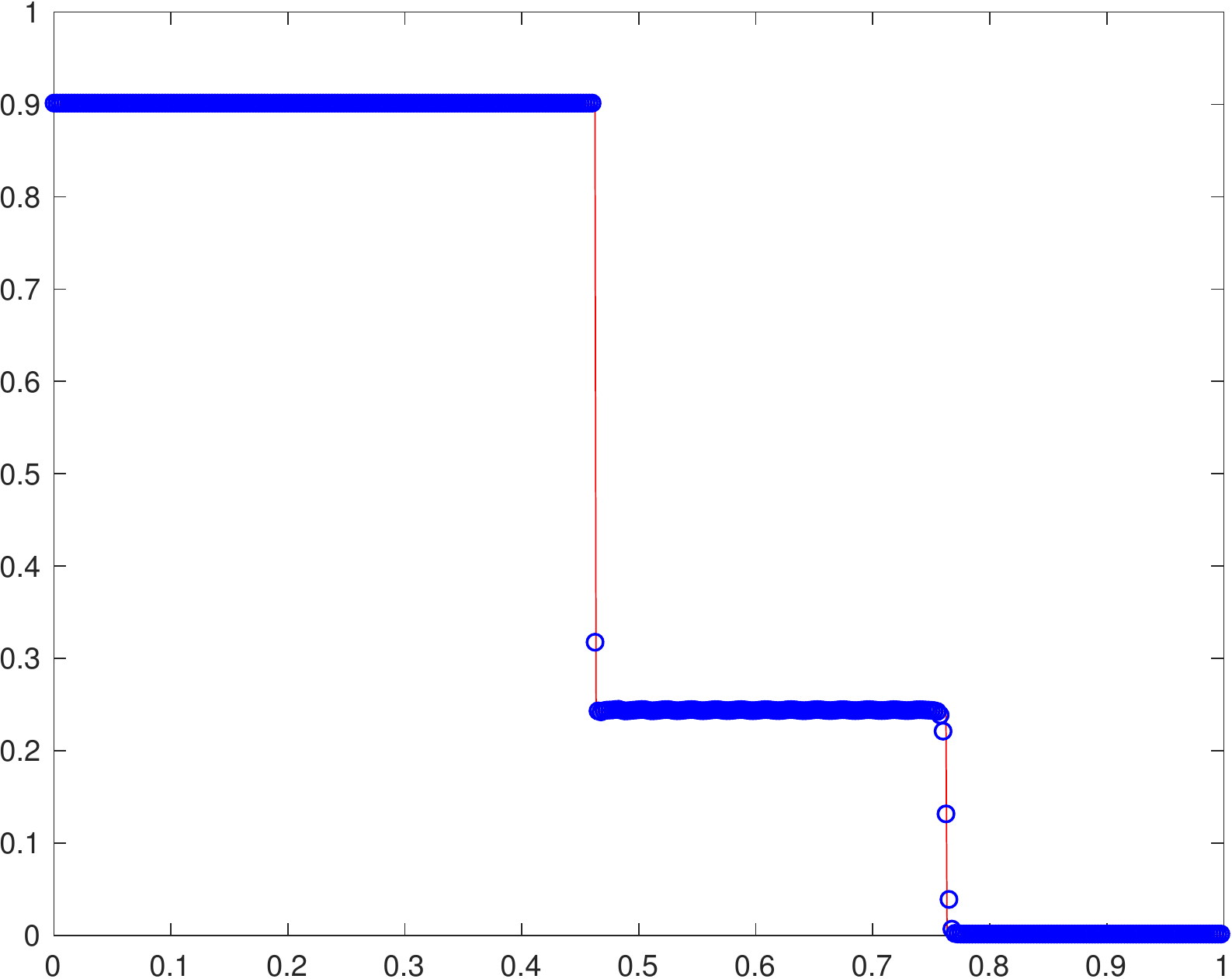}
			\end{minipage}
		}
		\subfigure[$p/20$]{
			\begin{minipage}[t]{0.3\textwidth}
				\centering
				\includegraphics[width=\textwidth]{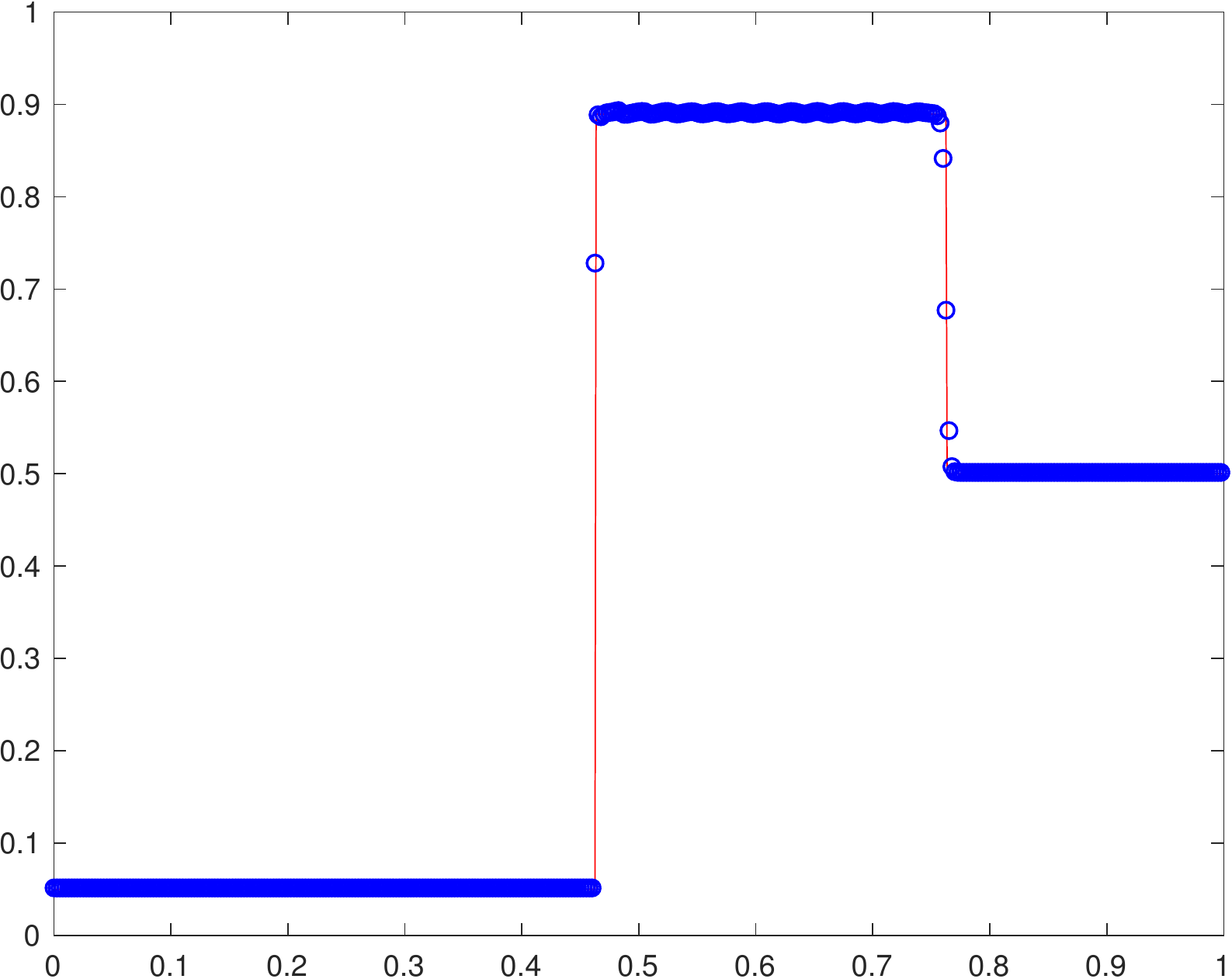}
			\end{minipage}
		}
		\caption{\small{{\tt RP3} in Example \ref{example3.2}: The solutions at $t = 0.4$. }}
		\label{fig:004}
	\end{figure}

	\begin{figure}[htbp]
		\setlength{\abovecaptionskip}{0.cm}
		\setlength{\belowcaptionskip}{-0.cm}
		\subfigure[$\rho$]{
			\begin{minipage}[t]{0.3\textwidth}
				\centering
				\includegraphics[width=\textwidth]{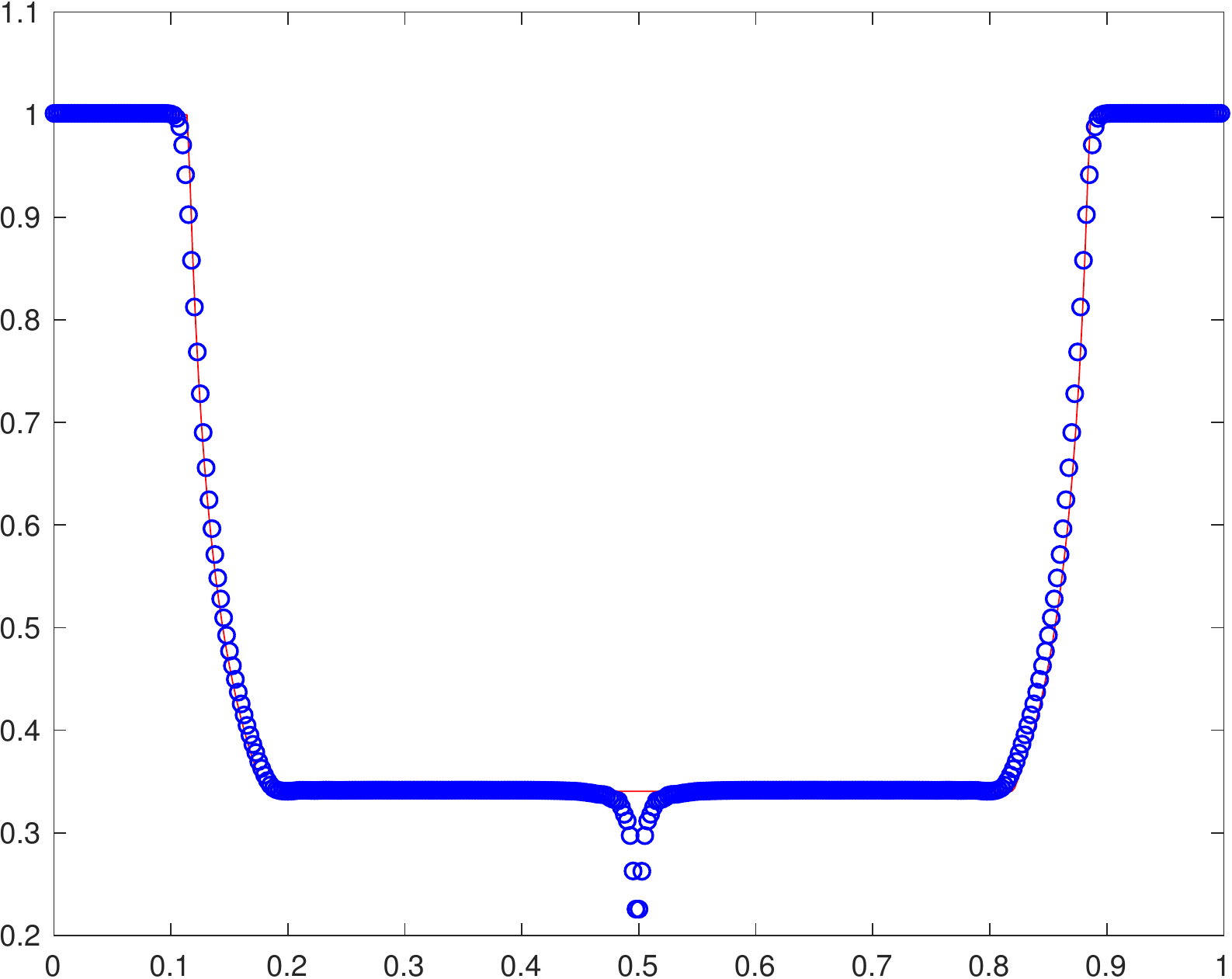}
			\end{minipage}
		}
		\subfigure[$u$]{
			\begin{minipage}[t]{0.3\textwidth}
				\centering
				\includegraphics[width=\textwidth]{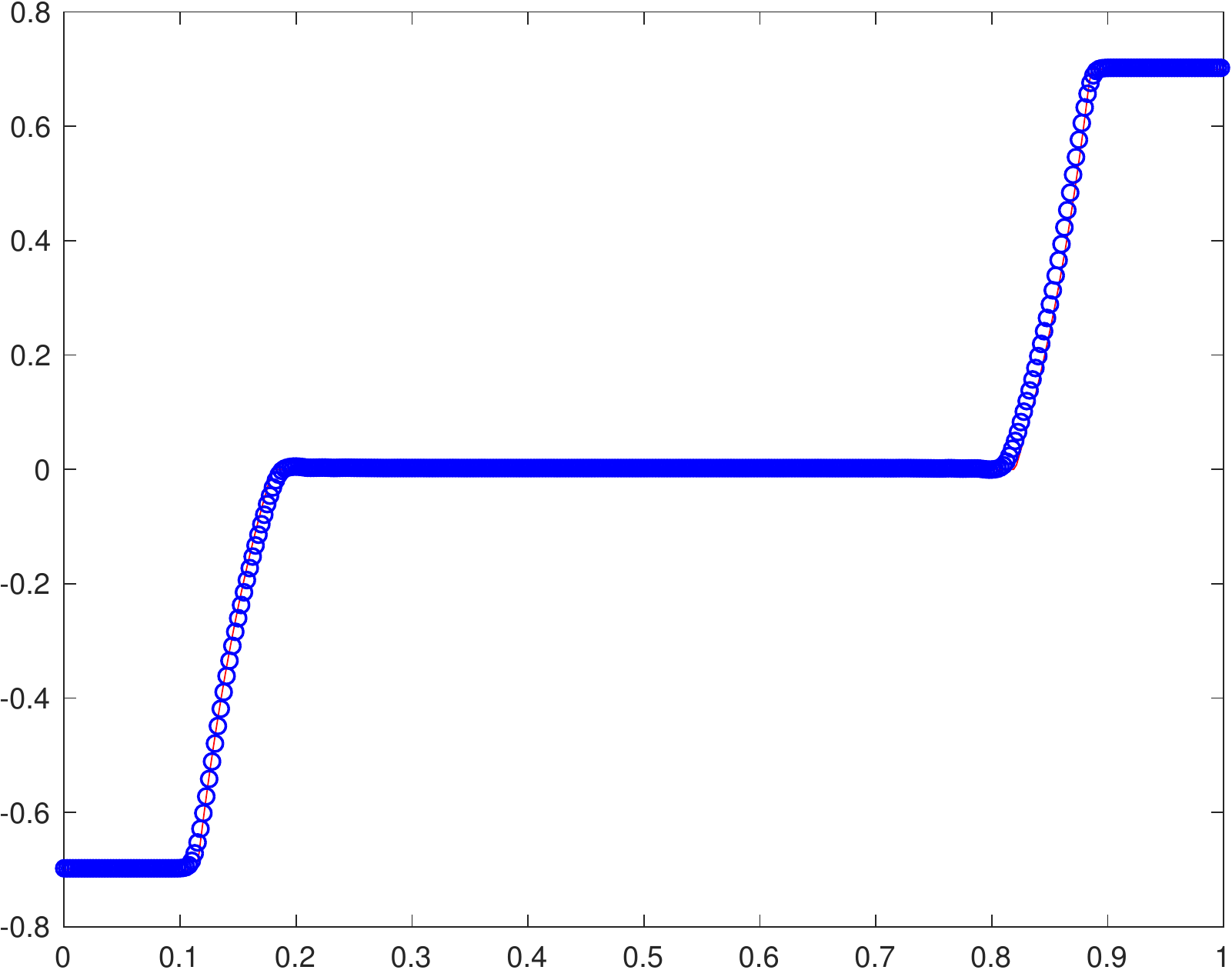}
			\end{minipage}
		}
		\subfigure[$p/20$]{
			\begin{minipage}[t]{0.3\textwidth}
				\centering
				\includegraphics[width=\textwidth]{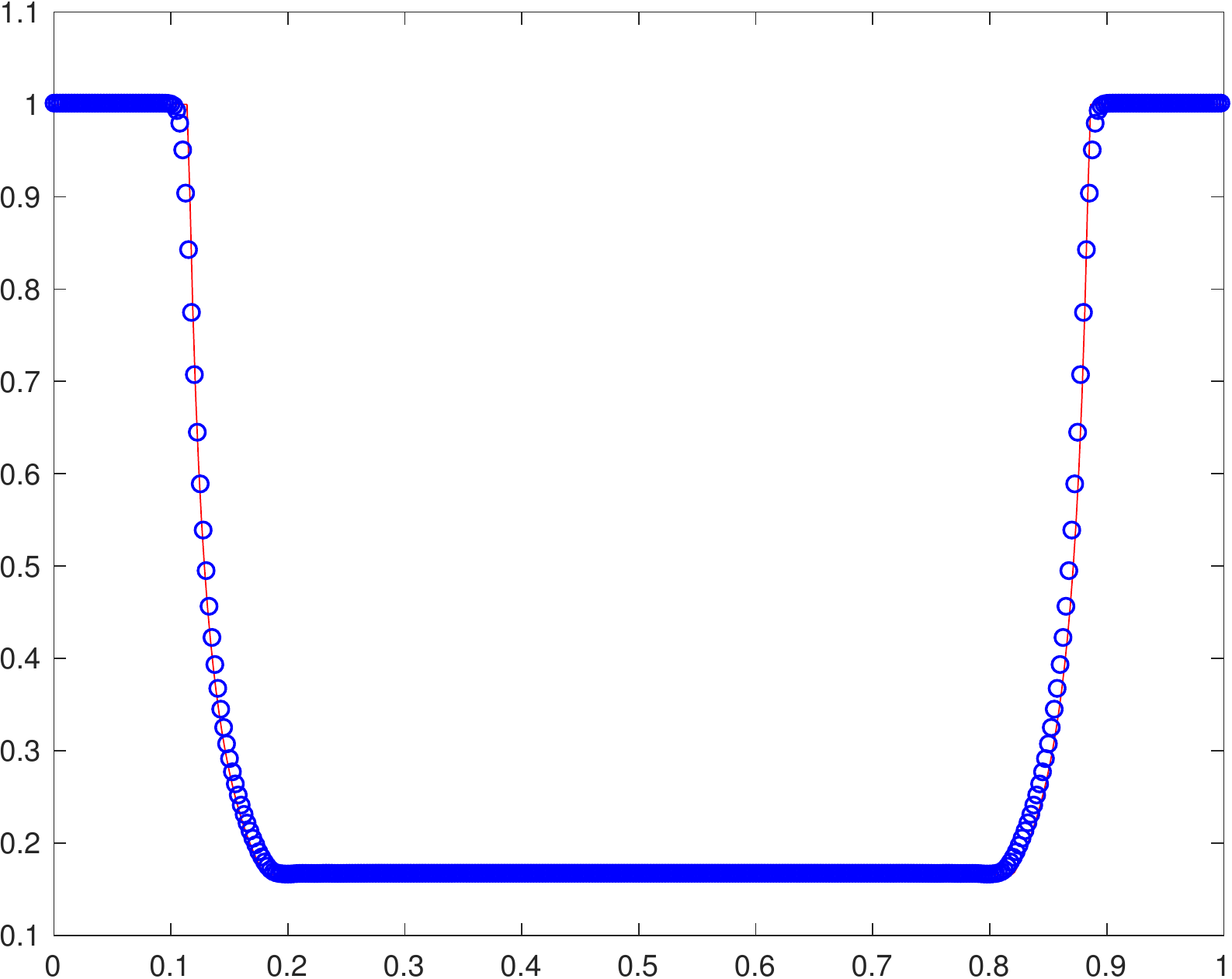}
			\end{minipage}
		}
		\caption{\small{{\tt RP4} in Example \ref{example3.2}: The solutions at $t = 0.4$. }}\label{fig:005}
	\end{figure}
\end{example}

\begin{example}[Density perturbation problem]\label{example3.3}\rm
	This is a more general Cauchy problem obtained
	by including a density perturbation in the initial data of   corresponding Riemann
	problem   in order to test the ability of shock-capturing schemes to resolve small scale flow features, which may give a good indication of the numerical (artificial) viscosity of the scheme.
	The initial data are given by
	\begin{equation*}
	(\rho ,  u , p)(x,0)   =
	\begin{cases}
	(5 ,\, 0 ,\, 50 ) , & x < 0.5, \\
	( 2+0.3\sin(50x) ,\, 0 ,\,5) , & x>0.5.
	\end{cases}
	\end{equation*}
	The computational domain is taken as $[0,\,1]$ with the out-flow boundary conditions.
	Fig. \ref{fig:006} shows the solutions at $t = 0.35$ with 400 uniform cells and $\G = 5/3$, where the reference solution ({``solid line"}) are obtained with 2000 uniform cells. It can be seen that our scheme resolves the high frequency waves   better
	than the third order GRP scheme \cite{Wu-Yang-Tang:2014}.

	\begin{figure}[htbp]
		\setlength{\abovecaptionskip}{0.cm}
		\setlength{\belowcaptionskip}{-0.cm}
		\subfigure[$\rho$]{
			\begin{minipage}[t]{0.3\textwidth}
				\centering
				\includegraphics[width=\textwidth]{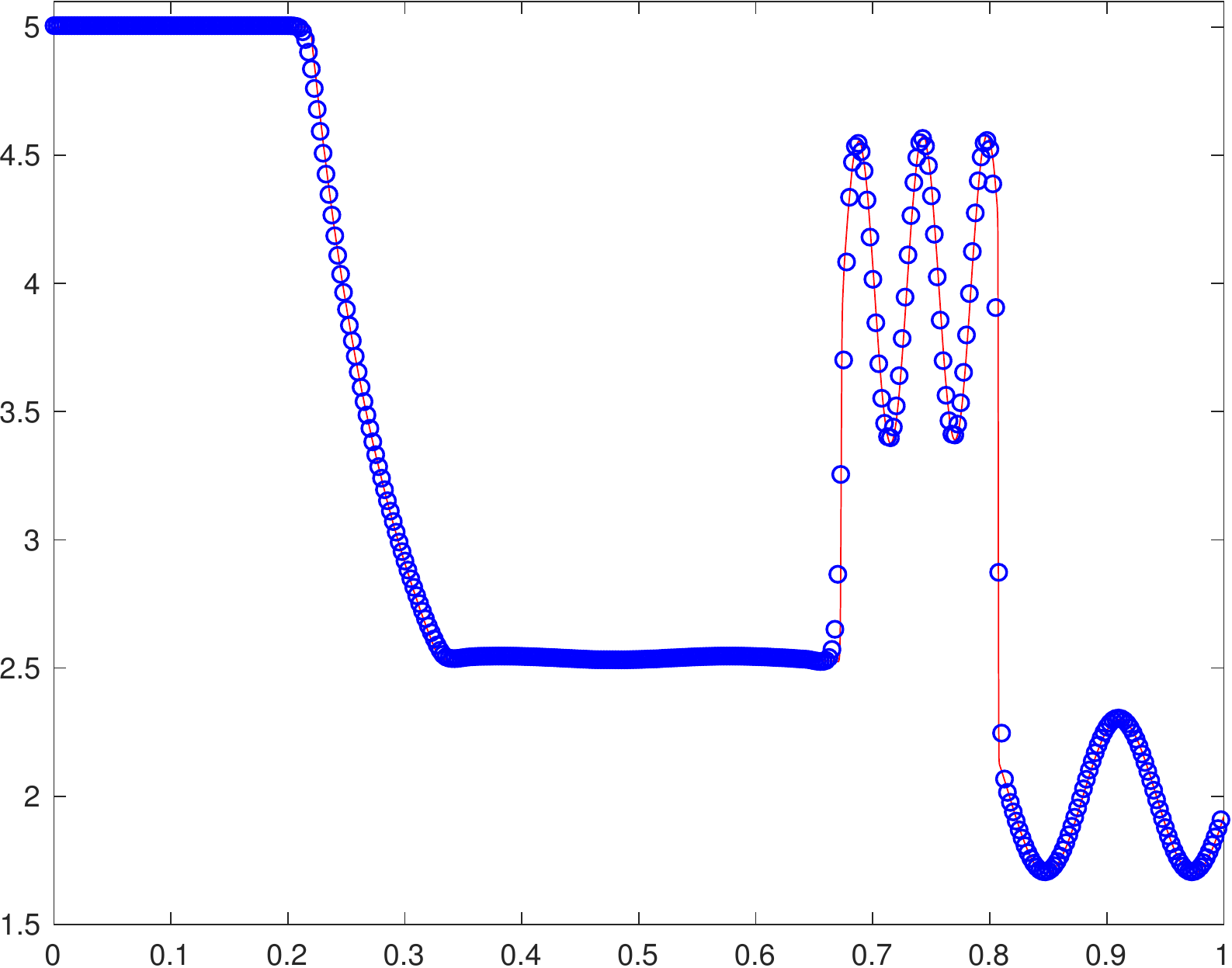}
			\end{minipage}
		}
		\subfigure[$u$]{
			\begin{minipage}[t]{0.3\textwidth}
				\centering
				\includegraphics[width=\textwidth]{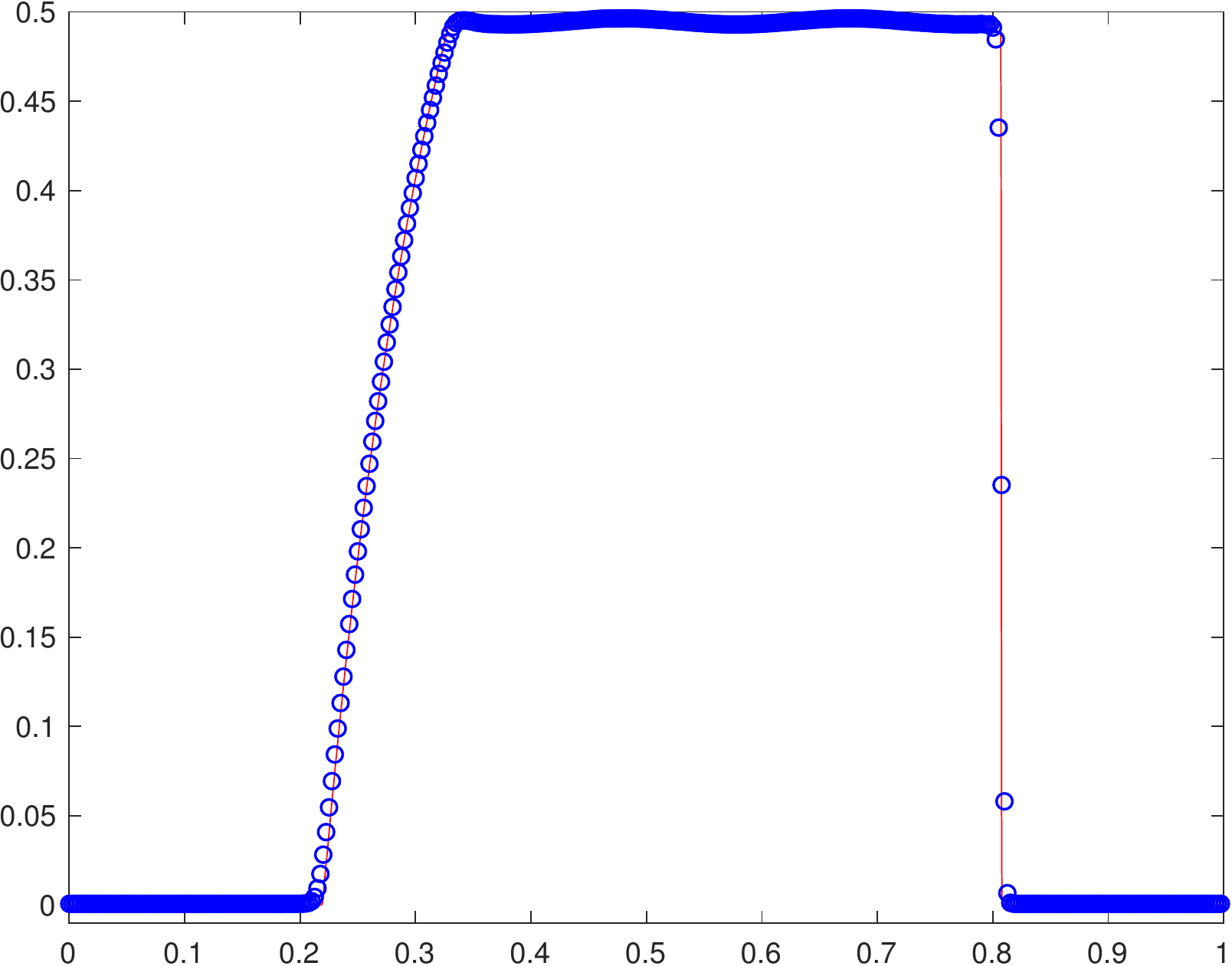}
			\end{minipage}
		}
		\subfigure[$p$]{
			\begin{minipage}[t]{0.3\textwidth}
				\centering
				\includegraphics[width=\textwidth]{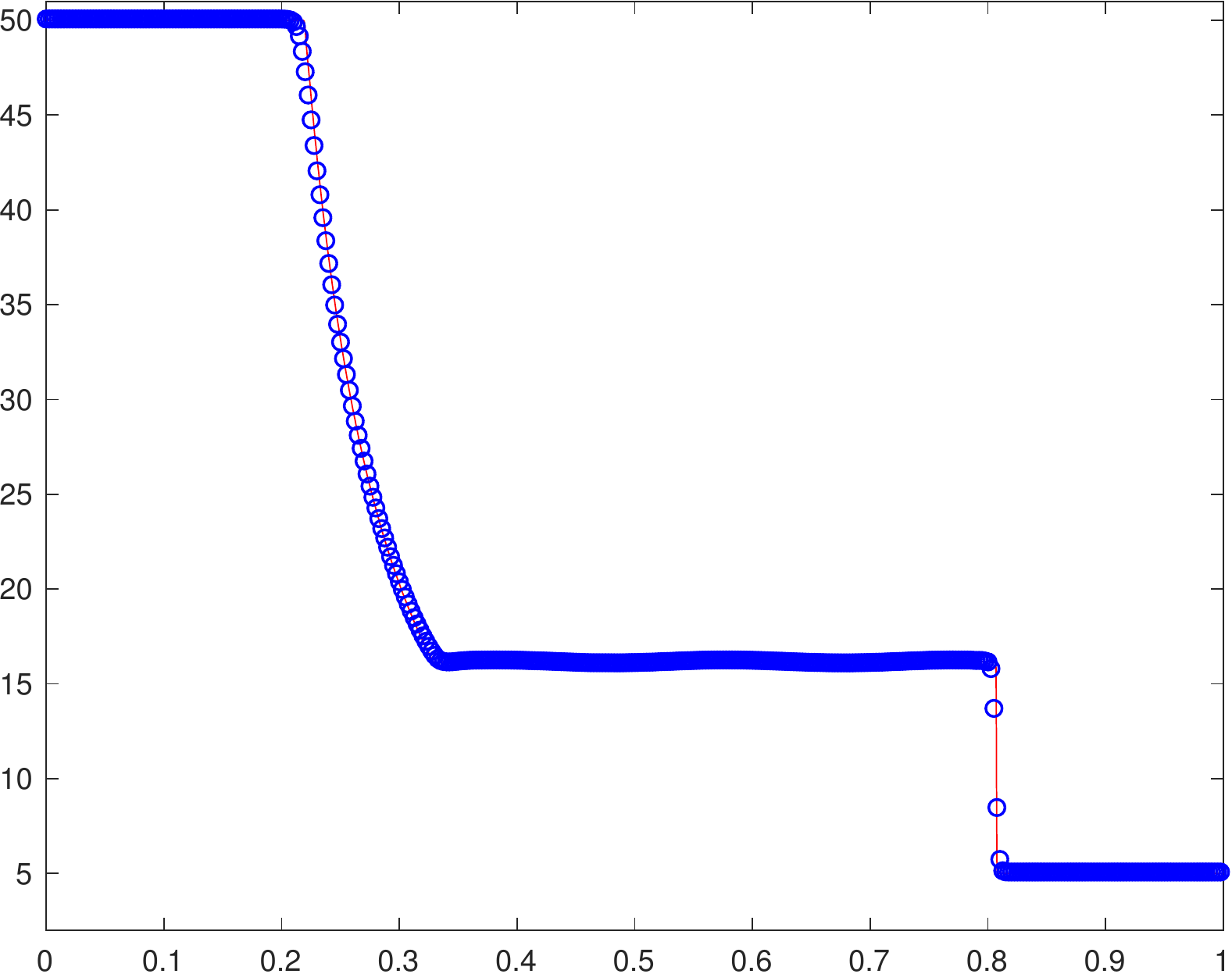}
			\end{minipage}
		}
		\caption{\small{Example \ref{example3.3}: The solutions at $t = 0.35$.}}
		\label{fig:006}
	\end{figure}
\end{example}

\begin{example}[Collision of two blast waves]\label{example3.4}\rm
	The last 1D example simulates the collision
	of two strong relativistic blast waves. The initial data for this initial-boundary value
	problems  consist  of three constant states of an ideal gas with $\G = 1.4$, at
	rest in the domain [0,1] with outflow boundary conditions at $x = 0$ and 1.
	The initial data are given {by}
	\begin{equation*}
	(\rho , u , p)(x,0) =
	\begin{cases}
	(1 ,\, 0 ,\, 10^3 ) , & 0 \leq x < 0.1, \\
	(1 ,\, 0 ,\, 10^{-2} ) , & 0.1 \leq x < 0.9, \\
	(1 ,\, 0 ,\,10^2) , & 0.9 \leq x < 1.0.
	\end{cases}
	\end{equation*}
	Two strong blast waves develop and collide, producing a
	new contact discontinuity.
	Figs. \ref{fig:007}$-$\ref{fig:007b} show
	the close-up of solutions at $t = 0.43$ with 4000 uniform cells
	and   different $\alpha$, where the exact solution ( ``solid line") are obtained by the exact RP solver with 4000 uniform cells. 	
	It is seen that our scheme  can well resolve those strong discontinuities, and
	clearly capture the relativistic wave configurations generated by the collision of the two strong relativistic blast waves.
	
	\begin{figure}[htbp]
		\setlength{\abovecaptionskip}{0.cm}
		\setlength{\belowcaptionskip}{-0.cm}
		\subfigure[$\rho$]{
			\begin{minipage}[t]{0.3\textwidth}
				\centering
				\includegraphics[width=\textwidth]{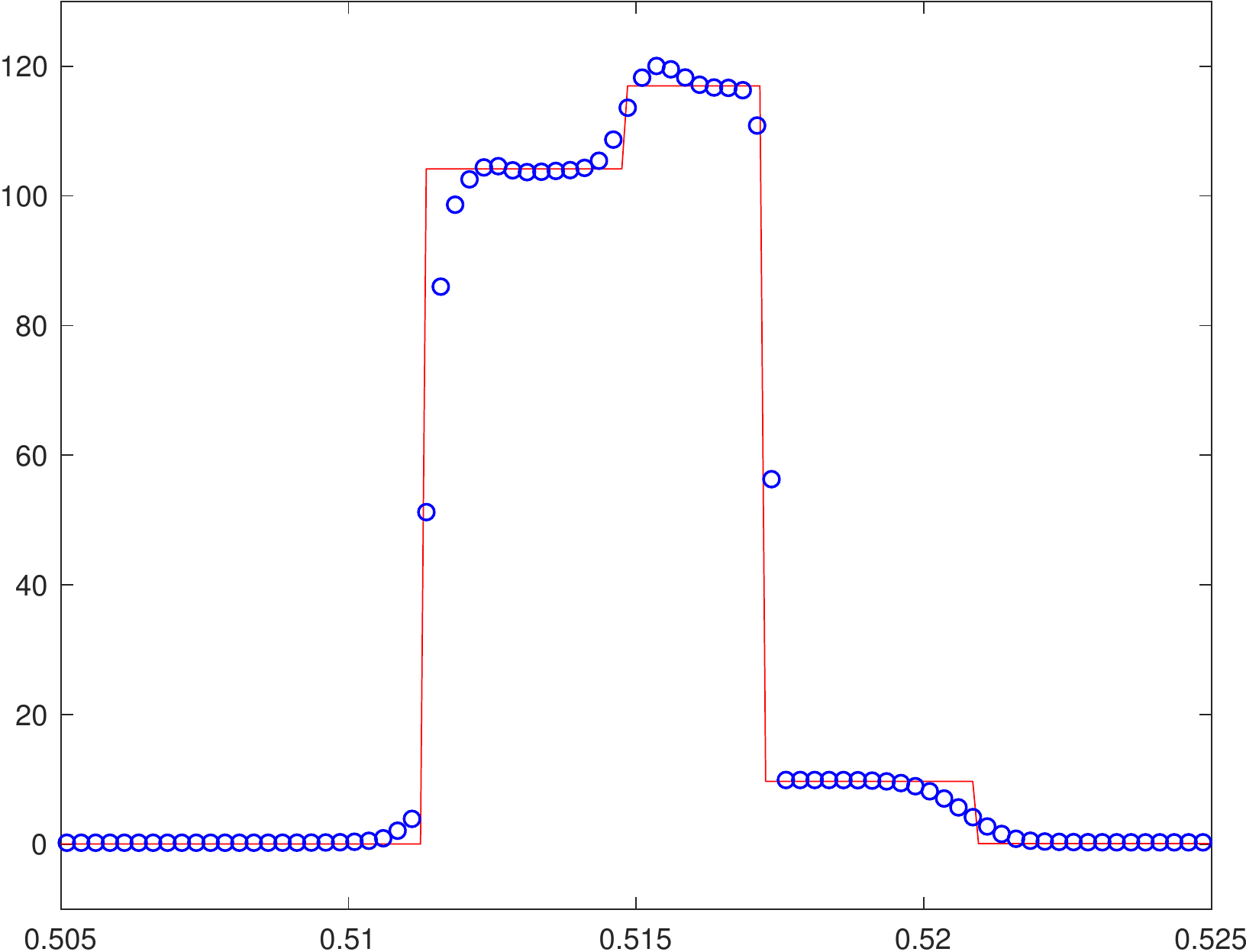}
			\end{minipage}
		}
		\subfigure[$u$]{
			\begin{minipage}[t]{0.3\textwidth}
				\centering
				\includegraphics[width=\textwidth]{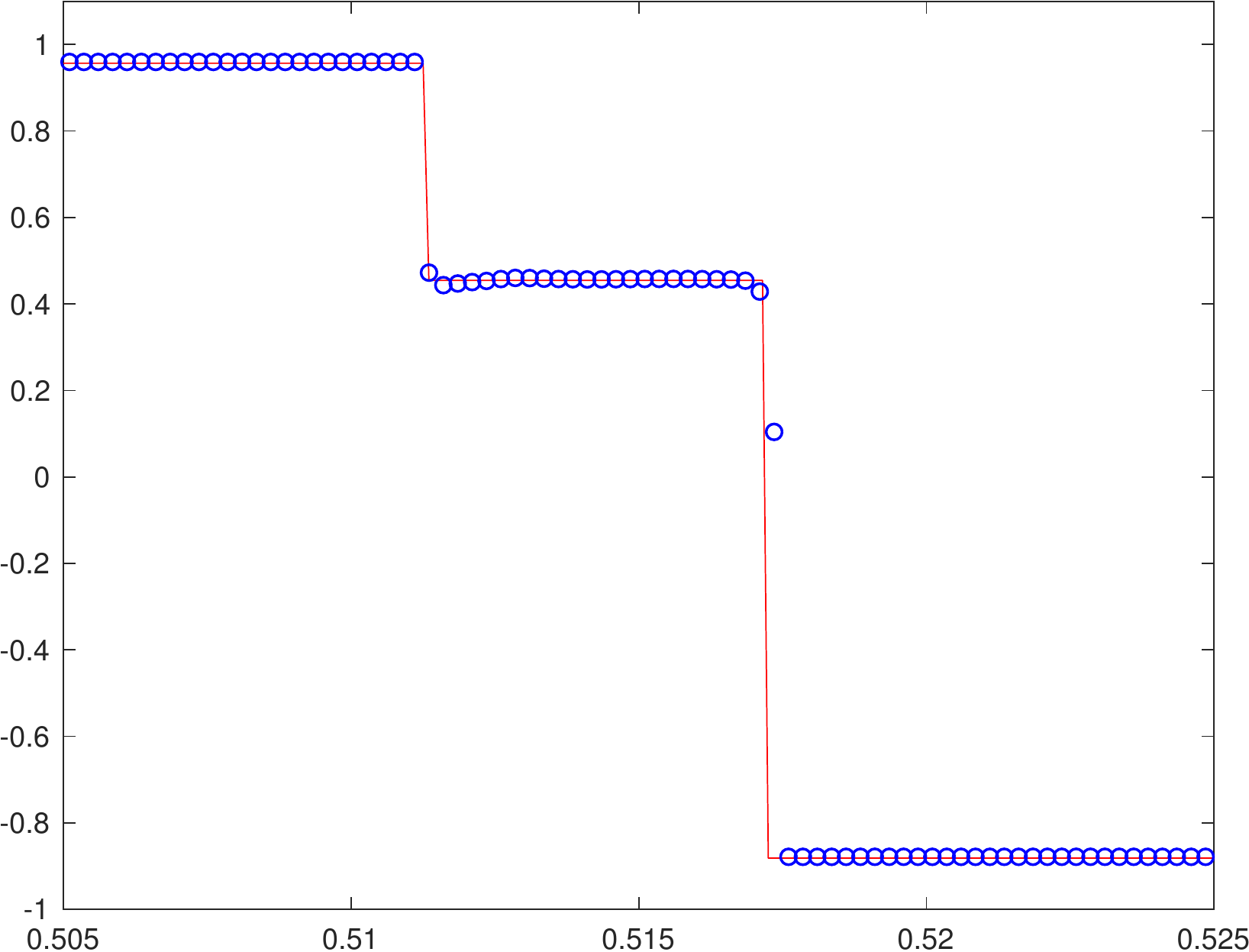}
			\end{minipage}
		}
		\subfigure[$p$]{
			\begin{minipage}[t]{0.3\textwidth}
				\centering
				\includegraphics[width=\textwidth]{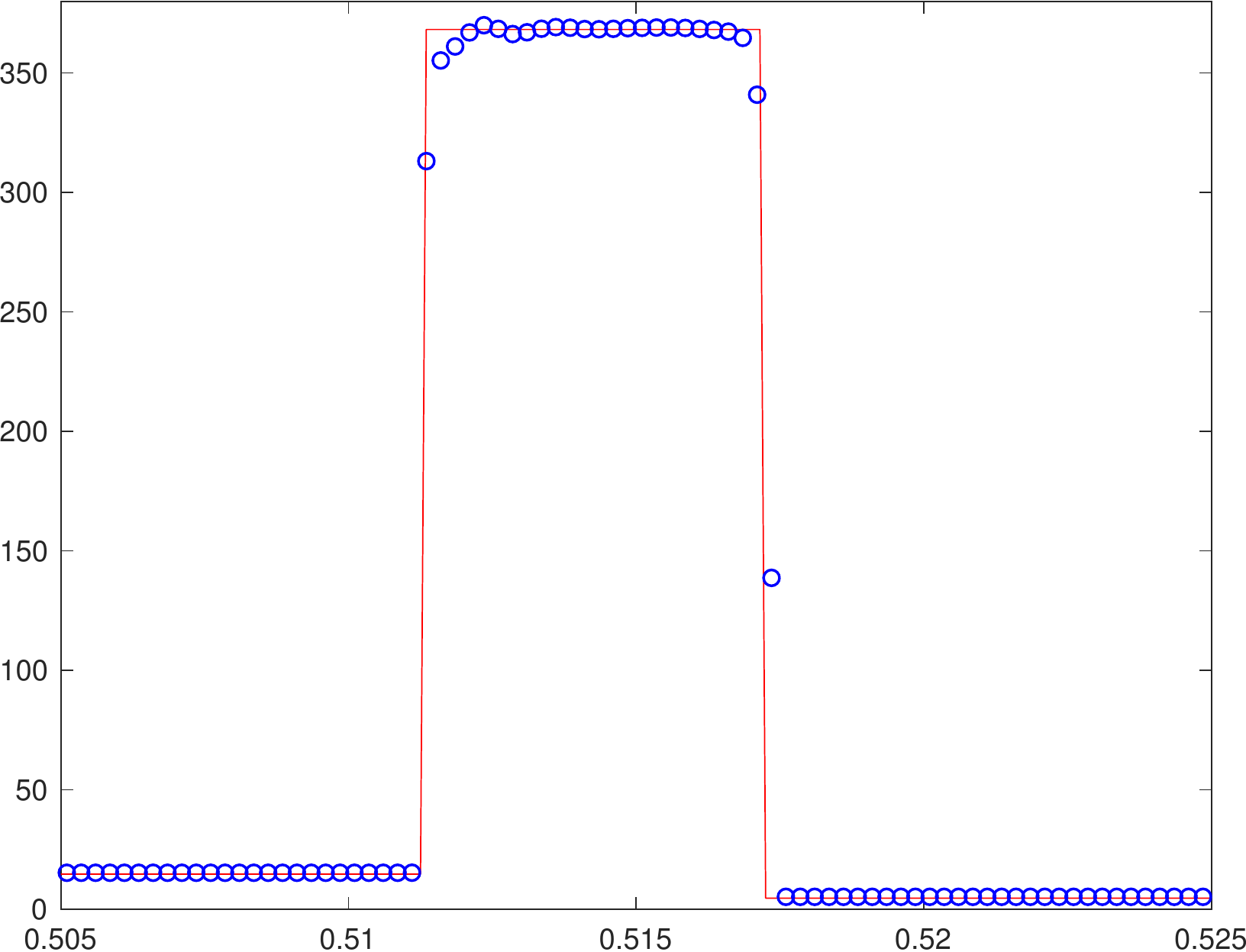}
			\end{minipage}
		}
		\caption{\small{Example \ref{example3.4}: Close-up of  the solutions  at $t = 0.43$. $\alpha=\frac13$.}}\label{fig:007}
	\end{figure}
	\begin{figure}[htbp]
		\setlength{\abovecaptionskip}{0.cm}
		\setlength{\belowcaptionskip}{-0.cm}
		\subfigure[$\rho$]{
			\begin{minipage}[t]{0.3\textwidth}
				\centering
				\includegraphics[width=\textwidth]{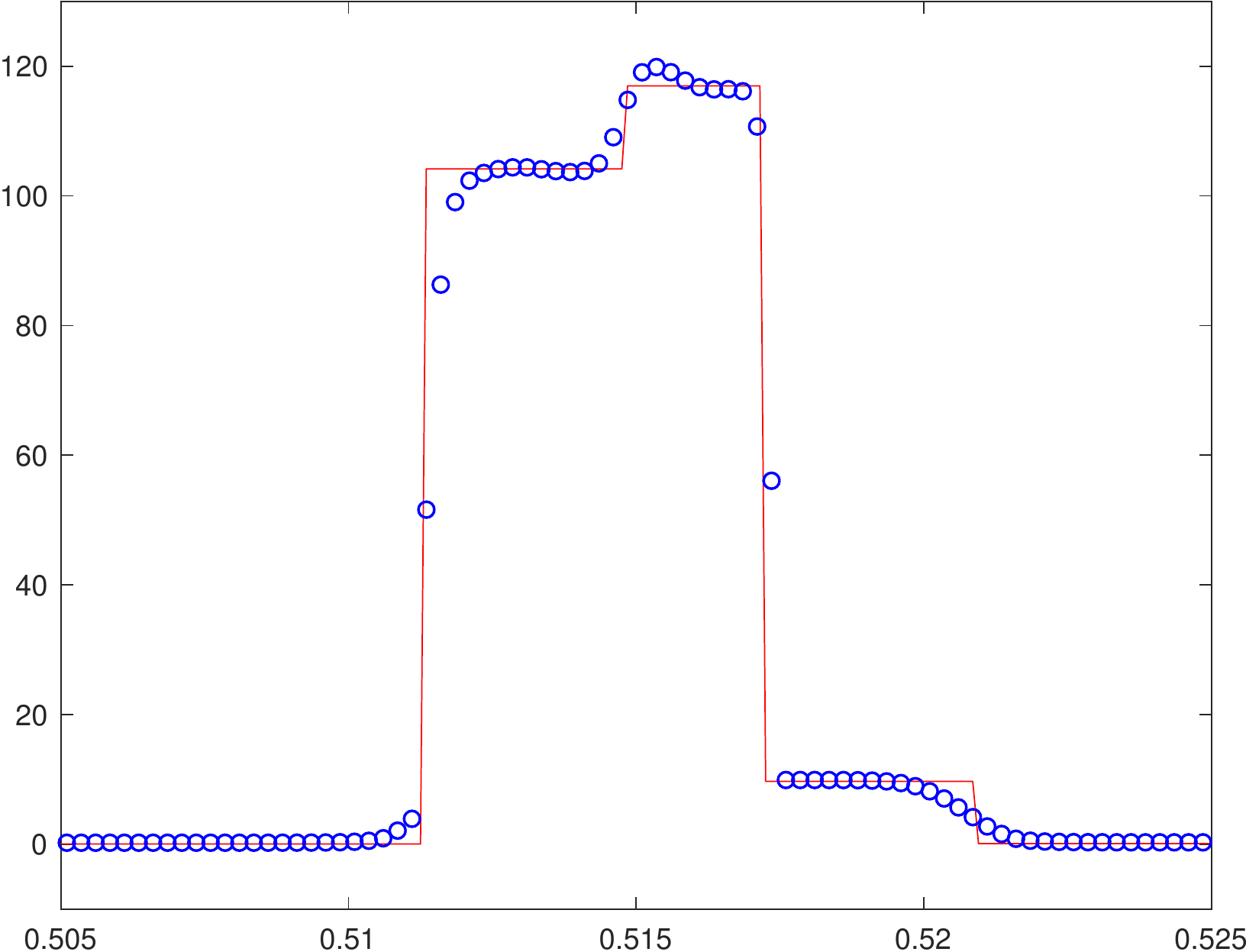}
			\end{minipage}
		}
		\subfigure[$u$]{
			\begin{minipage}[t]{0.3\textwidth}
				\centering
				\includegraphics[width=\textwidth]{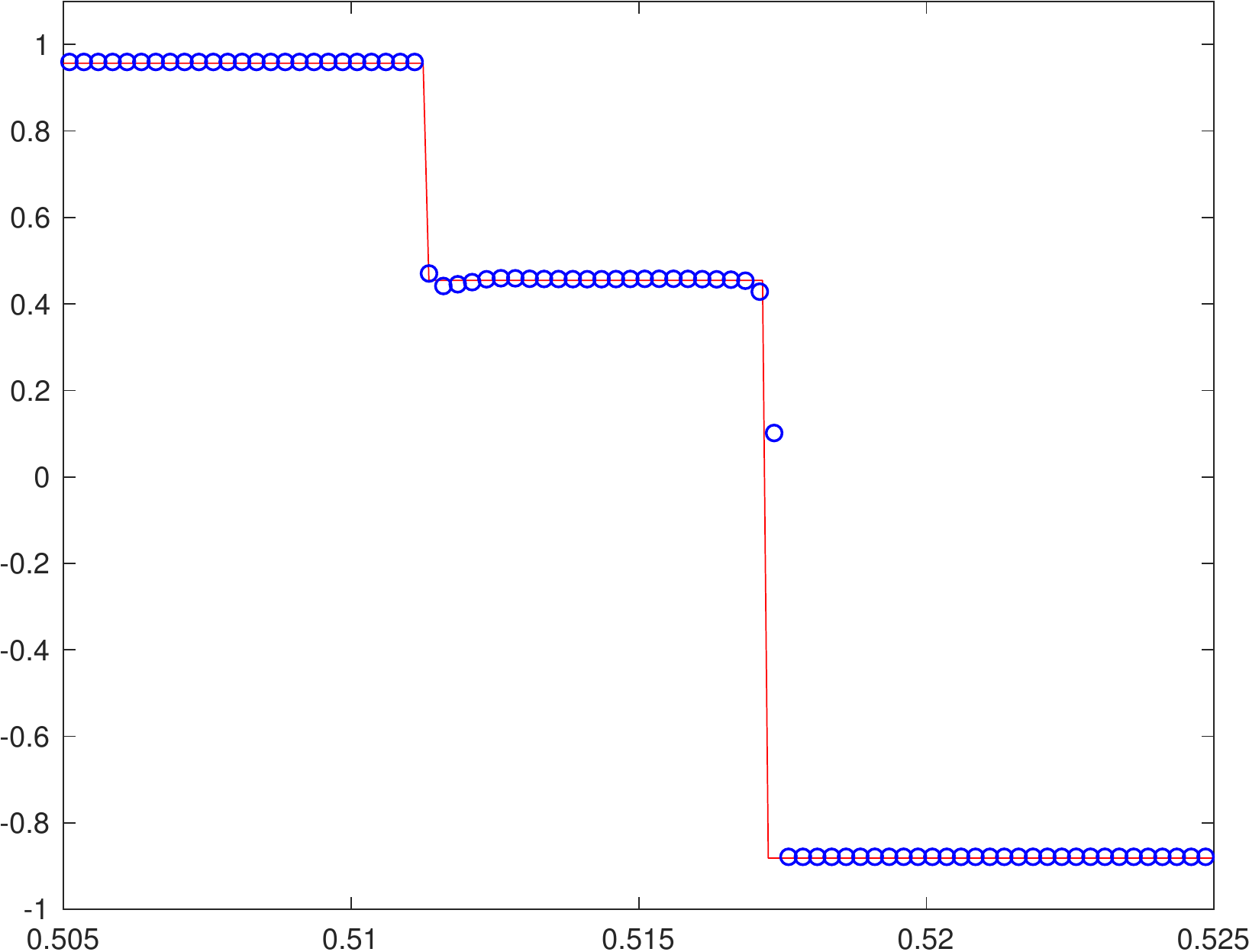}
			\end{minipage}
		}
		\subfigure[$p$]{
			\begin{minipage}[t]{0.3\textwidth}
				\centering
				\includegraphics[width=\textwidth]{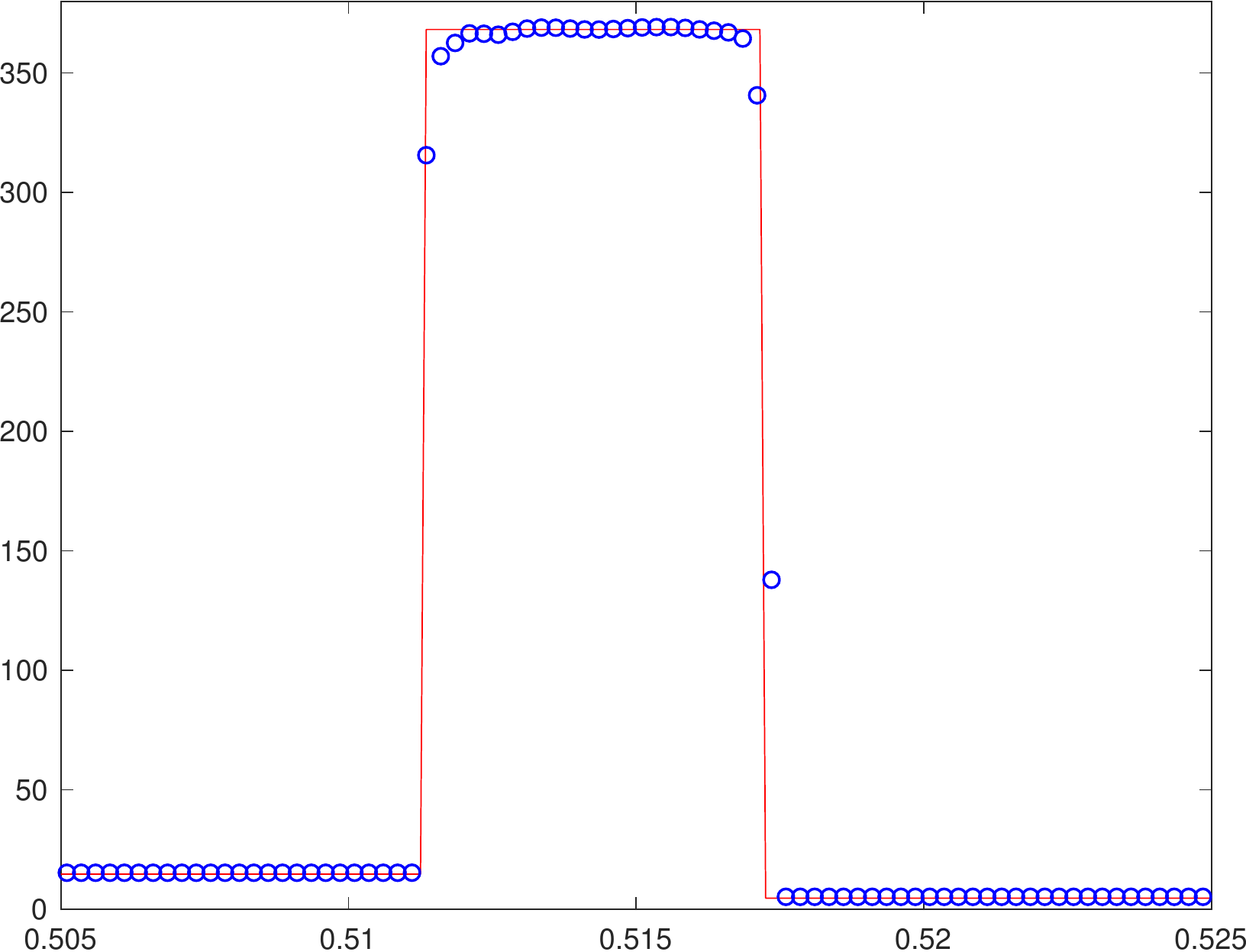}
			\end{minipage}
		}
		\caption{\small{Same as Fig. \ref{fig:007} except for $\alpha=\frac{1-6\tau}{3-6\tau}$.}}\label{fig:007a}
	\end{figure}
	\begin{figure}[htbp]
		\setlength{\abovecaptionskip}{0.cm}
		\setlength{\belowcaptionskip}{-0.cm}
		\subfigure[$\rho$]{
			\begin{minipage}[t]{0.3\textwidth}
				\centering
				\includegraphics[width=\textwidth]{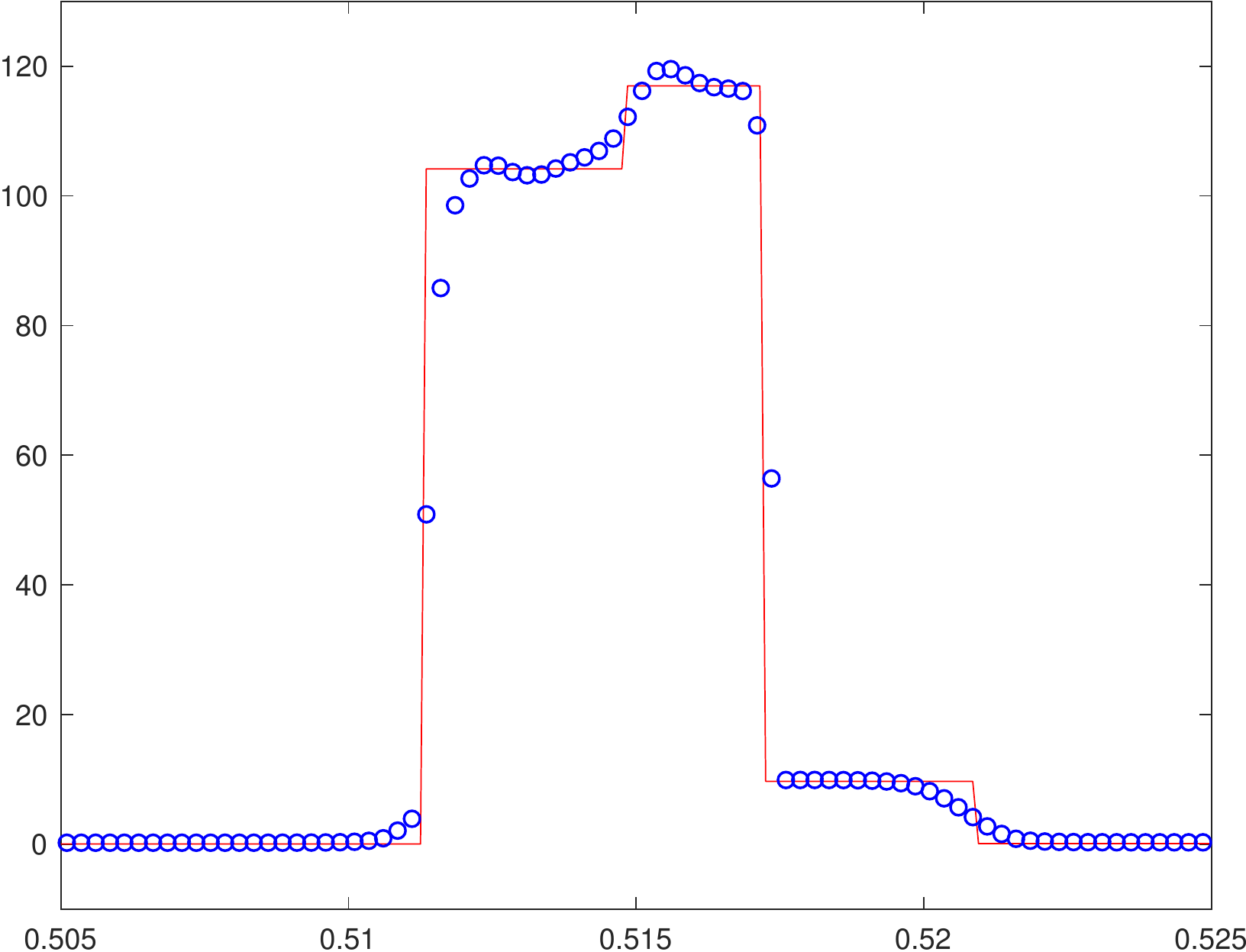}
			\end{minipage}
		}
		\subfigure[$u$]{
			\begin{minipage}[t]{0.3\textwidth}
				\centering
				\includegraphics[width=\textwidth]{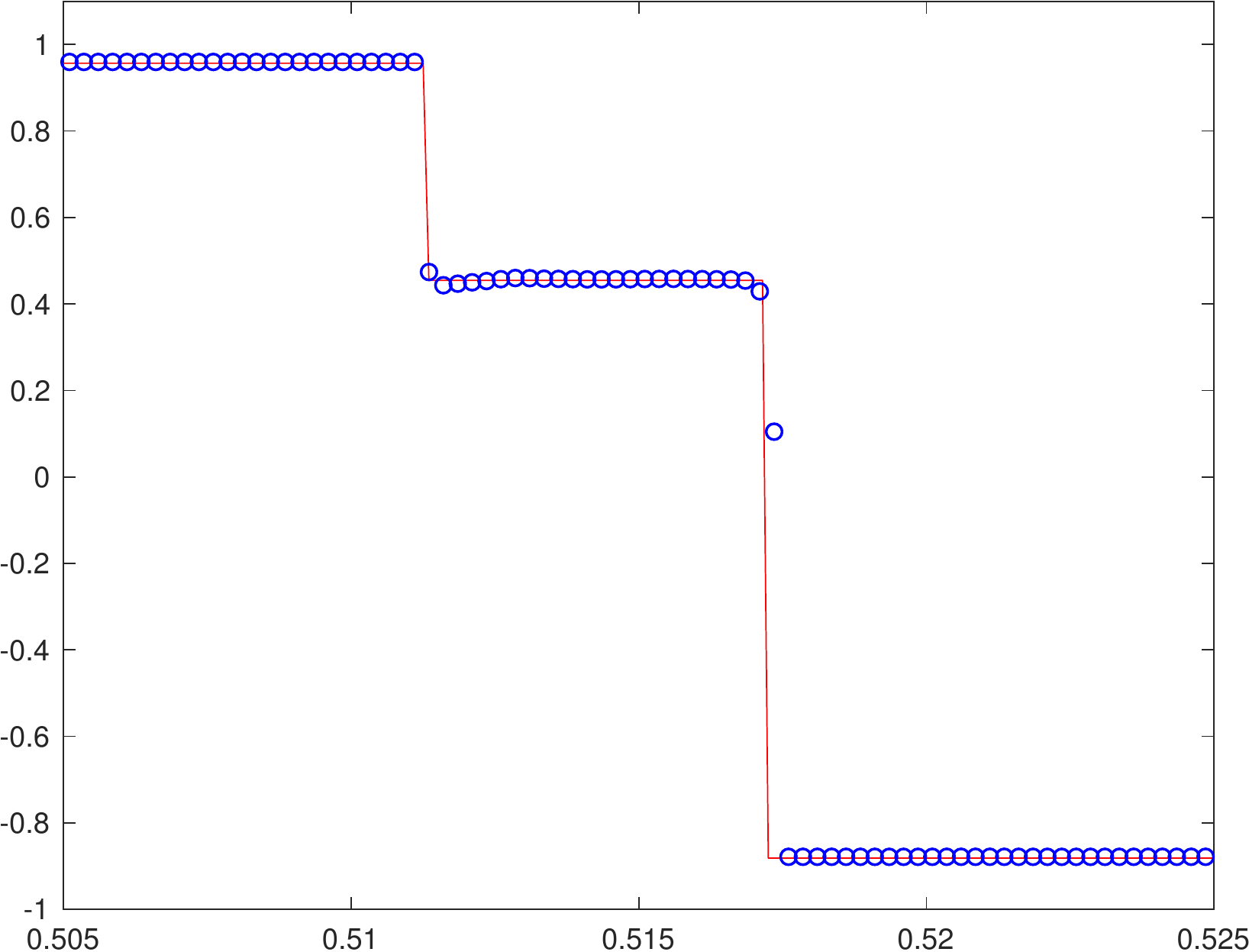}
			\end{minipage}
		}
		\subfigure[$p$]{
			\begin{minipage}[t]{0.3\textwidth}
				\centering
				\includegraphics[width=\textwidth]{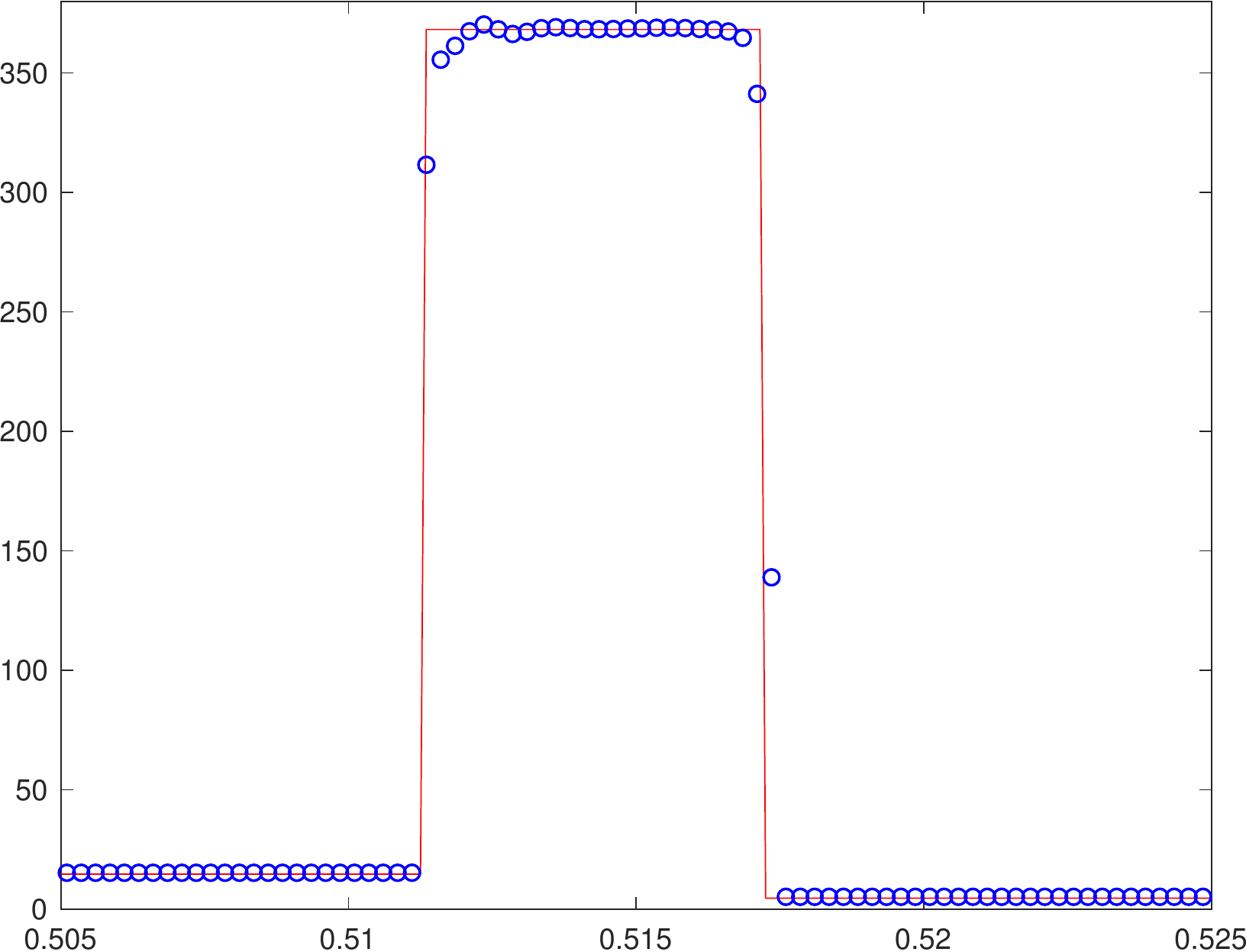}
			\end{minipage}
		}
		\caption{\small{Same as Fig. \ref{fig:007} except for $\alpha=\frac13 + \tau$. }}\label{fig:007b}
	\end{figure}
\end{example}

\subsection{Two-dimensional case}
Unless otherwise stated,  the adiabatic index $\G$ is taken as $5/3$
and the  parameter $\epsilon$ in the adaptive switch procedure is specified as $0.05$,
that is to say, $P_{\mbox{sw}}=1.05$.

\begin{example}[Smooth problem]\label{example3.5}\rm
	The problem considered here describes a RHD sine wave propagating periodically in the domain $\Om = [0,2/ \sqrt{3}]\times [0,2]$   at an angle of $\al = 30^{\circ}$ with the $x$-axis. The initial data are taken as
	\begin{equation*}
	\begin{cases}
	\rho(x,y,0) = 1+0.2\sin(2\pi (x\cos \al + y\sin \al)), \\
	u(x,y,0) = 0.2, \quad v(x,y,0) = 0.2,  \quad p(x,y,0) = 1.
	\end{cases}
	\end{equation*}
	The exact solution can be given {by}
	\begin{equation*}
	\begin{cases}
	\rho(x,y,t) = 1+0.2\sin(2\pi ((x-ut)\cos \al + (y-vt)\sin \al)), \\
	u(x,y,t) = 0.2, \quad v(x,y,t) = 0.2, \quad p(x,y,t) = 1.
	\end{cases}
	\end{equation*}
	{In our computations, $\tau = \frac{\mu}{
			\max_{\ell,j,k}\{   |\la_\ell^1 \left( \overline{\vec U}_{jk}^n \right) | \} /{\Delta x^{5/4}}
			+ \max_{\ell,j,k}\{   | \la_\ell^2\left( \overline{\vec U}_{jk}^n \right) | \} /{\Delta y^{5/4}} }$. }
	Tables \ref{table:003}$-$\ref{table:003b} list the errors and convergence
	rates in $\rho$ at $t = 2$ obtained by using our 2D scheme with
	$N\times N$ uniform cells and different $\alpha$.
	The results show that {our 2D two-stage schemes can have the theoretical orders}.

	\begin{table}[htpb]
		\centering
		\caption{ The errors and convergence rates for solution at $t = 2$.   $\alpha=\frac13$. } \label{table:003}	
		\begin{tabular}{ccccccc}
			\hline
			$N$ & $l^1$ error & order & $l^2$ error & order & $l^{\infty}$ error & order \\
			5    &8.8639e-02   & -    &6.4567e-02   & -    &6.1366e-02   & -\\
			10   &7.3103e-03   & 3.6000    &5.2506e-03   & 3.6202    &4.8578e-03   & 3.6591    \\
			20   &3.4830e-04   & 4.3915    &2.6427e-04   & 4.3124    &2.7917e-04   & 4.1211    \\
			40   &1.0722e-05   & 5.0217    &8.2909e-06   & 4.9943    &9.4647e-06   & 4.8824    \\
			80   &3.3578e-07   & 4.9969    &2.5428e-07   & 5.0271    &2.9210e-07   & 5.0180    \\
			160    &1.0576e-08   & 4.9887    &7.8638e-09   & 5.0150    &9.2428e-09   & 4.9820 \\
			\hline
		\end{tabular}
	\end{table}
	
	\begin{table}[htpb]
		\centering
		\caption{{Same as Table \ref{table:003} except for}
    $\alpha=\frac{1-6\tau}{3-6\tau}$.} \label{table:003a}	
		\begin{tabular}{ccccccc}
			\hline
			$N$ & $l^1$ error & order & $l^2$ error & order & $l^{\infty}$ error & order \\
			5    &8.8544e-02   & -    &6.4503e-02   & -    &6.1296e-02   & -\\
			10   &7.3094e-03   & 3.5986    &5.2467e-03   & 3.6199    &4.8472e-03   & 3.6606    \\
			20   &3.4800e-04   & 4.3926    &2.6414e-04   & 4.3120    &2.7981e-04   & 4.1146    \\
			40   &1.0709e-05   & 5.0222    &8.2864e-06   & 4.9944    &9.5006e-06   & 4.8803    \\
			80   &3.3590e-07   & 4.9946    &2.5419e-07   & 5.0268    &2.9340e-07   & 5.0171    \\
			160    &1.0589e-08   & 4.9875    &7.8733e-09   & 5.0128    &9.3046e-09   & 4.9788 \\
			\hline
		\end{tabular}
	\end{table}
	
	\begin{table}[htpb]
		\centering
		\caption{{Same as Table \ref{table:003} except for}  $\alpha=\frac13 + \tau$. } \label{table:003b}	
		\begin{tabular}{ccccccc}
			\hline
			$N$ & $l^1$ error & order & $l^2$ error & order & $l^{\infty}$ error & order \\
			5    &8.8719e-02   & -    &6.4621e-02   & -    &6.1424e-02   & -\\
			10   &7.3109e-03   & 3.6011    &5.2537e-03   & 3.6206    &4.8660e-03   & 3.6580    \\
			20   &3.4852e-04   & 4.3907    &2.6437e-04   & 4.3127    &2.7868e-04   & 4.1261    \\
			40   &1.0731e-05   & 5.0214    &8.2950e-06   & 4.9942    &9.4376e-06   & 4.8840    \\
			80   &3.3618e-07   & 4.9964    &2.5446e-07   & 5.0267    &2.9111e-07   & 5.0188    \\
			160    &1.0593e-08   & 4.9881    &7.8775e-09   & 5.0136    &9.1963e-09   & 4.9844 \\
			\hline
		\end{tabular}
	\end{table}
	
\end{example}

\begin{example}[Riemann problems]\label{example3.6}\rm
	This example solves three 2D Riemann problems to verify the capability of the 2D two-stage scheme in capturing the complex 2D relativistic wave configurations.
	The computational domain is taken as $[0,1]\times [0,1]$ and divided into $300 \times 300$ uniform cells.  {The output solutions at $t = 0.4$ will be plotted with $30$ equally spaced contour lines}.

	The initial data of {\tt RP1} are given {by}
	\begin{equation*}
	(\rho ,  u , v, p)(x,0)
	=  \begin{cases}
	(0.5 ,\, 0.5 ,\, -0.5 ,\, 0.5 ) , & x > 0.5,\,y>0.5, \\
	(1 ,\,  0.5,\, 0.5 , \, 5) , & x<0.5,\,y>0.5, \\
	(3 ,\,  -0.5,\, 0.5 , \, 5) , & x<0.5,\,y<0.5, \\
	(1.5 ,\,  -0.5,\, -0.5 , \, 5) , & x>0.5,\,y<0.5.
	\end{cases}
	\end{equation*}
	It describes the interaction of four contact discontinuities (vortex sheets) with the same sign (the negative sign).
	Fig. \ref{fig:RP1} shows the contour of the density   and   pressure logarithms.
	The results show that the four initial
	vortex sheets interact each other to form a spiral with the low density around the center of the domain as time increases,
	which is the typical cavitation phenomenon in gas dynamics.
\end{example}

\begin{figure}[htbp]
	\centering
	\includegraphics[width=0.36\textwidth]{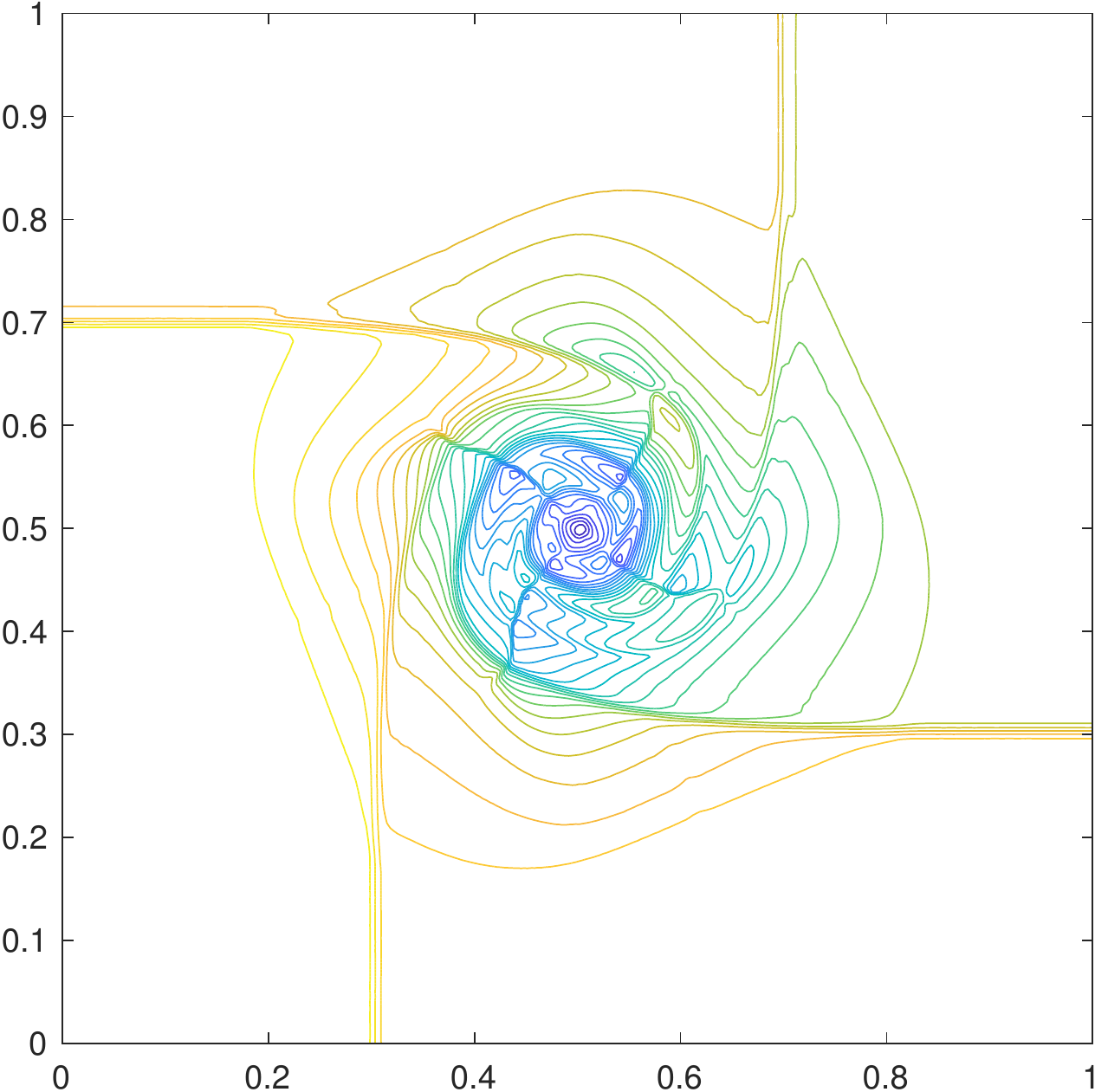}
	\includegraphics[width=0.36\textwidth]{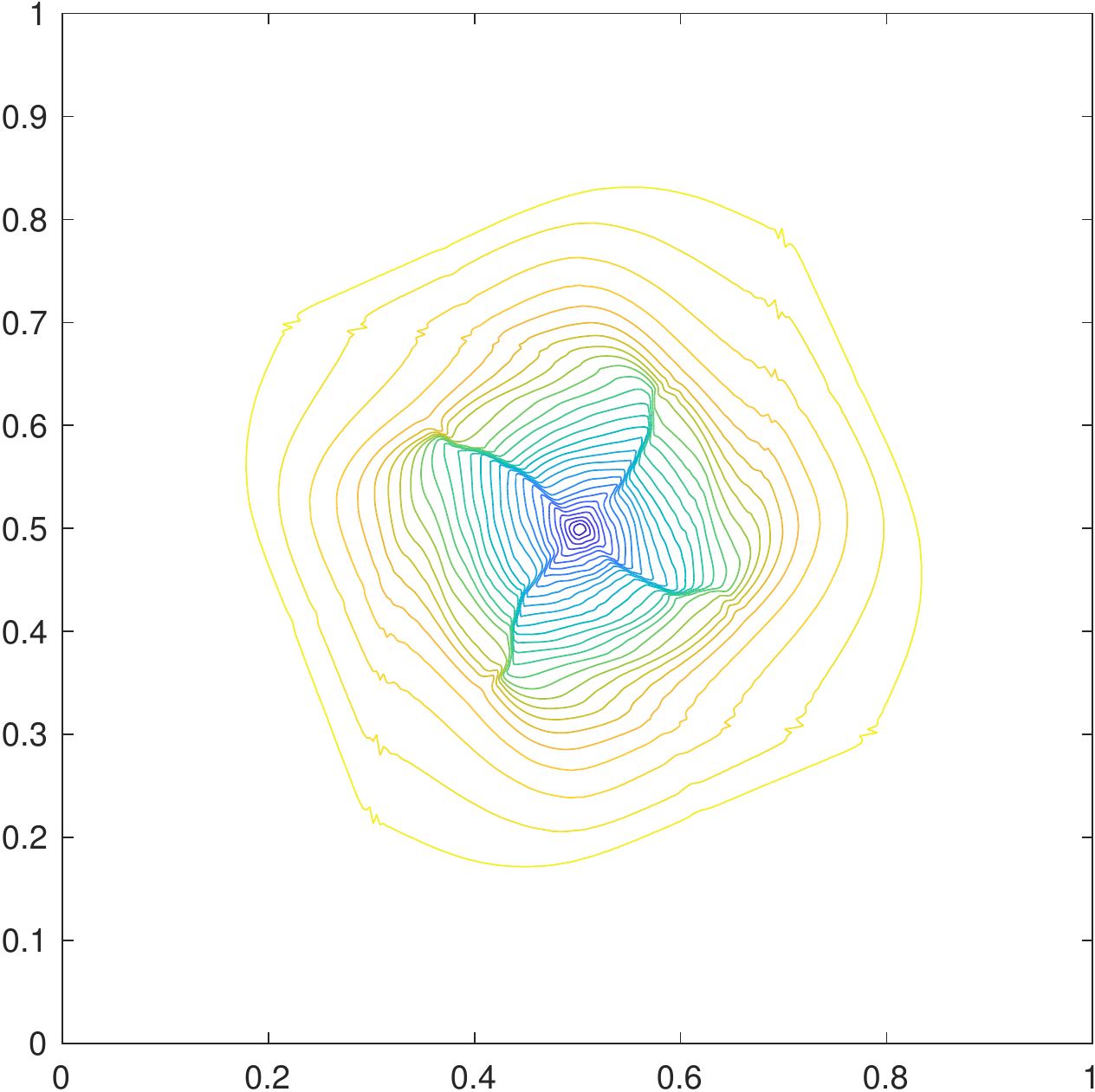}
	\caption{\small{{\tt RP1} of Example \ref{example3.6}:  Left:  $\log \rho$;
			right:   $\log p$. }} \label{fig:RP1}
\end{figure}


The initial data of {\tt RP2} are given {by}
\begin{equation*}
(\rho ,  u , v, p)(x,0)
\; = \; \begin{cases}
(1 ,\, 0 ,\, 0 ,\, 1 ) , & x > 0.5,\,y>0.5, \\
(0.5771 ,\,  -0.3529,\, 0 , \, 0.4) , & x<0.5,\,y>0.5,\\
(1 ,\,  -0.3529,\, -0.3529 , \, 	1) , & x<0.5,\,y<0.5,\\
(0.5771 ,\,  0,\, -0.3529 , \, 0.4) , & x>0.5,\,y<0.5.
\end{cases}
\end{equation*}
Fig. \ref{fig:RP2} shows the contour of the density  { and   pressure logarithms.
The results show }   that those four initial discontinuities first evolve as four rarefaction waves
and then interact each other and form two (almost parallel) curved shock waves perpendicular to the line $x = y$ as time
increases.

\begin{figure}[htbp]
	\centering
	\includegraphics[width=0.36\textwidth]{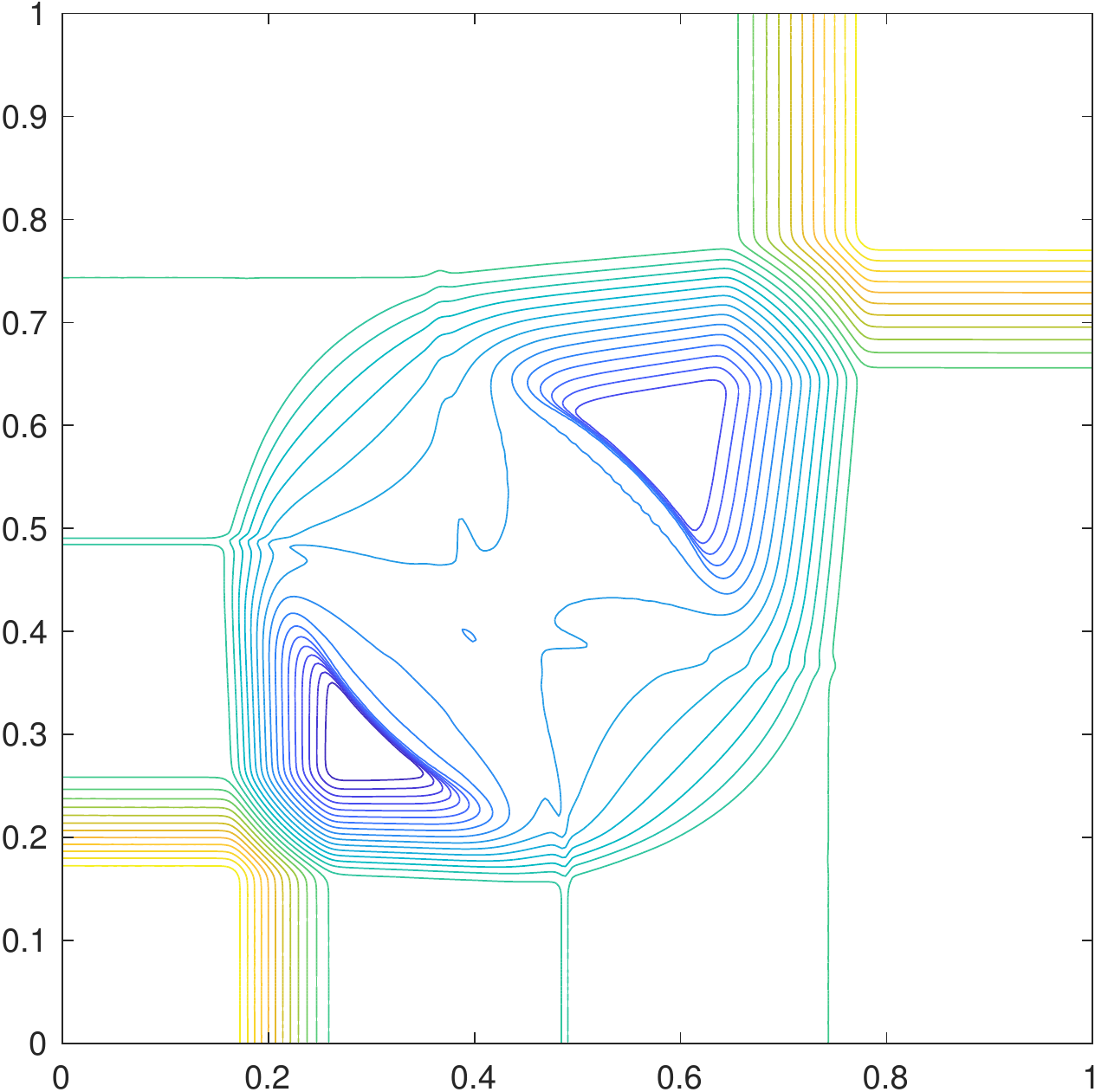}
	\includegraphics[width=0.36\textwidth]{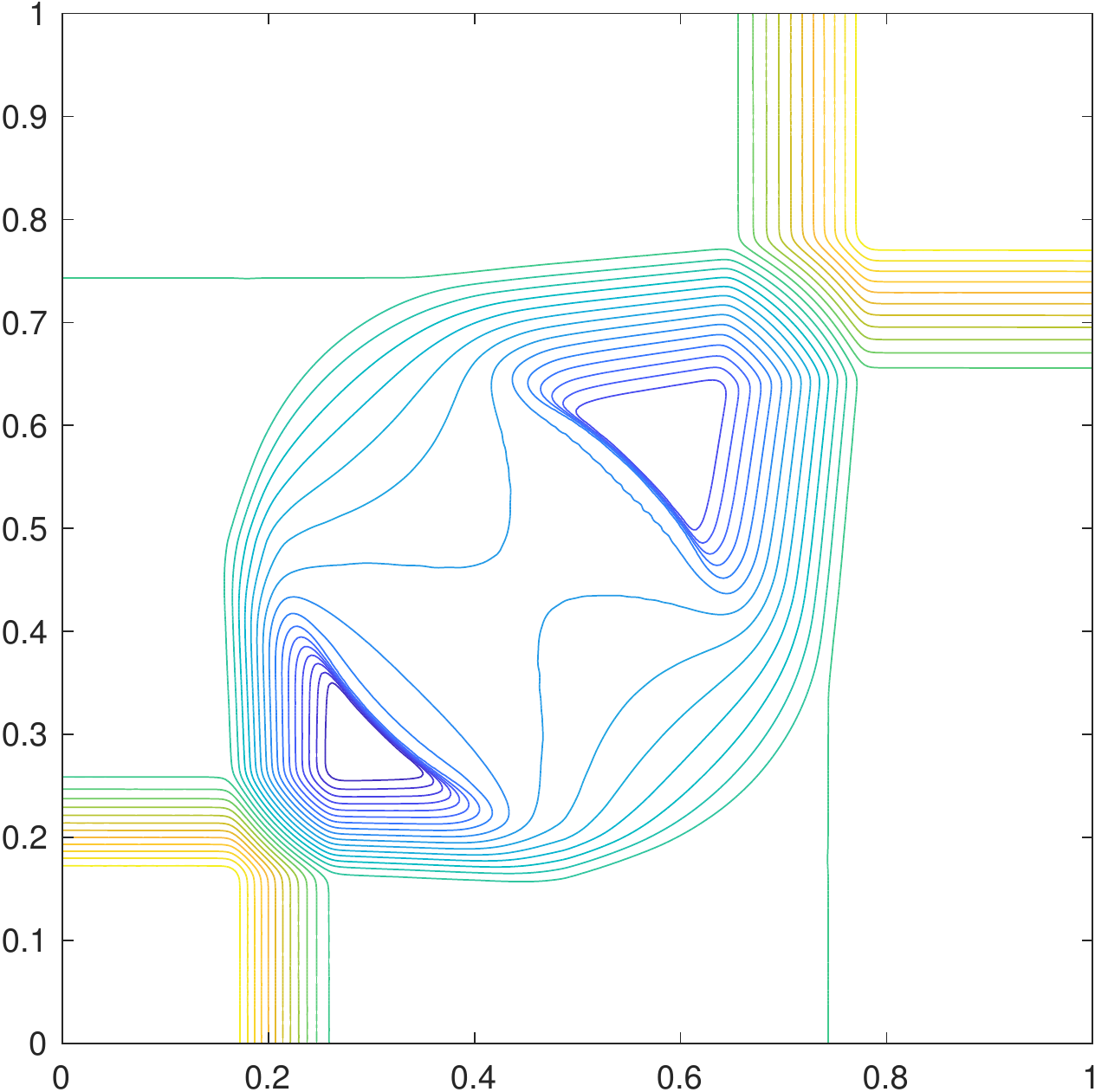}
	\caption{\small{  {\tt RP2} of Example \ref{example3.6}: Left:    { $\log \rho$; right:
			$\log p$.}  }} \label{fig:RP2}
\end{figure}


The initial data of {\tt RP3} are given {by}
\begin{equation*}
(\rho ,  u , v, p)(x,0)
\; = \; \begin{cases}
(0.035145216124503 ,\, 0 ,\, 0 ,\, 0.162931056509027 ) , & x > 0.5,\,y>0.5, \\
(0.1 ,\,  0.7,\, 0.0 , \, 1.0) , & x<0.5,\,y>0.5, \\
(0.5 ,\,  0.0,\, 0.0 , \, 1.0) , & x<0.5,\,y<0.5, \\
(0.1 ,\,  0.0,\, 0.7 , \, 1.0) , & x>0.5,\,y<0.5,
\end{cases}
\end{equation*}
where the left and bottom discontinuities are two contact discontinuities and the top and right are two shock waves with the speed of $0.9345632754$.

Fig. \ref{fig:RP3} shows the contour of the density  { and   pressure logarithms}. We see that four initial discontinuities interact each
other and form a ``mushroom cloud'' around the point $(0.5, 0.5)$ as $t$ increases.

\begin{figure}[htbp]
	\centering
	\includegraphics[width=0.36\textwidth]{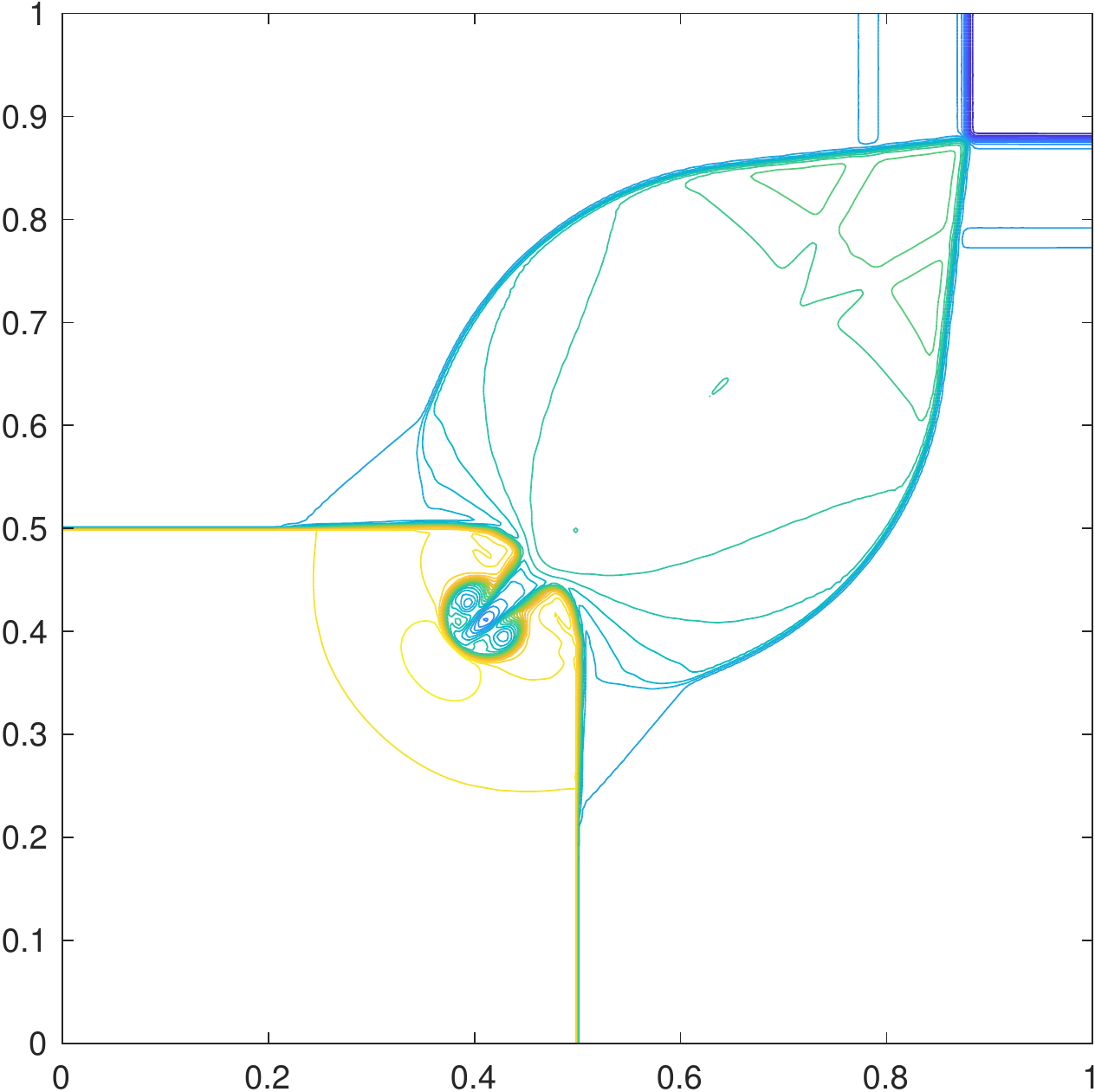}
	\includegraphics[width=0.36\textwidth]{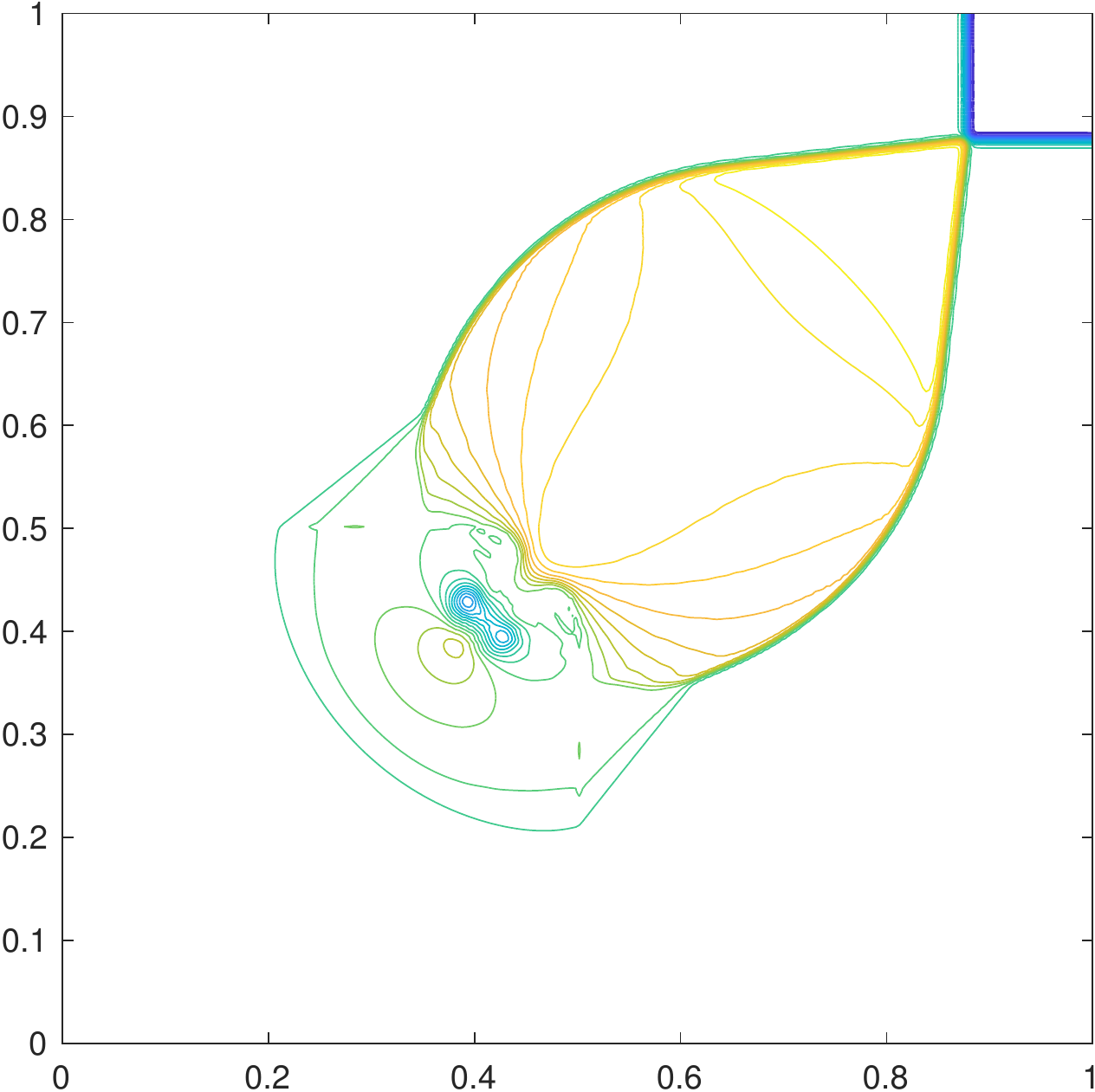}
	\caption{\small{{\tt RP3} of Example \ref{example3.6}: Left:  { $\log \rho$; right:
			$\log p$.}}} \label{fig:RP3}
\end{figure}


\begin{example}[Double Mach reflection problem]\label{example3.9}\rm
	The double Mach reflection problem for the ideal relativistic fluid with the adiabatic index $\Gamma = 1.4$ within the domain $\Omega = [0, 4]\times[0, 1]$ has been
	widely used to test the high-resolution shock-capturing scheme, see e.g. \cite{HeAdaptiveRHD,WuEGRHD,Yang-Tang:2012}.
	Initially, a right-moving oblique shock with speed $v_s = 0.4984$
	is located at $(x, y) = (1/6, 0)$ and makes a $60^\circ$ angle with $x$-axis.
	Thus its position at time $t$ may be given by
	$h(x,t) = \sqrt{3}(x-1/6) - 2v_s t$.
	The left and right states of the shock wave for the primitive variables  are given by
	\begin{equation*}
	\vec V(x,y,0)
	\; = \; \begin{cases}
	\vec V_L , & y > h(x,0), \\
	\vec V_R , & y < h(x,0),
	\end{cases}
	\end{equation*}
	with $\vec V_L = (8.564,0.4247\sin(\pi/3),-0.4247\cos(\pi/3),0.3808)^T$ and
	$\vec V_R = (1.4, 0.0,0.0,0.0025)^T$.
	The  setup of boundary conditions can be found in \cite{HeAdaptiveRHD,WuEGRHD,Yang-Tang:2012}.
	Figs. \ref{fig:RP4}$-$\ref{fig:RP4b} give
	the contours of the density and pressure at time $t = 5.5$ with $640 \times 160$ uniform cells and different $\alpha$.  We see that  the complicated structure around the double Mach region can be clearly captured.
	
	\begin{figure}[htbp]
		\setlength{\abovecaptionskip}{0.cm}
		\setlength{\belowcaptionskip}{-0.cm}
		\subfigure{
			\begin{minipage}[t]{\textwidth}
				\centering
				\includegraphics[width=0.8\textwidth]{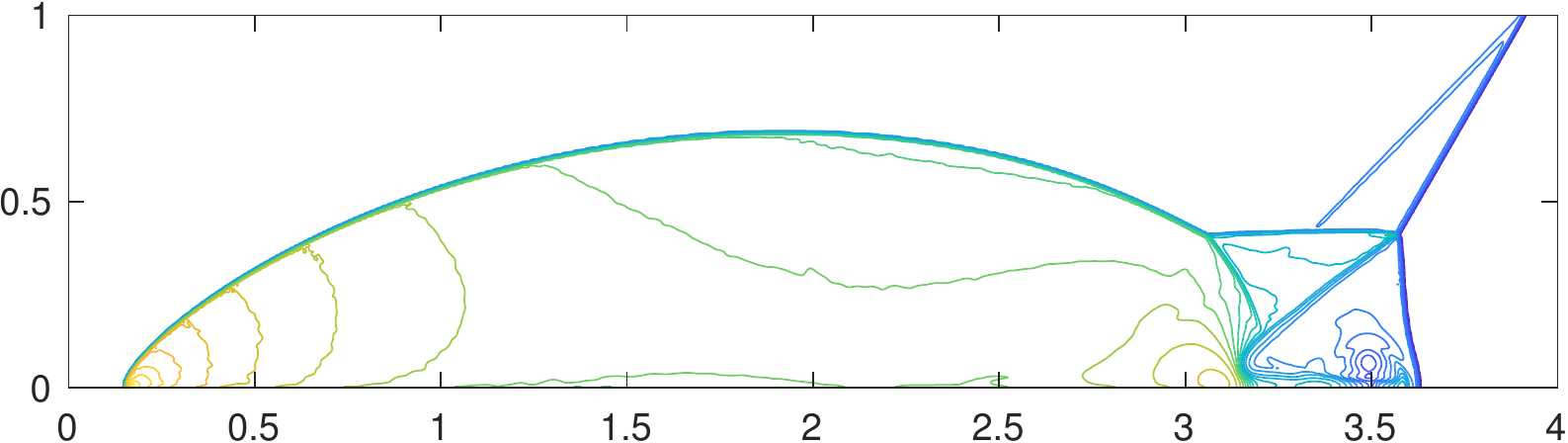}
			\end{minipage}
		}
		\subfigure{
			\begin{minipage}[t]{\textwidth}
				\centering
				\includegraphics[width=0.8\textwidth]{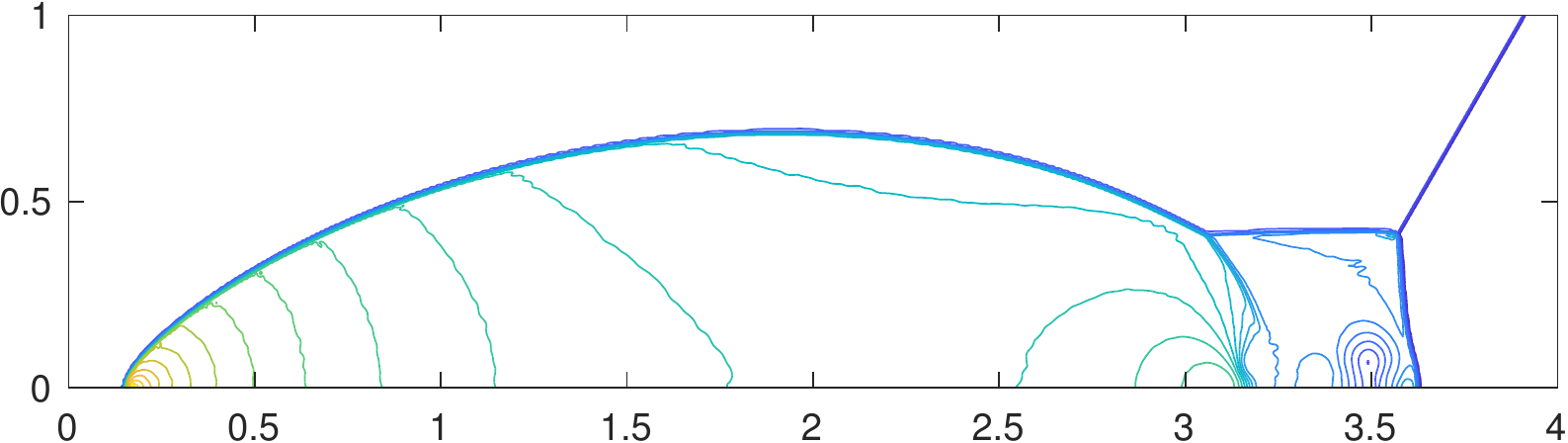}
			\end{minipage}
		}
		\caption{\small{Example \ref{example3.9}: {the contours of $\rho$ (top) and   $p$ (bottom)} with $30$ equally spaced contour lines. $\alpha=\frac13$.}} \label{fig:RP4}
	\end{figure}
	
	\begin{figure}[htbp]
		\setlength{\abovecaptionskip}{0.cm}
		\setlength{\belowcaptionskip}{-0.cm}
		\subfigure{
			\begin{minipage}[t]{\textwidth}
				\centering
				\includegraphics[width=0.8\textwidth]{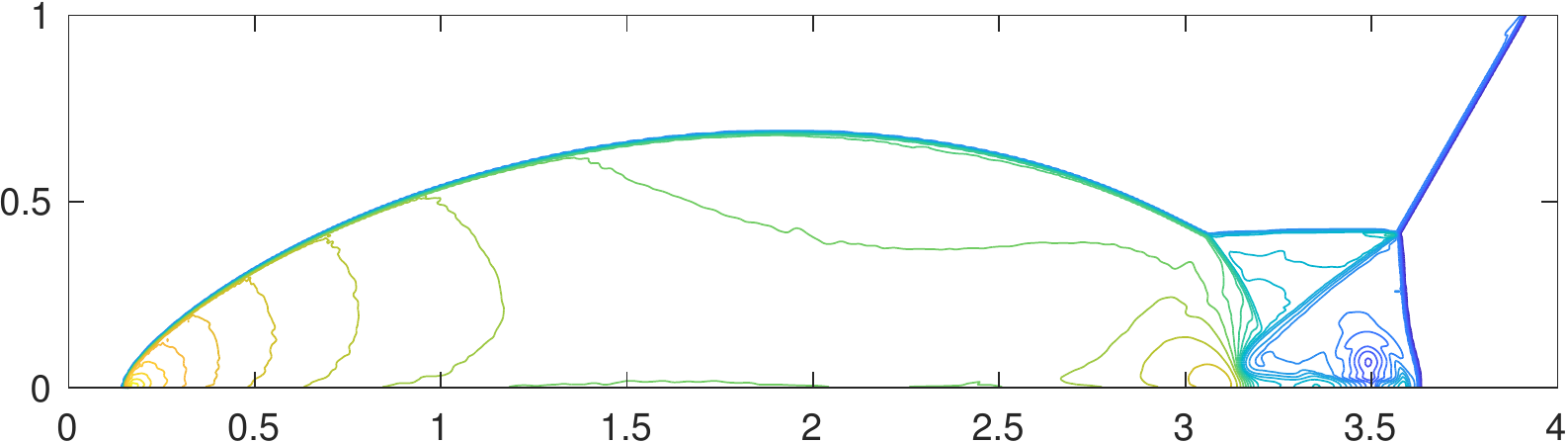}
			\end{minipage}
		}
		\subfigure{
			\begin{minipage}[t]{\textwidth}
				\centering
				\includegraphics[width=0.8\textwidth]{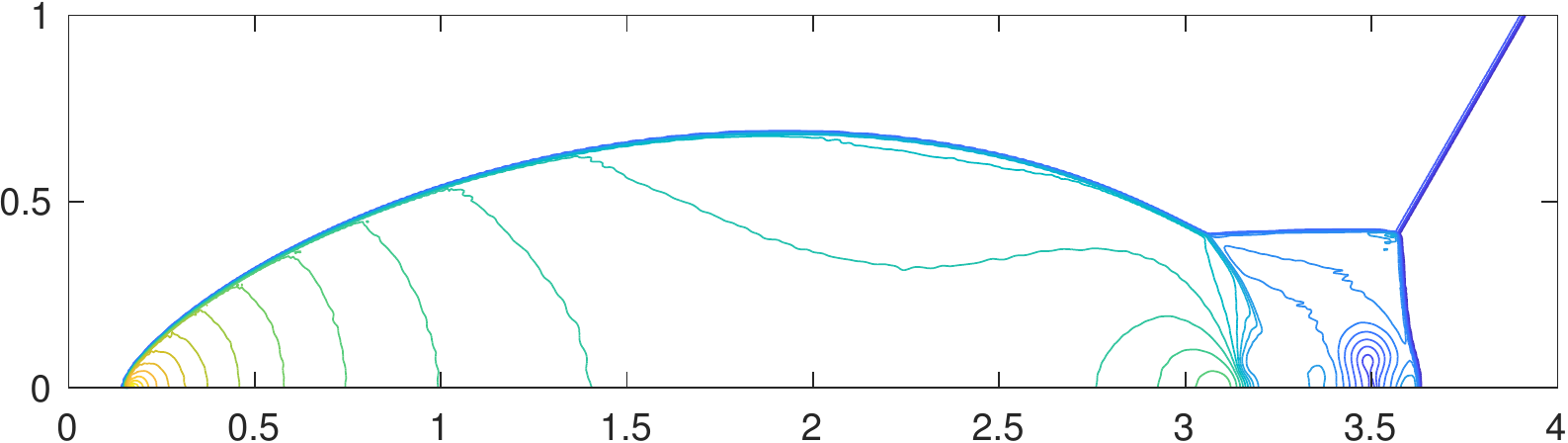}
			\end{minipage}
		}
		\caption{\small Same as Fig. \ref{fig:RP4} except for $\alpha=\frac{1-6\tau}{3-6\tau}$.} \label{fig:RP4a}
	\end{figure}
	
	\begin{figure}[htbp]
		\setlength{\abovecaptionskip}{0.cm}
		\setlength{\belowcaptionskip}{-0.cm}
		\subfigure{
			\begin{minipage}[t]{\textwidth}
				\centering
				\includegraphics[width=0.8\textwidth]{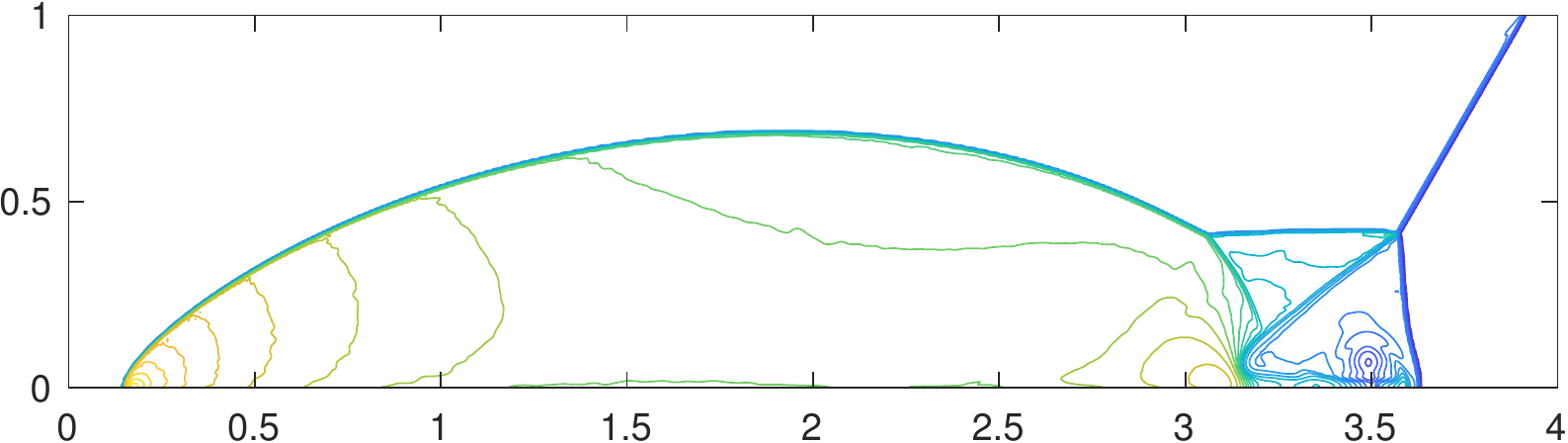}
			\end{minipage}
		}
		\subfigure{
			\begin{minipage}[t]{\textwidth}
				\centering
				\includegraphics[width=0.8\textwidth]{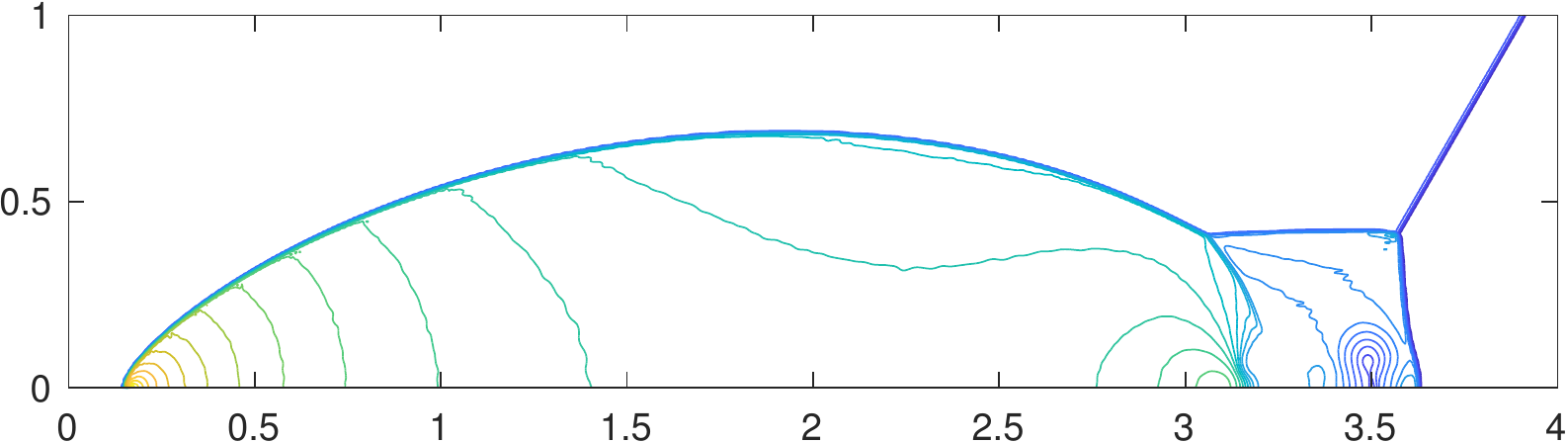}
			\end{minipage}
		}
		\caption{\small Same as Fig. \ref{fig:RP4} except for $\alpha=\frac13 + \tau$.} \label{fig:RP4b}
	\end{figure}
	
\end{example}


\begin{example}[Shock-bubble interaction problems]\label{example3.10}\rm
	The final example considers two shock-bubble interaction problems {within} the
	computational domain $[0,325]\times [0,90]$.
	Their detailed setup can be found in  \cite{WuEGRHD}.
	
	For the first shock-bubble interaction problem,
	the left and right states of planar shock wave moving left are given
	{by}
	\begin{equation*}
	\vec V(x,y,0)
	= \begin{cases}
	(1,0,0,0.05)^T , & x < 265, \\
	(1.865225080631180,-0.196781107378299,0,0.15)^T , & x > 265,
	\end{cases}
	\end{equation*}
	and the bubble is described as
	$\vec V(x,y,0) = (0.1358,0,0,0.05)^T$ if $\sqrt{(x-215)^2+(y-45)^2} \leq 25$.
	The setup of the second shock-bubble problem is the same as the first,
	except for that the initial state of the fluid in the bubble is replaced with
	$
	\vec V(x,y,0) = (3.1538,0,0,0.05)^T$ { if } $\sqrt{(x-215)^2+(y-45)^2} \leq 25$.
	
	Fig. \ref{fig:RP5} gives the contour plots of the density  at $t = 90, 180, 270, 360, 450$ {(from top to bottom)} of the first shock-bubble interaction problem,
	obtained by using our  scheme with $325 \times 90$ uniform cells.
	Fig. \ref{fig:RP6} presents  the contour plots  of the density  at
	several moments $t = 100, 200, 300, 400, 500$ (from top to bottom)
	of the second shock-bubble interaction problem,
	obtained by using our 2D two-stage scheme with $325 \times 90$ uniform cells.
	Those results show that
	the discontinuities and some small wave structures including the
	curling of the bubble interface are captured well and accurately, and at the same time, the multi-dimensional wave structures are also resolved clearly.
	Those plots are also clearly displaying
	the dynamics of the interaction between the shock wave and the bubble
	and obviously different wave patterns of the {interactions
	between those shock waves and the bubbles}.
\end{example}


\begin{figure}[htbp]
	\setlength{\abovecaptionskip}{0.cm}
	\setlength{\belowcaptionskip}{-0.cm}
	\subfigure{
		\begin{minipage}[t]{\textwidth}
			\centering
			\includegraphics[width=0.8\textwidth]{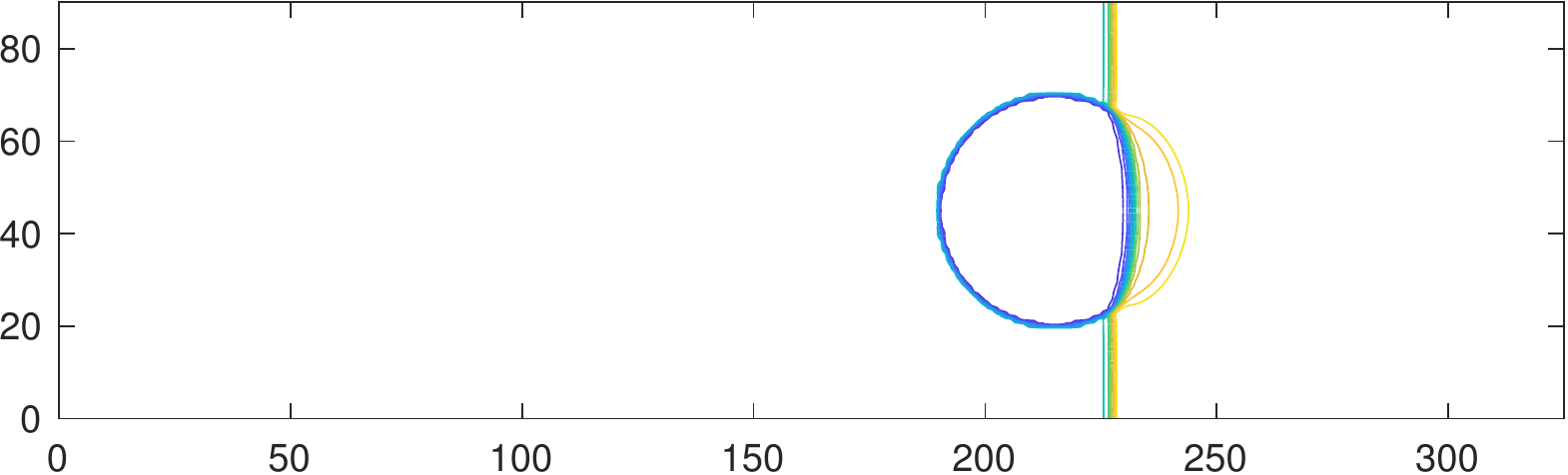}
		\end{minipage}
	}
	\subfigure{
		\begin{minipage}[t]{\textwidth}
			\centering
			\includegraphics[width=0.8\textwidth]{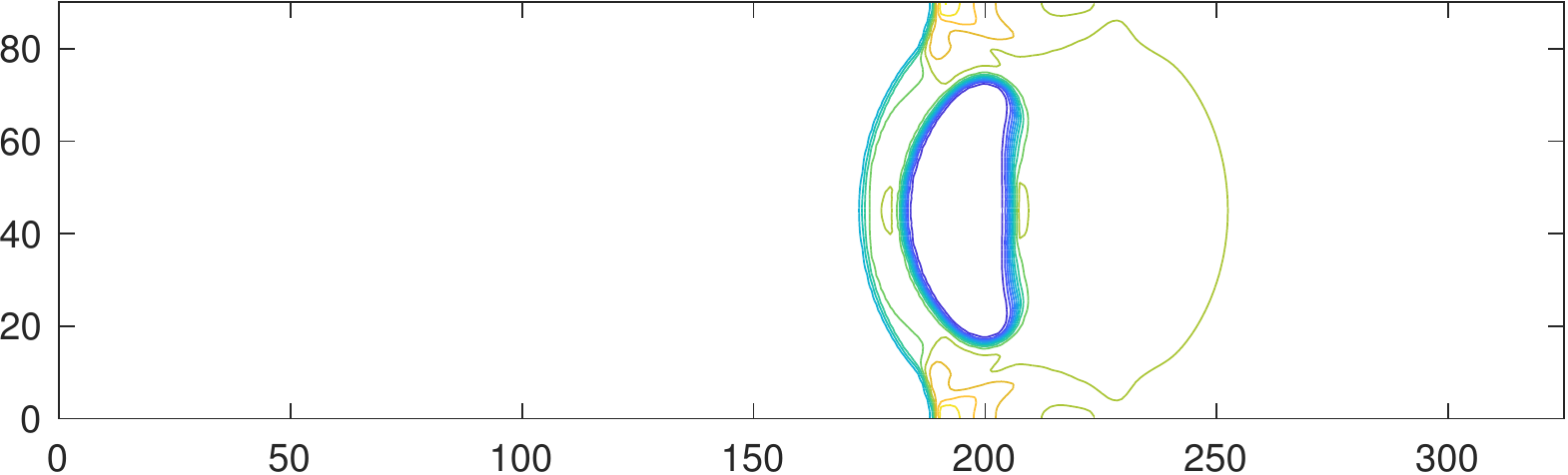}
		\end{minipage}
	}
	\subfigure{
		\begin{minipage}[t]{\textwidth}
			\centering
			\includegraphics[width=0.8\textwidth]{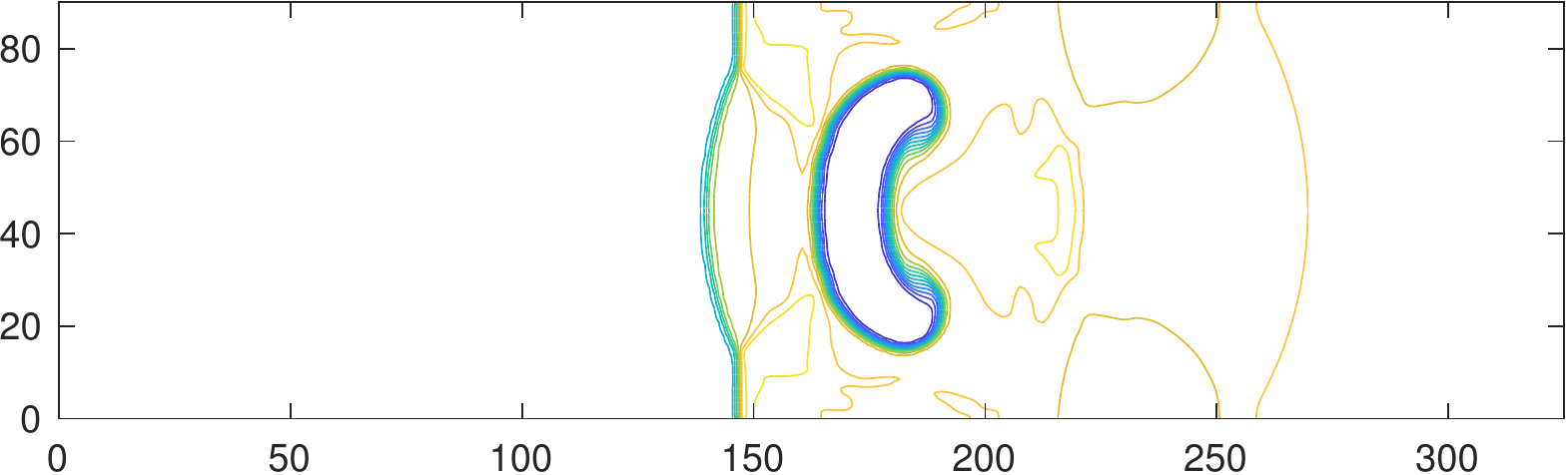}	
		\end{minipage}
	}
	\subfigure{
		\begin{minipage}[t]{\textwidth}
			\centering
			\includegraphics[width=0.8\textwidth]{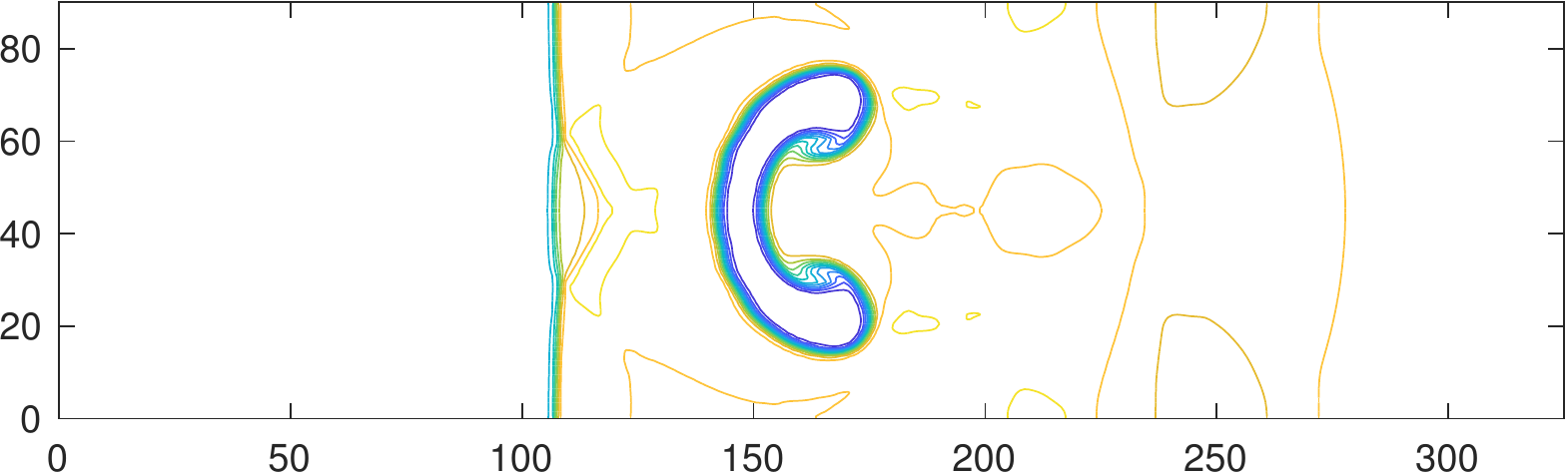}	
		\end{minipage}
	}
	\subfigure{
		\begin{minipage}[t]{\textwidth}
			\centering
			\includegraphics[width=0.8\textwidth]{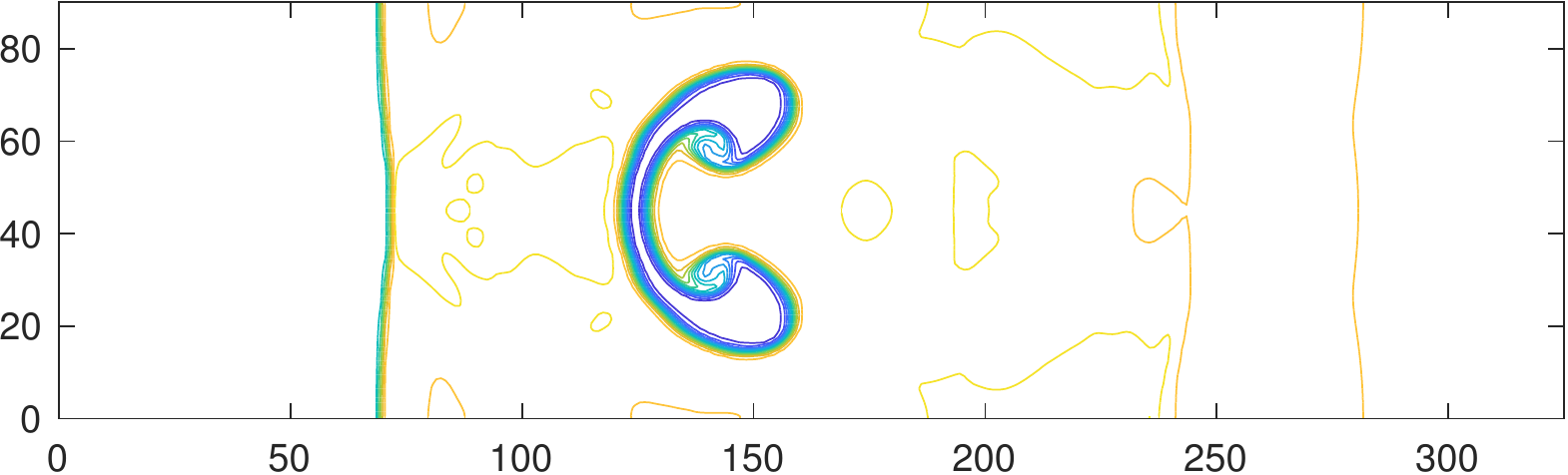}
		\end{minipage}
	}
	\caption{\small{The first problem of Example \ref{example3.10}: the contours of $ \rho$ {at $t = 90,180,270,360,450$} with $15$ equally spaced contour lines.}} \label{fig:RP5}
\end{figure}


\begin{figure}[htbp]
	\setlength{\abovecaptionskip}{0.cm}
	\setlength{\belowcaptionskip}{-0.cm}
	\subfigure{
		\begin{minipage}[t]{\textwidth}
			\centering
			\includegraphics[width=0.8\textwidth]{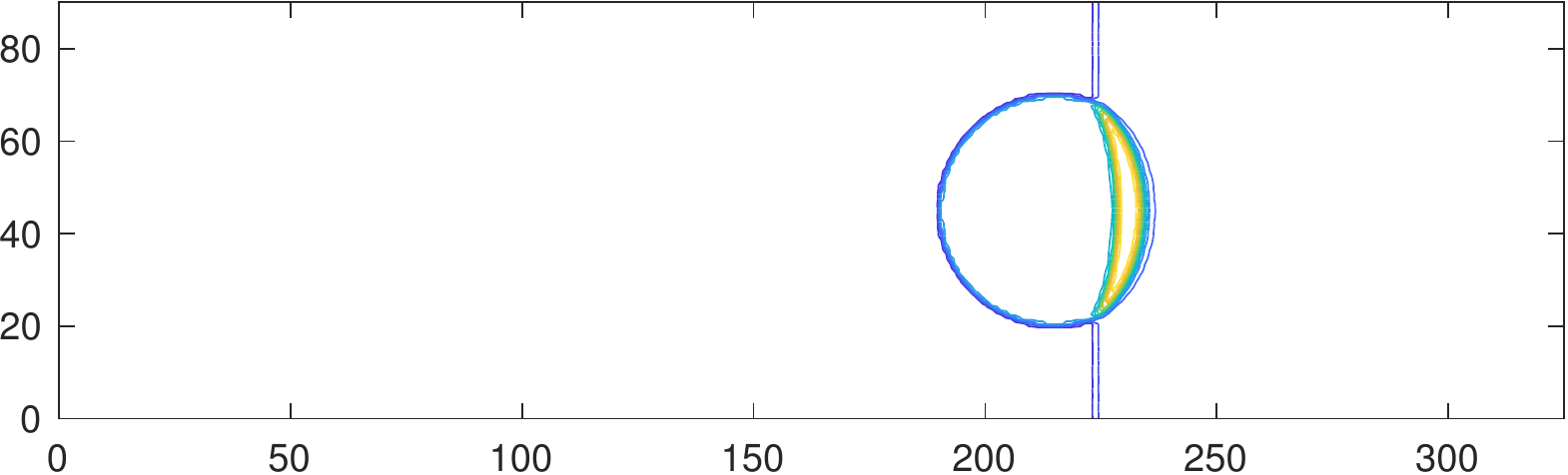}
		\end{minipage}
	}
	\subfigure{
		\begin{minipage}[t]{\textwidth}
			\centering
			\includegraphics[width=0.8\textwidth]{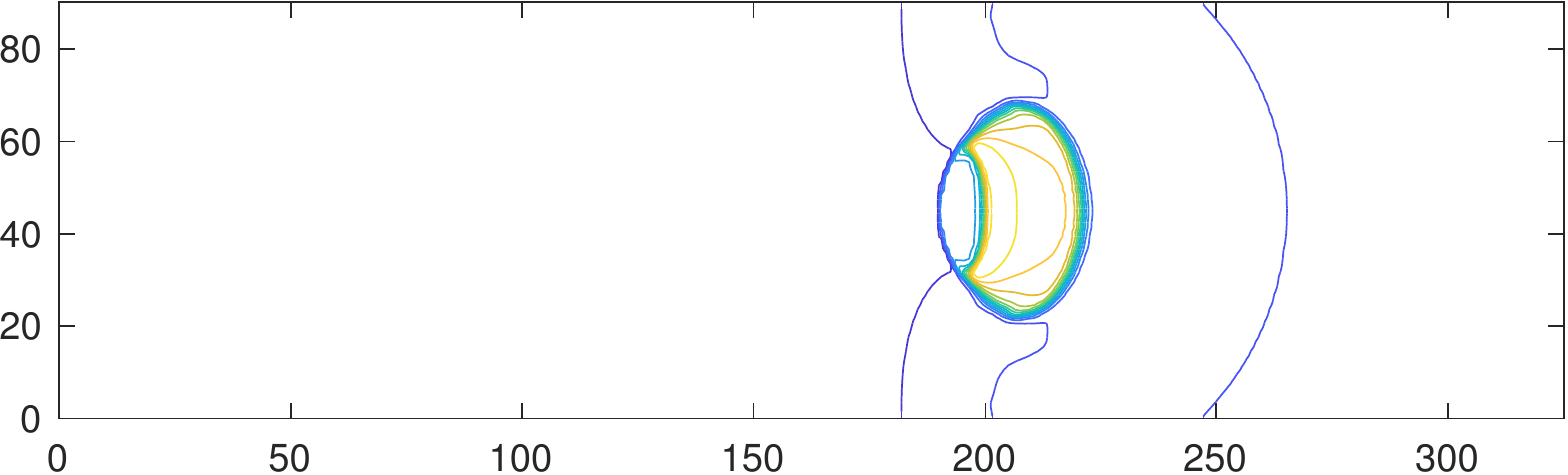}
		\end{minipage}
	}
	\subfigure{
		\begin{minipage}[t]{\textwidth}
			\centering
			\includegraphics[width=0.8\textwidth]{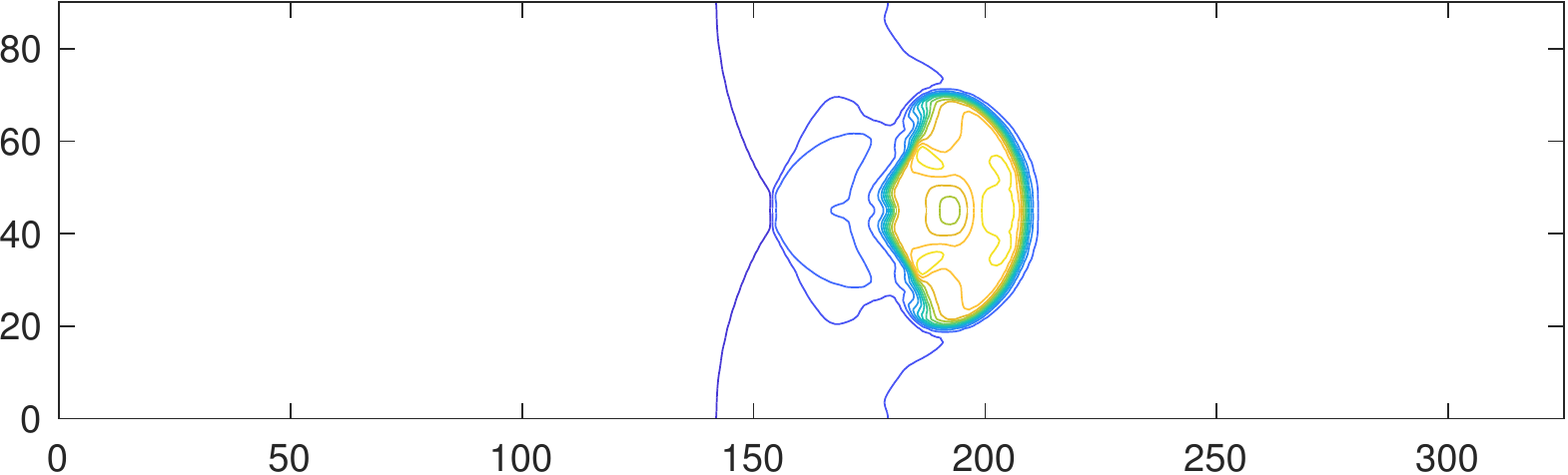}
		\end{minipage}
	}
	\subfigure{
		\begin{minipage}[t]{\textwidth}
			\centering
			\includegraphics[width=0.8\textwidth]{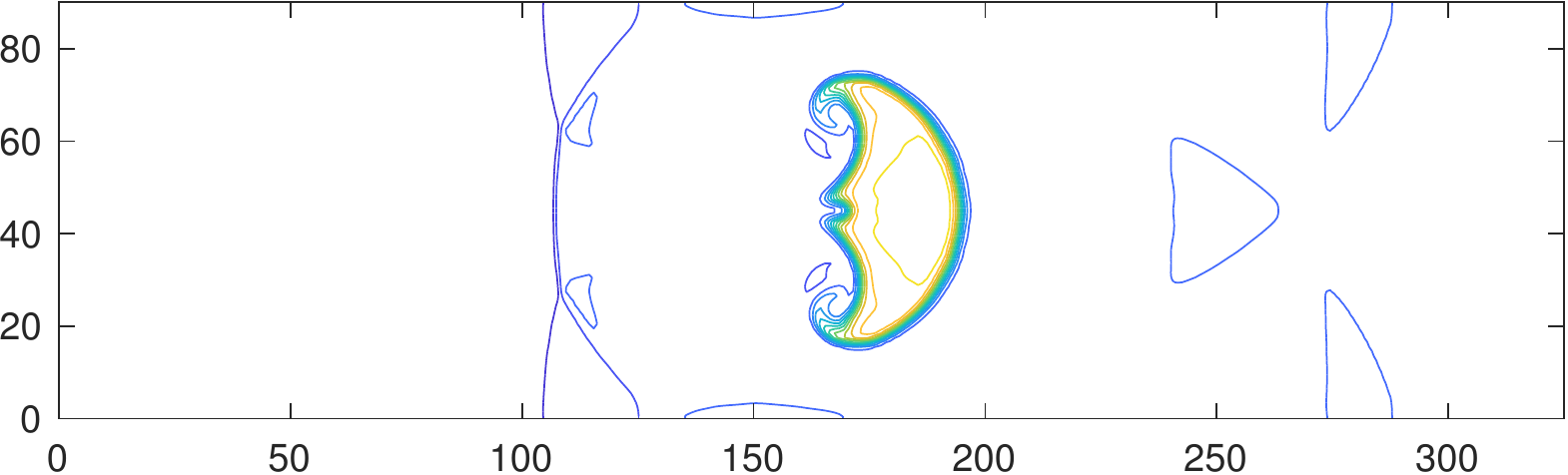}
		\end{minipage}
	}
	\subfigure{
		\begin{minipage}[t]{\textwidth}
			\centering
			\includegraphics[width=0.8\textwidth]{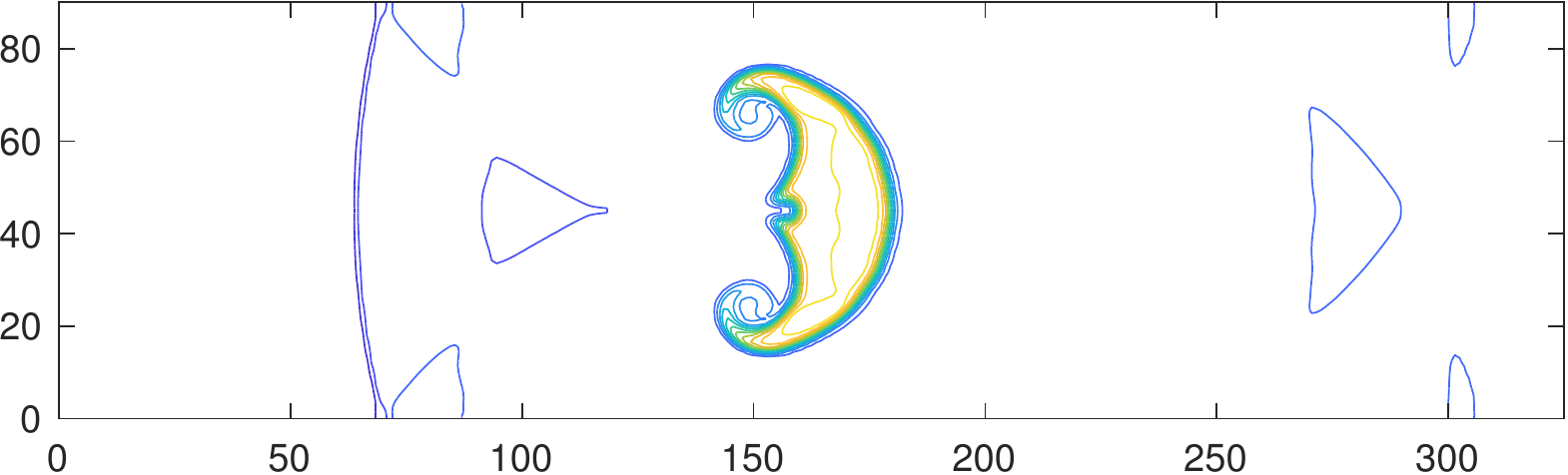}
		\end{minipage}
	}
	\caption{\small{The second problem of Example \ref{example3.10}: the contours of $ \rho$ { at $t=100,200,300,400,500$} with $15$ equally spaced contour lines.}} \label{fig:RP6}
\end{figure}

\section{Conclusions}
\label{Section-conclusion}

The paper   studied the two-stage fourth-order accurate
time discretization \cite{LI-DU:2016} and  its application to the special relativistic hydrodynamical {(RHD)} equations. It was
shown that 
new two-stage fourth-order accurate time discretizations
could be proposed.
The  local ``quasi 1D'' GRP {(generalized Riemann
problem)} of  { the special RHD } equations was
also analytically resolved.
With the aid of the direct Eulerian {GRP}  methods \cite{Yang-He-Tang:2011,Yang-Tang:2012}
and  the analytical resolution of   local ``quasi 1D'' GRP as well as
the adaptive primitive-conservative scheme  \cite{E.F.Toro:2013},
the two-stage fourth-order accurate time discretizations were successfully  implemented
for the 1D and 2D special RHD equations.
The adaptive primitive-conservative scheme was used to reduce the spurious solution  generated by the conservative scheme across the contact discontinuity.
Several  numerical
experiments were conducted to demonstrate the performance and accuracy as well as robustness of our
schemes.

	\appendix
\section{The resolution of quasi 1D GRP of special RHD equations}\label{Appendix:001}
{The   equation}  \eqref{eq:quasi1D} can be {equivalently written} into the  primitive variable form
\begin{equation} \label{eq:quasi1D-pri}
\pa_t \vec V + \widetilde{\vec A}(\vec V) \pa_x \vec V= \vec C=[C_1,
C_2,
C_3,
C_4]^T,
\end{equation}
where
\begin{align*}
C_4=&\frac{1}{1-(u^2+v^2)c_s^2}   [ h_4(\G-1+(u^2+v^2)c_s^2) - \frac{\G-1}{W}h_1 - (c_s^2+\G-1)(uh_2+vh_3) ],\\
C_2=&\frac{1}{\rho h W^2}  [ h_2 - uh_4 - uC_4 ],\
C_3=\frac{1}{\rho h W^2}  [ h_3 - vh_4 - vC_4 ],
C_1
=
\frac{h_1}{W} - W^2\rho(uC_2+vC_3),
\end{align*}
and $\vec h := (h_1,\, h_2,\, h_3,\, h_4)^T$  denotes {the right hand of }  \eqref{eq:quasi1D}.

For the matrix $\widetilde{\vec A}(\vec V)$ given in  \eqref{eq:matrixA},  its eigenvalues and  (right and left) eigenvector matrices can be easily given as follows
\begin{equation} \label{eq:la}
\la_1 = \dfrac{u(1-c_s^2)-c_sW^{-1}H^{1/2}}{1-(u^2+v^2)c_s^2} , \  \la_2=\la_3 = u, \  \la_4 = \dfrac{u(1-c_s^2)+c_sW^{-1}H^{1/2}}{1-(u^2+v^2)c_s^2} ,
\end{equation}
\begin{equation} \label{eq:MatRR}
\widetilde{\vec R}(\vec V)  =
\begin{pmatrix}
\dfrac{\rho W^2}{c_s} & 1 & 0 & \dfrac{\rho W^2}{c_s} \\
-W H^{1/2}& 0 & 0 & W H^{1/2} \\
\dfrac{v(uW H^{1/2}-c_s)}{1-u^2} & 0 & 1 & -\dfrac{v(uW H^{1/2}+c_s)}{1-u^2} \\
\rho h c_sW^2 & 0 & 0 & \rho h c_sW^2
\end{pmatrix},
\end{equation}
and 
\begin{equation} \label{eq:MatLL}
\widetilde{\vec R}^-(\vec V)  = \frac{1}{|\widetilde{\vec R}(\vec V) |}
\begin{pmatrix}
0 &  -\rho hc_s W^2 & 0 & W H^{1/2} \\
|\widetilde{\vec R}(\vec V) |& 0 & 0 & \dfrac{-2\rho W^3H^{1/2}}{c_s} \\
0 & \dfrac{2uv\rho h c_sW^3 H^{1/2}}{1-u^2} & |\widetilde{\vec R}(\vec V)| & \dfrac{2v c_sW H^{1/2}}{1-u^2}  \\
0 & \rho hc_s W^2 & 0 & W H^{1/2}
\end{pmatrix},
\end{equation}
where $|\widetilde{\vec R}(\vec V)| = 2\rho h c_sW^3H^{1/2}$.

For the sake of brevity, we will omit
the notation $( \cdot )_*$ widely used in the direct Eulerian
GRP methods \cite{Han-Li-Tang:2010,Han-Li-Tang:2011,Wu-Tang:2016,Wu-Yang-Tang:2014,Wu-Yang-Tang:2014b,Yang-He-Tang:2011,Yang-Tang:2012}.

\subsection{Resolution of shock wave}
It is very similar to that given in \cite{Yang-Tang:2012} except for $(3.43)${,  because} the source terms   affect
$\frac{D }{D s} := \pa_t + s\cdot \pa_x$ as follows
\begin{equation*}
\begin{cases}
\frac{D u}{D s} = \frac{1-us}{1-u^2}\cdot \frac{D u}{D t} + \frac{u-s}{1-u^2}\cdot\frac{H}{\rho h c_s^2}\cdot \frac{D p}{D t} + \frac{s-u}{(1-u^2)} \cdot \left[ u\, C_2  + \frac{H}{\rho h c_s^2} C_4\right], \\
\frac{D p}{D s} = \frac{1-us}{1-u^2} \cdot \frac{D p}{D t} + \frac{u-s}{1-u^2}\cdot \rho h W^2 \cdot \frac{D u}{D t} + \frac{s-u}{(1-u^2)}\cdot \left[ u\, C_4 + \rho h W^2 C_2 \right],
\end{cases}
\end{equation*}
where $s$ is the shock speed. The present result is given in the following lemma.

\begin{lemma} \label{lem:shock} \rm
	The limiting values of $(Du/Dt)_*$ and $(Dp/Dt)_*$ satisfy
	\begin{equation} \label{eq:shock}
	a_R \, \left(\frac{D u}{D t}\right)_* \;+\; b_R \, \left(\frac{D p}{D t}\right)_* \;=\; d_R,
	\end{equation}
	where
	\begin{equation*}
	A_{us} = u C_2 + \frac{H}{\rho h c_s^2} C_4 , \  A_{ps} = \rho h W^2 C_2 + u C_4 ,
	\end{equation*}
	and
	\begin{equation*}
	\begin{cases}
	a_R = A_u + \frac{u-s}{1-us}\rho h W^2 A_p ,   \\
	b_R = A_p + \frac{u-s}{1-us} \frac{H}{\rho h c_s^2} A_u ,  \\
	d_R = \frac{1-u^2}{1-us} A_{rhs} + \frac{u-s}{1-us} \cdot \left( A_{us}A_u + A_{ps} A_p \right), 	
	\end{cases}
	\end{equation*}
	with $A_u$ and $A_p$ defined in \cite{Yang-Tang:2012}.
\end{lemma}

\subsection{Resolution of   centered rarefaction wave}
With the help of $\widetilde{\vec R}(\vec V)$,  we can easily derive the Riemann invariants
\begin{equation}
\begin{cases}
\mbox{for} \;\la_1 :& \psi_+,\,S,\,V,  \\
\mbox{for} \;\la_2(\la_3): &  u,\,p ,\\
\mbox{for} \;\la_4: & \psi_-,\,S,\,V,
\end{cases}
\end{equation}
where $V=hvW$,
$
\psi_{\pm} = \dfrac{1}{2}\ln\left( \dfrac{1+u}{1-u} \right) \pm \int^p \varphi dp = \dfrac{1}{2}\ln\left( \dfrac{1+u}{1-u} \right) \pm \int^{\rho} \tilde{\varphi} d\rho$,
and
\begin{equation}
\varphi = \dfrac{H^{1/2}}{\rho h c_sW(1-u^2)} = \dfrac{\sqrt{h^2+V^2(1-c_s^2)}}{\rho c_s(h^2+V^2)}, \quad \tilde{\varphi} = h c_s^2 \varphi.
\end{equation}
The Riemann invariants $\psi_{\pm}$ satisfy
\begin{equation}\label{eq:dpsi0}
d \psi_{\pm} = \dfrac{1}{1-u^2}du \pm \left[ \varphi dp + K_S dS + K_V dV \right],
\end{equation}
where
\begin{equation} \label{eq:dpsi1}
\begin{cases}
K_S = \tilde{K}_S - \varphi \frac{\pa p}{\pa S},  \\
K_V = \frac{\pa }{\pa V} \int^p \varphi dp, \\
\tilde{K}_S = \frac{\pa}{\pa S} \int^{\rho} \tilde{\varphi} d\rho, \\
\frac{\pa \varphi}{\pa V} = \varphi V \left( \frac{1-c_s^2}{h^2+V^2(1-c_s^2)} - \frac{2}{h^2+V^2} \right).
\end{cases}
\end{equation}
Note that the lower limits of two integrals in \eqref{eq:dpsi1} may be $\rho_0 > 0$ and $p_0 >0$, respectively.

Thanks to the thermodynamic relation $TdS = de - \frac{p}{\rho^2}d\rho$ and the $\G$-law $p = (\G-1)\rho e$, one has
\begin{equation*}
TdS = \frac{1}{(\G-1)\rho} \left(dp - hc_s^2d\rho\right).
\end{equation*}
Together with $V = hW v$, \eqref{eq:dpsi0} and \eqref{eq:quasi1D-pri}, we  obtain
\begin{equation} \label{eq:quasi1D-char}
\begin{cases}
\pa_t \psi_- + \la_1 \pa_x \psi_- = B_1 - (\la_1-u)\cdot (K_S\pa_x S + K_V \pa_x V), \\
\pa_t \psi_+ + \la_4 \pa_x \psi_+ = B_2 + (\la_4-u)\cdot (K_S\pa_x S + K_V \pa_x V),  \\
T(\pa_t S + u \pa_x S) =  B_3,\\
\pa_t V + u \pa_x V = B_4,  \\
\end{cases}
\end{equation}
where
\begin{equation*}
\begin{cases}
B_1 = \frac{1}{1-u^2}C_2 - \varphi C_4 - \frac{K_S}{T} B_3 - K_V B_4, \\
B_2 = \frac{1}{1-u^2}C_2 + \varphi C_4 + \frac{K_S}{T} B_3 + K_V B_4,  \\
B_3 = \frac{1}{(\G-1)\rho}\cdot \left( C_4 - hc_s^2 C_1 \right)= \frac{1}{\rho}\cdot \left( h_4 - uh_2 -vh_3 - \frac{h}{W} h_1 \right),\\
B_4 = -\frac{hv}{\rho} h_1 + \frac{1}{\rho W}h_3.
\end{cases}
\end{equation*}

Based on the above  preparation, following the procedure in  \cite{Yang-Tang:2012},
one can get the main result of resolving the left rarefaction waves.

\begin{lemma} \label{lem:rarefaction} \rm
	The limiting values of $(Du/Dt)_*$ and $(Dp/Dt)_*$ satisfy
	\begin{equation} \label{eq:rare}
	a_L \, \left(\frac{D u}{D t}\right)_* \;+\; b_L \, \left(\frac{D p}{D t}\right)_* \;=\; d_L,
	\end{equation}
	with
	\begin{equation*}
	\begin{cases}
	a_L =  \frac{1}{1-u^2} ,   \ \
	b_L = \varphi,  \\
	d_L =  \dfrac{\la_4 - u}{\la_4-\la_1}  \left(\frac{\pa t}{\pa \al}\right)^{-1} \frac{\pa \psi_+}{\pa \al}  + \dfrac{u-\la_1}{\la_4-\la_1} \left[ B_2 + (\la_4-u)
	(K_S\pa_x S + K_V \pa_x V) \right] - \frac{K_S}{T}B3 - K_V B_4,	
	\end{cases}
	\end{equation*}
	where $\pa_x S$, $\pa_x V$ and $\frac{\pa \psi_+}{\pa \al}$ are calculated
	by \eqref{eq:dxS}, \eqref{eq:dxV} and  \eqref{eq:dpsi}, respectively.
\end{lemma}

\begin{proof}
	Since
	\begin{align}
	T\frac{\pa S}{\pa \al}  = & \frac{\pa t}{\pa \al} \cdot \left[ B_3 + (\la_1 - u)T\frac{\pa S}{\pa x} \right],   \
	T\frac{\pa S}{\pa \b} =   \frac{\pa t}{\pa \b} \cdot \left[ B_3 + (\la_4 - u)T\frac{\pa S}{\pa x} \right], \label{eq:TS}
	\end{align}
	taking $\al$-derivative to  \eqref{eq:TS} gives
	\begin{equation*}
	\frac{\pa }{\pa \al} \left(  T\frac{\pa S}{\pa \b} \right) = \frac{\pa^2 t}{\pa \b \pa \al} \cdot \left[ B_3 + (\la_4 - u)T\frac{\pa S}{\pa x} \right] + \frac{\pa t}{\pa \b} \cdot \frac{\pa }{\pa \al} \left[ B_3 + (\la_4 - u)T\frac{\pa S}{\pa x} \right].
	\end{equation*}
	Together  with
	\begin{equation*}
	\frac{\pa }{\pa \b} \left(  T\frac{\pa S}{\pa \al} \right) = \frac{\pa T}{\pa \b} \frac{\pa S}{\pa \al} + \frac{\pa^2 S}{\pa \b \pa \al}, \
	\frac{\pa }{\pa \al} \left(  T\frac{\pa S}{\pa \b} \right) = \frac{\pa T}{\pa \al} \frac{\pa S}{\pa \b} + \frac{\pa^2 S}{\pa \b \pa \al},
	\end{equation*}
	setting $\al=0$ gives
	\begin{equation*}
	\frac{\pa }{\pa \b} \left(T\frac{\pa S}{\pa \al}(0,\b)\right) = \left[\frac{\pa \ln T}{\pa \b} + \frac{(\la_4 - u)}{(\la_4 - \la_1)(\la_1 -u)}\right] \bigg|_{\al = 0}
	\left(T\frac{\pa S}{\pa \al}(0,\b)\right)
	+ \frac{B_3}{u-\la_1} \bigg|_{\al = 0}    \frac{\pa t}{\pa \al}(0,\b).
	\end{equation*}
	Hence it holds
	\begin{align*}
	\left(\frac{\pa t}{\pa \al}\bigg|_L\right)^{-1} \left( T\frac{\pa S}{\pa \al} \right)\bigg|_{\al = 0}
	= & \dfrac{T}{T_L}\bigg|_{\al = 0} \mbox{exp}\left(\int^{\b}_{\b_L} \frac{(\la_4 - u)}{(\la_4 - \om)(\om -u)} d\om \right) \Bigg[ B_3^L + (\la_1 - u)|_L T_LS'_L \\
	& +\int^{\b}_{\b_L} \frac{B_3}{u-\b} \dfrac{\pa_{\al} t (0,\b)}{\pa_{\al} t (0,\b_L)} \dfrac{T_L}{T} \mbox{exp}\left(\int^{\b}_{\b_L} \frac{(\la_4 - u)}{(\la_4 - \om)(u-\om)} d\om \right) d\b \Bigg],
	\end{align*}
	where
	\begin{equation*}
	\dfrac{\pa_{\al} t (0,\b)}{\pa_{\al} t (0,\b_L)} = \exp\left(\int^{\b}_{\b_L} \frac{1}{(\la_4 - \om)} d\om\right), \quad \dfrac{T}{T_L}\bigg|_{\al = 0}  = \dfrac{(hc_s^2)|_{\al = 0}}{(hc_s^2)|_L},
	\end{equation*}
	and the subscript $L$ represents $(0,\beta_L)$.
	Therefore, one gets
	\begin{equation} \label{eq:dxS}
	\left( T\frac{\pa S}{\pa x} \right)\bigg|_{\al = 0}  \;=\; \frac{1}{u-\la_1}\bigg|_{\al = 0}\cdot\left[B_3 - \left(\frac{\pa t}{\pa \al}\right)^{-1} T\frac{\pa S}{\pa \al} \right]\bigg|_{\al = 0}.
	\end{equation}
	
	Similarly,  one can calculate $\frac{\pa V}{\pa x} (0,\beta)$ as
	\begin{equation} \label{eq:dxV}
	\frac{\pa V}{\pa x} (0,\beta)\;=\; \frac{1}{u-\la_1}\bigg|_{\al = 0} \cdot\left[B_4 - \left(\frac{\pa t}{\pa \al}\right)^{-1} \frac{\pa V}{\pa \al} \right]\bigg|_{\al = 0},
	\end{equation}
	where
	\begin{align*}
	\left(\frac{\pa t}{\pa \al}\bigg|_L\right)^{-1} \frac{\pa V}{\pa \al} (0,\beta)
	= &  \mbox{exp}\left(\int^{\b}_{\b_L} \frac{(\la_4 - u)}{(\la_4 - \om)(\om -u)} d\om \right) \Bigg[ B_4^L + (\la_1 - u)|_L V'_L, \\
	&+   \int^{\b}_{\b_L} \frac{B_4}{u-\b} \dfrac{\pa_{\al} t (0,\b)}{\pa_{\al} t (0,\b_L)} \mbox{exp}\left(\int^{\b}_{\b_L} \frac{(\la_4 - u)}{(\la_4 - \om)(u-\om)} d\om \right) d\b \Bigg].
	\end{align*}

	Finally, let us calculate $\frac{\pa \psi_+}{\pa \al}(0,\beta)$.
	Since
	\begin{align*}
	\frac{\pa \psi_+ }{\pa \al} =& \frac{\pa t}{\pa \al} \cdot \left[ \frac{D_+  }{D t}\psi_+ + (\la_1 - \la_4)\frac{\pa \psi_+}{\pa x} \right],  \
	\frac{\pa \psi_+ }{\pa \b} =  \frac{\pa t}{\pa \b} \cdot \frac{D_+  }{D t}\psi_+ ,
	\end{align*}
	where $\frac{D_{+} }{D t} = \frac{\pa }{\pa t} + \la_4 \frac{\pa }{\pa x}${, setting} $\al = 0$ gives
	\begin{equation*}
	\frac{\pa }{\pa \b} \left(\frac{\pa \psi_+}{\pa \al}(0,\b)\right) = \bigg( \dfrac{1}{\la_4-\la_1}\frac{\pa t}{\pa \al} \cdot \left[ B_2 + (\la_4-u)\cdot (K_S\pa_x S + K_V \pa_x V) \right] \bigg) \bigg|_{\al=0}.
	\end{equation*}
	Some tedious manipulations yield
	\begin{align} \label{eq:dpsi}
	\left(\frac{\pa t}{\pa \al}\bigg|_L\right)^{-1} & \frac{\pa \psi_+}{\pa \al}(0,\beta)
	= B_2^L + (\la_1-u)|_L \cdot \left[ \dfrac{K_S^L}{T_L}T_L S'_L + K_V^L V'_L \right] \nonumber
	\\ &+ (\la_1-\la_4)|_L \cdot \left[ \dfrac{1}{1-u_L^2} u'_L+ \varphi_L p'_L \right] \nonumber \\
	&+ \int^{\b}_{\b_L} \frac{1}{\la_4 -\om} \dfrac{\pa_{\al} t (0,\om)}{\pa_{\al} t (0,\b_L)} \left[ B_2 + (\la_4-u)\cdot (K_S\pa_x S + K_V \pa_x V) \right] \;d\om .
	\end{align}
	Together with \eqref{eq:dpsi0} and \eqref{eq:quasi1D-char}, the proof of  Lemma \ref{lem:rarefaction} is completed.
	\qed \end{proof}

\subsection{Time derivatives of solutions at singularity point}
Solving the $2\times 2$ linear system formed by \eqref{eq:shock} in Lemma \ref{lem:shock} and \eqref{eq:rare} in Lemma \ref{lem:rarefaction} may give the values of the total
derivatives of the normal velocity and the pressure and the the limiting values of time derivatives $\,(\pa u/\pa t)_*\,$ and $\,(\pa p/\pa t)_*$.

\subsubsection{General case}
\begin{Theorem}
	\rm
	The limiting value $\,(\pa u/\pa t)_*\,$ and $\,(\pa p/\pa t)_*\,$ are calculated as follows.
	\begin{itemize}
		\item[(i)] (Nonsonic case) 
		The limiting values of time derivatives $\,(\pa u/\pa t)_*\,$ and $\,(\pa p/\pa t)_*\,$ can be calculated as
		\begin{equation*}
		\left\{
		\begin{array}{l}
		\pa_t u = \frac{1}{1-u^2} \left[  \frac{Du }{D t} + \frac{uH}{\rho h c_s^2}  \frac{D p}{D t} - u\cdot \left(uC_2+ \frac{H}{\rho h c_s^2}C_4\right)\right],   \\
		\pa_t p = \frac{1}{1-u^2} \left[  \frac{Dp }{D t} + u\rho h W^{2} \frac{D u}{D t} - u\cdot \left(\rho h W^{2}C_2+ uC_4\right)\right].
		\end{array}
		\right.
		\end{equation*}
		
		\item[(ii)] (Sonic case) If assuming that the $t$-axis is located inside the rarefaction wave associated with the $\,\la_1\,$ characteristic family, then
		\begin{equation*} \label{eq:sonic}
		\begin{cases}
		\pa_t u = \frac{1-u^2}{2}\left[ \frac{\la_4}{\la_4-u}\cdot\Big(d_L(0) - \varphi C_4\Big) + \frac{\la_4-2u}{\la_4 - u}\cdot \frac{1}{1-u^2} C_2\right],   \\
		\pa_t p = \frac{1}{2\varphi}\left[ \frac{\la_4}{\la_4-u}\cdot\Big(d_L(0) - \dfrac{1}{1-u^2} C_2\Big) + \frac{\la_4-2u}{\la_4 - u}\cdot \varphi C_4\right].
		\end{cases}
		\end{equation*}
	\end{itemize}
\end{Theorem}
\begin{proof}
	Here consider the sonic case. As the $t$-axis is located inside the rarefaction wave associated with the $\,\la_1\,$ characteristic family, then we have $\,\la_1= 0$.
	
	Thanks to
	\begin{equation*}
	\frac{D_{-}u }{D t} - C_2 = \Phi \cdot \left( \frac{D_{-}p }{D t} - C_4 \right),
	\end{equation*}
	with $\Phi = \frac{H^{1/2}}{\rho h c_s W} = (1-u^2) \varphi$ and $\frac{D_{-} }{D t} := \frac{\pa }{\pa t} + \la_1 \frac{\pa }{\pa x}$, one has
	\begin{equation} \label{eq:sonic1}
	\dfrac{1}{1-u^2} \pa_t u - \varphi \pa_t p = \dfrac{1}{1-u^2} C_2 - \varphi C_4.
	\end{equation}
	On the other hand, together with  \eqref{eq:rare}, $\frac{D}{D t}  = \frac{u-\la_1}{\la_4-\la_1} \frac{D_{-} }{D t}  + \frac{\la_4 - u}{\la_4-\la_1}\frac{D_{+}}{D t}$ and
	\begin{equation*}
	\frac{D_{+}u }{D t} - C_2 = - \Phi \cdot  \left( \frac{D_{+}p }{D t} - C_4 \right),
	\end{equation*}
	with $\frac{D_{+} }{D t} := \frac{\pa }{\pa t} + \la_4 \frac{\pa }{\pa x}$,
	one gets
	\begin{equation} \label{eq:sonic2}
	\dfrac{1}{1-u^2} \pa_t u + \varphi \pa_t p = \dfrac{\la_4}{\la_4-u}d_L(0) - \dfrac{u}{\la_4-u} \cdot [\dfrac{1}{1-u^2} C_2 + \varphi C_4].
	\end{equation}
	With the help of \eqref{eq:sonic1} and \eqref{eq:sonic2}, the proof can be completed.\qed
\end{proof}

\begin{Theorem}\rm
	The limiting value $\,(\pa \rho/\pa t)_*\,$ and $\,(\pa v/\pa t)_*\,$ are calculated as follows.
	\begin{itemize}
		\item[(i)] If $\,u_* > 0\,$, they are obtained by
		\begin{align*}
		\left( \frac{\pa \rho}{\pa t} \right)_* =& \frac{1}{hc^2_{1*}} \cdot \left[ \left( \frac{\pa p}{\pa t} \right)_* - \frac{1}{T}\left( \frac{\pa p}{\pa S} \right)_{\rho} \left( B_3 - u_* \left( T\frac{\pa S}{\pa x} \right)_* \right)\right],  \\
		\left( \frac{\pa v}{\pa t} \right)_* =& \frac{-1}{hW (1+W^2 v^2)}\Big[\, W v \pa_t h + VW^2 u \pa_t u + u\pa_x V - B_4 \,\Big],
		\end{align*}
		where $\pa_x V$ is given in \eqref{eq:dxV} and
		$
		dh = \frac{\G}{\G-1}\left(\frac{1}{\rho}dp - \frac{p}{\rho^2}d\rho\right)
		$.
		
		\item[(ii)] If $\,u_* < 0\,$, they are given by
		\begin{align*}
		& \left( \frac{\pa \rho}{\pa t} \right)_* = \dfrac{1}{s-u}\left[ s\left( \frac{D \rho}{D t} \right)_* - u\left( \frac{D \rho}{D s} \right)_* \right], \\
		& \left( \frac{\pa v}{\pa t} \right)_* = \frac{-1}{hW (1+W^2 v^2)}\Big[\, W v \pa_t h + VW^2 u \pa_t u + u\pa_x V - B_4 \,\Big],
		\end{align*}
		where $\left( \frac{D \rho}{D s} \right)_*$ is given in \cite[(3.39)]{Yang-Tang:2012}, and
		\begin{equation*}
		\begin{cases}
		\left( \frac{D \rho}{D t} \right)_*  = \dfrac{1}{hc^2}\cdot \left(\left( \frac{D p}{D t} \right)_* - C_4\right)+ C_1 , \\
		\pa_x V = \frac{1}{s-u_*} \cdot \Big[B_4^R + (s-u_R)\pa_x V_R - B_4\Big], \\	
		\pa_x V_R = \Big[W v \pa_x h + VW^2 u \pa_x u +  hW (1+W^2 v^2) \pa_x v\Big] \Big|_R.
		\end{cases}
		\end{equation*}
	\end{itemize}
\end{Theorem}

\subsubsection{Acoustic case}
When $\vec U_L = \vec U_R = \vec U_*$ and $\vec U'_L \neq \vec U'_R$, we meet the acoustic case.
\begin{Theorem}\rm
	If assuming  {$(\la_1)_* <0$ and $(\la_4)_* >0$}, then $(\pa u/\pa t)_*$ and $(\pa p/\pa t)_*$ can be calculated by
	\begin{equation*}
	\begin{cases}
	\left( \frac{\pa u}{\pa t} \right)_* = \frac{1}{\la_4 - \la_1} \Big( \la_4 \cdot [C_2^L + (M+\la_1-u)|_L u'_L - N_u^L p'_L] - \la_1\cdot [C_2^R + (M+\la_4-u)|_R u'_R - N_u^L p'_R]\Big),  \\
	\left( \frac{\pa p}{\pa t} \right)_* = \frac{1}{\la_4 - \la_1} \Big( \la_4 \cdot [C_4^L + (M+\la_1-u)|_L p'_L - N_p^L u'_L] - \la_1\cdot [C_4^R + (M+\la_4-u)|_R p'_R - N_p^L u'_R]\Big),
	\end{cases}
	\end{equation*}
	where
	\begin{equation*}
	M = u-\dfrac{\la_4 + \la_1}{2}, \quad  N_u = \dfrac{\la_4 - \la_1}{2}\Phi, \quad N_p = \dfrac{\la_4 - \la_1}{2}\Phi^{-1}.
	\end{equation*}
	With the aid of the EOS $p=p(\rho,S)$,
	$(\pa \rho/\pa t)_*$ is calculated by
	\begin{equation*}
	\left( \frac{\pa \rho}{\pa t} \right)_* =
	\begin{cases}
	\frac{1}{h c_s^2}\cdot\left[ \left( \frac{\pa p}{\pa t} \right)_* + u \; \Big(p'_L - (c_s^2 h)|_L \rho'_L \Big) - (\G-1)\rho B_3 \right], & u > 0, \\
	\frac{1}{h c_s^2} \cdot \left[ \left( \frac{\pa p}{\pa t} \right)_* + u\; \Big(p'_R - (c_s^2 h)|_R \rho'_R \Big) - (\G-1)\rho B_3 \right], & u < 0,
	\end{cases}
	\end{equation*}
	and $\,(\pa v/\pa t)_*\,$ is gotton by
	\begin{equation*}
	\left( \frac{\pa v}{\pa t} \right)_* =\frac{-1}{hW(v^2W^2+1)} \cdot
	\begin{cases}
	W v \pa_t h + hvW^3 u\pa_t u - B_4 + u \pa_x V|_L,  & u_* > 0, \\
	W v \pa_t h + hvW^3 u\pa_t u - B_4 + u \pa_x V|_R,  & u_* < 0,
	\end{cases}
	\end{equation*}
	where
	$
	\pa_x V = W v \pa_x h + hvW^3 u\pa_x u + hW(v^2W^2+1) \pa_x v$.
	
\end{Theorem}

%
\section*{Acknowledgements}
{The  authors} was partially supported by
the Special Project on High-performance Computing under the National Key R\&D Program (No. 2016YFB0200603),
Science Challenge Project (No. JCKY2016212A502),  and
the National Natural Science Foundation of China (Nos. 91630310 \& 11421101).

\bibliography{ref/journalname,ref/pkuth}

\begin{thebibliography}{200}
\bibitem{Anderson:2006}
M.~Anderson, E.W. Hirschmann, S.L. Liebling, and D.~Neilsen,
\newblock Relativistic {MHD} with adaptive mesh refinement,
\newblock {\em Class. Quantum Grav.}, 23(2006), 6503-6524.

\bibitem{Balsara1994}
D.S. Balsara, Riemann solver for relativistic hydrodynamics,
{\em J. Comput. Phys.}, 114(1994), 284-297.

\bibitem{Ben-Artzi-Li2007}
    M. Ben-Artzi and J. Q. Li,
    Hyperbolic conservation laws: Riemann invariants and the generalized Riemann problem,
    {\em  Numer. Math.}, 106(2007), 369-425.




\bibitem{Chen-Kuang-Tang:2017}
Y.P. Chen, Y.Y. Kuang, and H.Z. Tang,
Second-order accurate genuine BGK schemes for the ultra-relativistic flow simulations,
\newblock {\em J. Comput. Phys.}, 349(2017), 300-327.

\bibitem{ZannaBucciantini:2002}
L.~Del Zanna and N.~Bucciantini,
\newblock An efficient shock-capturing central-type scheme for multidimensional
  relativistic flows {I}: Hydrodynamics,
\newblock {\em Astron. Astrophys.}, 390(2002), 1177-1186.

\bibitem{Zanna:2003}
L.~Del Zanna, N.~Bucciantini, and P.~Londrillo,
\newblock An efficient shock-capturing central-type scheme for multidimensional
  relativistic flows I. Hydrodynamics, 
\newblock {\em Astron. Astrophys.}, 390(2002), 1177-1186.  

\bibitem{Donat:1998}
R. Donat, J.A. Font, J.M. Ib$\rm\acute{a}\tilde{n}$ez, and A. Marquina, A flux-split algorithm applied to relativistic flows,
 {\em J. Comput. Phys.}, 146(1998), 58-81.

\bibitem{Han-Li-Tang:2010}
EE Han, J.Q. Li, and H.Z. Tang, An adaptive GRP scheme for compressible fluid flows,
\emph{J. Comput. Phys.}, 229(2010), 1448-1466.

\bibitem{Han-Li-Tang:2011}
EE Han, J.Q. Li, and H.Z. Tang, Accuracy of the adaptive GRP scheme
and the simulation of 2-D Riemann problems for compressible Euler equations,
\emph{Commun. Comput. Phys.}, 10(2011), 577-606.

\bibitem{HeAdaptiveRHD}
P.~He and H.Z. Tang,
\newblock An adaptive moving mesh method for two-dimensional relativistic
  hydrodynamics,
\newblock {\em Commun. Comput. Phys.}, 11(2012), 114-146.

\bibitem{HeAdaptiveRMHD}
P.~He and H.Z. Tang,
\newblock An adaptive moving mesh method for two-dimensional relativistic
  magnetohydrodynamics,
\newblock {\em Computers \& Fluids}, 60(2012), 1-20.


\bibitem{Honkkila:2007}
V.~Honkkila and P.~Janhunen,
\newblock {HLLC} solver for ideal relativistic {MHD},
\newblock {\em J. Comput. Phys.}, 223(2007), 643-656.

\bibitem{Host:2008}
B.~van~der Holst, R.~Keppens, and Z.~Meliani,
\newblock A multidimensional grid-adaptive relativistic magnetofluid code,
\newblock {\em Comput. Phys. Comm.}, 179(2008), 617-627.

\bibitem{Jiang-Shu1996}
{ G.S. Jiang and C.-W. Shu,
Efficient implementation of Weighted ENO schemes,
{\em J. Comput. Phys.}, 126(1996), 202-228.}


\bibitem{Koldoba:2002}
A.V. Koldoba, O.A. Kuznetsov, and G.V. Ustyugova,
\newblock An approximate {Riemann} solver for relativistic
  magnetohydrodynamics,
\newblock {\em Mon. Not. R. Astron. Soc.}, 333(2002), 932-942.

\bibitem{GodunovRMHD}
S.S. Komissarov,
\newblock A {Godunov-type} scheme for relativistic magnetohydrodynamics,
\newblock {\em Mon. Not. R. Astron. Soc.}, 303(1999), 343-366.


\bibitem{E.F.Toro:2013}
B.J. Lee, E.F. Toro, C.E. Castro, and N. Nikiforakis,
 Adaptive Osher-type scheme for the Euler equations with highly nonlinear equations of state,
  \emph{J. Comput. Phys.}, 246(2013), 165-183.


\bibitem{LI-DU:2016}
J.Q. Li and Z.F. Du, A two-stage fourth order time-accurate discretization
for Lax-Wendroff type flow solvers I. Hyperbolic conservation laws,
\emph{SIAM J. Sci. Comput.}, 38(2016), A3046-A3069.

\bibitem{May-White1966}
M.M. May and R.H.White,
Hydrodynamic calculations of general-relativistic collapse,
{\em Phys. Rev.}, 141(1966), 1232-1241.

\bibitem{May-White1967}
M.M. May and R.H. White,
 Stellar dynamics and gravitational collapse,
in {\em Methods in Computational
Physics, Vol. 7, Astrophysics} (B. Alder, S. Fernbach, and M.
Rotenberg edited), Academic Press, 1967, 219-258.


\bibitem{MignoneHLLCRMHD}
A.~Mignone and G.~Bodo,
\newblock 
  An HLLC Riemann solver for relativistic flows - I. Hydrodynamics,
\newblock {\em Mon. Not. R. Astron. Soc.}, 364(2005), 126-136. 

\bibitem{PXLL:2016}
L. Pan, K. Xu, Q.B. Li, and J.Q. Li,
An efficient and accurate two-stage fourth-order gas-kinetic scheme for
the Euler and Navier-Stokes equations,
  \emph{J. Comput. Phys.}, 326(2016), 197-221.

\bibitem{QamarKinetic2004}
S.~Qamar and G.~Warnecke,
\newblock A high-order kinetic flux-splitting method for the relativistic
  magnetohydrodynamics,
\newblock {\em J. Comput. Phys.}, 205(2005), 182-204.

\bibitem{QinShuYang-JCP2016}
T. Qin, C.-W. Shu and Y. Yang,
 Bound-preserving discontinuous Galerkin methods for relativistic hydrodynamics,
 {\em J. Comput. Phys.}, 315(2016), 323-347.

\bibitem{Toro-book2009}
 E.F. Toro,
 {\em Riemann Solvers and Numerical Methods for Fluid Dynamics}, 3rd edition,
 Springer, Berlin, 2009.


\bibitem{Wilson:1972}
J.R. Wilson,
\newblock Numerical study of fluid flow in a {Kerr} space,
\newblock {\em Astrophys. J.}, 173(1972), 431-438.

\bibitem{Wilson:1979}
J.R. Wilson,
\newblock A numerical method for relativistic hydrodynamics,
\newblock in  {\em Sources of Gravitational Radiation}, L.L. Smarr edited,
  Cambridge University Press, 1979,  423-446.

\bibitem{Wu-PRD2017}
K. Wu, Design of provably physical-constraint-preserving methods
for general relativistic hydrodynamics,
  {\em Phys. Rev. D}, 95(2017), 103001.

\bibitem{WuEGRHD}
K.L. Wu and H.Z. Tang,
\newblock Finite volume local evolution {Galerkin} method for two-dimensional
  relativistic hydrodynamics,
\newblock {\em J. Comput. Phys.}, 256(2014), 277-307.

\bibitem{Wu-Tang:2016}
K.L. Wu and H.Z. Tang,
A direct Eulerian GRP scheme for spherically symmetric general relativistic hydrodynamics,
\emph{SIAM J. Sci. Comput.}, 38(2016), B458-B489.


\bibitem{Wu-Tang-JCP2015}
K.L. Wu and H.Z. Tang,
 High-order accurate physical-constraints-preserving
finite difference WENO schemes for special relativistic hydrodynamics,
{\em J. Comput. Phys.}, 298(2015), 539-564.

\bibitem{Wu-Tang-RMHD2016}
K.L. Wu and H.Z. Tang,
 Admissible states and physical constraints preserving
 numerical schemes for special relativistic magnetohydrodynamics,
{\em Math. Models and Meth. in Appl. Sci.},  27(2017), 1871-1928.

\bibitem{Wu-Tang-RHD2016}
K.L. Wu and H.Z. Tang,
Physical-constraints-preserving central discontinuous Galerkin methods
for special relativistic hydrodynamics with
a general equation of state, {\em Astrophys. J. Suppl. series}, 228(2017), 3.

\bibitem{Wu-Tang-RMHD2017b}
K.L. Wu and H.Z. Tang, On physical-constraints-preserving schemes for
special relativistic magnetohydrodynamics with a general equation of state,
submitted to {\em ZAMP},   {\tt arXiv: 1709.05838}, 2017.



\bibitem{Wu-Yang-Tang:2014}
 K.L. Wu, Z.C. Yang, and H.Z. Tang, A third-order accurate direct
 Eulerian GRP scheme for one-dimensional relativistic hydrodynamics,
 \emph{East Asian J. Appl. Math.}, 4(2014), 95-131.

 \bibitem{Wu-Yang-Tang:2014b}
K.L. Wu, Z.C. Yang, and H.Z. Tang, A third-order accurate direct Eulerian
GRP scheme for the Euler equations in gas dynamics,
\emph{J. Comput. Phys.}, 264(2014), 177-208.


\bibitem{Yang-He-Tang:2011}
Z.C. Yang, P. He, and H.Z. Tang, A direct Eulerian GRP scheme
for relativistic hydrodynamics: One-dimensional case,
\emph{J. Comput. Phys.}, 230(2011), 7964-7987.


\bibitem{Yang-Tang:2012}
Z.C. Yang and H.Z. Tang, A direct Eulerian GRP scheme for
relativistic hydrodynamics: Two-dimensional case,
 \emph{J. Comput. Phys.}, 231(2012), 2116-2139.

\bibitem{Zhao-Tang:2017a}
J. Zhao and H.Z. Tang, Runge-Kutta discontinuous Galerkin methods
for the special relativistic magnetohydrodynamics, {\em J. Comput. Phys.},
343(2017), 33-72.

\bibitem{Zhao-Tang:2017b}
J. Zhao and H.Z. Tang, Runge-Kutta central discontinuous Galerkin methods
 for the special relativistic hydrodynamics, {\em Commun. Comput. Phys.},
  22(2017), 643-682.




\end{thebibliography}
           \bibliographystyle{plain}


\end{document}